\documentclass[12pt]{amsart}
\usepackage{amsthm,amsmath,amssymb,amscd,amsfonts,mathrsfs,color,graphicx,bm}
\theoremstyle{plain}
\newtheorem{theorem}{Theorem}[section]
\newtheorem{lemma}[theorem]{Lemma}
\newtheorem{nodesignation}[theorem]{}
\newtheorem{proposition}[theorem]{Proposition}

\theoremstyle{definition}

\theoremstyle{remark}
\newtheorem{remark}[theorem]{Remark}

\def\cal{\mathcal}
\begin{document}

\centerline{\bf Hyperbolicity of Generic High-Degree Hypersurfaces}
\centerline{\bf in Complex Projective Space}

\bigbreak\centerline{\it Dedicated to the Memory of Hans Grauert}

\bigbreak
\centerline{Yum-Tong Siu\ %
\footnote{Partially supported by grant DMS-1001416 from the National Science Foundation.}
}

\bigbreak

\bigbreak

\tableofcontents

\setcounter{section}{-1}
\section{\sc Introduction} In this paper we
are going to present the proofs of the following two theorems on the hyperbolicity of generic hypersurfaces of sufficiently high degree and of their complements, together with a number of related results, obtained by the same methods, such as: (i) a Big-Picard-Theorem type statement concerning extendibility, across the puncture, of holomorphic maps from a punctured disk to a generic hypersurface of high degree, (ii) entire holomorphic functions satisfying polynomial equations with slowly varying coefficients, and (iii) Second Main Theorems for jet differentials and slowly moving targets.

\begin{theorem}\label{main_theorem}For any integer
$n\geq 3$ there exists a positive integer $\delta_n$ (which is explicitly expressible as
a function of $n$) with the
following property. For any generic hypersurface $X$ in ${\mathbb
P}_n$ of degree $\delta\geq\delta_n$ there is no nonconstant holomorphic map from ${\mathbb C}$ to $X$.
\end{theorem}

\begin{theorem}\label{hyperbolicity_complement_hypersurface}For any integer
$n\geq 2$ there exists a positive integer $\delta_n^*$ (which is explicitly expressible as
a function of $n$) with the
following property. For any generic hypersurface $X$ in ${\mathbb
P}_n$ of degree $\delta\geq\delta_n^*$ there is no nonconstant holomorphic map from ${\mathbb C}$ to ${\mathbb P}_n-X$.
\end{theorem}

\smallbreak Theorem \ref{main_theorem} was presented with a sketch of its proof in \cite{Si02} and \cite{Si04}.  The methods used, though rather tedious in some of their details, consist essentially just of some skillful manipulations in linear algebra and the chain rule of differentiation.  The underlying ideas in these methods can be traced to the techniques which Bloch developed in his 1926 paper \cite{B26}.  To explain this link to Bloch's paper \cite{B26}, we first very briefly describe Bloch's techniques with explanations about how they foreshadow to a certain extent our techniques in this paper.

\begin{nodesignation}\label{bloch_techniques}{\it Bloch's Technique of Construction of Jet Differential.}\end{nodesignation}  In his 1926 paper \cite{B26} Bloch proved the nonexistence of nonconstant holomorphic maps from ${\mathbb C}$ to a submanifold $Y$ of an abelian variety $A$, which does not contain a translate of a positive-dimensional abelian subvariety of $A$, by producing sufficiently independent holomorphic jet differentials $\omega$ on $Y$ vanishing on some ample divisor of $Y$ and using the fact that the pullbacks of such jet differentials by holomorphic maps from ${\mathbb C}$ to $Y$ vanish identically.

\medbreak He produced such holomorphic jet differentials $\omega$ on $Y$, not by applying to $Y$ the theorem of Riemann-Roch (which was not yet readily available at the time of Bloch's paper for the case needed for its application to $Y$), but explicitly by pulling back to $Y$ constant-coefficient polynomials $P$ (with homogeneous weight) of differentials of the coordinates (including higher-order differentials) of the universal cover $\tilde A$ of the abelian variety $A$.   When the constant-coefficient polynomials $P$ of differentials of coordinates of $\tilde A$ are pulled back to $Y$, the condition of $Y$ not containing a translate of a positive-dimensional abelian subvariety of $A$ causes new vanishing of the pullbacks on $Y$.  Moreover, the new vanishing is on some ample divisor of $Y$ when the constant-coefficient polynomials $P$ of differentials of coordinates of $\tilde A$ are appropriately chosen.  The reason why it is possible to choose $P$ so that its pullback $\omega$ to $Y$ vanishes on an ample divisor of $Y$ is that the condition of not containing a translate of a positive-dimensional abelian subvariety of $A$ guarantees that the dimension of the ${\mathbb C}$-vector space of the pullbacks to $Y$ of all possible such polynomials $P$ is so high that at least one ${\mathbb C}$-linear combination $\omega$ of such pullbacks vanishes on some ample divisor of $Y$.

\medbreak Bloch's construction is related to the classical construction of a ${\mathbb C}$-basis of holomorphic $1$-forms for a regular plane curve $C$ defined by an equation $R(x,y)=0$ of degree $\delta\geq 3$ in the inhomogeneous coordinates $x,y$ of ${\mathbb P}_2$, which are constructed by pulling back to $C$ meromorphic $1$-forms
$$
P(x,y)\frac{dx}{R_y(x,y)}=P(x,y)\frac{-dy}{R_x(x,y)}
$$
of ``low pole order'' on ${\mathbb P}_2$, where $R_x(x,y)$ and $R_y(x,y)$ are the first-order partial derivatives of $R(x,y)$ and $P(x,y)$ is a polynomial of degree $\leq\delta-3$.  The adjunction formula for the plane curve $C$ causes new vanishing to cancel the ``low pole order'' of the meromorphic $1$-forms on ${\mathbb P}_2$ to yield holomorphic $1$-forms on $C$  when the meromorphic $1$-forms on ${\mathbb P}_2$ are pulled back to the plane curve $C$.

\medbreak In this paper, the construction of holomorphic jet differentials on a generic hypersurface $X$ of sufficiently high degree $\delta$ in ${\mathbb P}_n$ combines Bloch's method and the classical construction of holomorphic $1$-forms
on plane curves of high degree.  We take meromorphic jet differentials of low pole orders (of magnitude $\delta^{1-\varepsilon}$ for some appropriate $0<\varepsilon<1$)
on ${\mathbb P}_n$ and pull them back to $X$.  The high degree $\delta$ of $X$ will guarantee (according to Lemma \ref{injectivity_pullback_map_jet_differential} concerning the injectivity of the pullback map for certain jet differentials) that the dimension of the ${\mathbb C}$-vector space of such pullbacks is so high that some ${\mathbb C}$-linear combination of such pullbacks will be a non identically zero holomorphic jet differential on $X$ vanishing on some ample divisor of $X$ (see Proposition \ref{construction_jet_differential_as_polynomial} below).  One key point in this argument is that, because the dimension of ${\mathbb P}_n$ is higher than that of $X$, there are more degrees of freedom in constructing meromorphic $(n-1)$-jet differentials of low pole order on ${\mathbb P}_n$ and, if the pullback map to $X$ of such meromorphic $(n-1)$-jet differentials is injective, there are sufficient independent pullbacks to $X$ to form a non identically zero ${\mathbb C}$-linear combination which vanishes on an ample divisor of $X$.

\begin{nodesignation}\label{key_technique_slanted_vector_field}{\it Key Technique of Slanted Vector Fields.}\end{nodesignation} After so many decades of impasse, the real key which opens the way to the proof of the hyperbolicity of generic hypersurface of sufficiently high degree (in the sense stated in Theorem \ref{main_theorem}) is the introduction in \cite{Si02}\cite{Si04} of the technique of slanted vector fields in the subspace $J^{\rm{\scriptstyle (vert)}}_{n-1}\left({\mathcal X}\right)$ of vertical $(n-1)$-jets in the $(n-1)$-jet space $J_{n-1}({\mathcal X})$ of the universal hypersurface ${\mathcal X}$ of degree $\delta$ in ${\mathbb P}_n\times{\mathbb P}_N$ (where $N={\delta+n\choose n}-1$).

\medbreak For a complex manifold $Y$ the space $J_k(Y)$ of $k$-jets of $Y$ consists of all $k$-jets of $Y$ (each of which is represented by a parametrized complex curve germ).
The universal hypersurface ${\mathcal X}$ of degree $\delta$ in ${\mathbb P}_n\times{\mathbb P}_N$ (with $N={\delta+n\choose n}-1$) is defined by
$$
\sum_{\nu_0+\cdots+\nu_n=\delta}\alpha_{\nu_0,\cdots,\nu_n}z_0^{\nu_0}\cdots z_n^{\nu_n}=0,\leqno{(\ref{key_technique_slanted_vector_field}.1)}
$$
where $\alpha=\left[\alpha_{\nu_0,\cdots,\nu_n}\right]_{\nu_0+\cdots+\nu_n=\delta}$ is the homogeneous coordinate of ${\mathbb P}_N$ and $\left[z_0,\cdots,z_n\right]$ is the homogeneous coordinate of ${\mathbb P}_n$.  For $\alpha\in{\mathbb P}_N$ let $X^{(\alpha)}$ be the hypersurface of degree $\delta$ defined by $(\ref{key_technique_slanted_vector_field}.1)$ when $\alpha$ is fixed as constant.
A vertical $k$-jet of ${\mathcal X}$ is a $k$-jet in ${\mathcal X}$ representable by some (parametrized) complex curve germ lying completely in some fiber $X^{(\alpha)}$ of ${\mathcal X}$.  We denote by $J^{\rm{\scriptstyle (vert)}}_k\left({\mathcal X}\right)$ the space of all vertical $k$-jets  on ${\mathcal X}$.  There is a projection map $\pi_{k,{\rm vert}}: J^{\rm{\scriptstyle (vert)}}_k\left({\mathcal X}\right)\to{\mathbb P}_N$ such that an element $P_0$ of  $J^{\rm{\scriptstyle (vert)}}_k\left({\mathcal X}\right)$ is represented by a (parametrized) complex curve germ in $X^{(\alpha)}$ with $\alpha=\pi_{k,{\rm vert}}\left(P_0\right)$.

\medbreak A slanted vector field $\xi$ on $J^{\rm{\scriptstyle (vert)}}_k\left({\mathcal X}\right)$ means a vector field of $J^{\rm{\scriptstyle (vert)}}_k\left({\mathcal X}\right)$ which at a generic point $P_0$ of $J^{\rm{\scriptstyle (vert)}}_k\left({\mathcal X}\right)$ is not tangential to the space $J_k\left(X^{(\alpha)}\right)$ of $k$-jets of the fiber $X^{(\alpha)}$ at the point $P_0$ of $J_k\left(X^{(\alpha)}\right)$ with $\alpha=\pi_{k,{\rm vert}}\left(P_0\right)$.  When a local $k$-jet differential form on $X^{(\alpha)}$ defined for $\alpha$ in some open subset $U$ of ${\mathbb P}_N$ is regarded as a local function on $J^{\rm{\scriptstyle (vert)}}_k\left({\mathcal X}\right)$ and is differentiated with respect to $\xi$, the result is a local function on  $J^{\rm{\scriptstyle (vert)}}_k\left({\mathcal X}\right)$ which is represented by a local $k$-jet differential on $X^{(\alpha)}$ for $\alpha\in U$. In the case of $k=n-1$, meromorphic slanted vector fields $\xi$ of low vertical pole-order (of the magnitude ${\mathcal O}_{{\mathbb P}_n}\left(n^2\right)$ on the vertical fiber), whose existence is given in Proposition \ref{sufficient_slanted_vector_fields} below,
play the following indispensable role in generating sufficiently independent holomorphic $(n-1)$-jet differentials on $X^{(\alpha)}$ vanishing on an ample divisor for a generic $\alpha$ and for $\delta$ sufficiently large.

\medbreak On a regular hypersurface $X^{(\alpha)}$ of high degree $\delta$, there cannot be any nonzero meromorphic vector fields on $X^{(\alpha)}$ of low pole order.  However, the universal hypersurface ${\mathcal X}$ of degree $\delta$ in ${\mathbb P}_n\times{\mathbb P}_N$ has bidegree $(\delta,1)$ with respect to the two hyperplane section line bundles
${\mathcal O}_{{\mathbb P}_n}(1)$ and ${\mathcal O}_{{\mathbb P}_N}(1)$.  Because of the second component $1$ in the bidegree $(\delta,1)$ of ${\mathcal X}$, when slanted vector fields are used, it is possible to get meromorphic slanted vector fields of $J^{\rm{\scriptstyle (vert)}}_k\left({\mathcal X}\right)$ with low vertical pole order.

\medbreak For a generic $\hat\alpha\in{\mathbb P}_n$ the holomorphic $(n-1)$-jet differential $\omega^{(\hat\alpha)}$ on $X^{(\hat\alpha)}$ vanishing on an appropriate ample divisor (constructed by pulling back an appropriate meromorphic $(n-1)$-jet differential of low pole order on ${\mathbb P}_n$ according to Proposition \ref{construction_jet_differential_as_polynomial} below) can be extended to a holomorphic family $\omega^{(\alpha)}$ on $X^{(\alpha)}$ for $\alpha$ in some open neighborhood $U$ of $\hat\alpha$ in ${\mathbb P}_N$ so that successive application of different finite sets of meromorphic slanted vector fields $\xi_1,\cdots,\xi_\ell$ of low vertical pole order (as constructed in in Proposition \ref{sufficient_slanted_vector_fields} below) would yield sufficiently independent holomorphic jet differentials vanishing on ample divisor on $X^{(\hat\alpha)}$ so that the application of the Schwarz lemma of the vanishing of pullbacks, to ${\mathbb C}$ by a holomorphic map ${\mathbb C}\to X^{(\hat\alpha)}$, of holomorphic jet differentials vanishing on ample divisor of $X^{(\hat\alpha)}$ would force every holomorphic map from ${\mathbb C}$ to $X^{(\hat\alpha)}$ to be constant (see Proposition \ref{basepoint_freeness_jet_differential}, and the proof of Theorem \ref{main_theorem} given in \ref{proof_main_theorem}
below).

\medbreak We would like to remark that the Schwarz lemma of the vanishing of pullbacks, to ${\mathbb C}$ by a holomorphic map ${\mathbb C}\to X^{(\hat\alpha)}$, of holomorphic jet differentials vanishing on ample divisor of $X^{(\hat\alpha)}$ also has its origin in Bloch's 1926 paper \cite{B26}, though its formulation and its proof there are in a form very different from our current way of mathematical presentation.  Bloch's proof of applying Nevanlinna's logarithmic derivative lemma (p.51 of \cite{Ne25}) to local coordinates which are the logarithms of global meromorphic functions still is the best proof of the Schwarz lemma.  It is recast in the current language of mathematical presentation on pp.1162-1164 of \cite{SY97}.

\begin{nodesignation}\label{slanted_vector_field_bloch_translation_technique}{\it Slanted Vector Fields and Bloch's Technique of Maps by Translation.}\end{nodesignation}  The technique of slanted vector fields in a way also finds some remote ancestry in the 1926 paper of Bloch \cite{B26}, though the connection is not so transparent.  We point this out here in order to dispel the wrong perception that the technique of slanted vector fields is applicable to generic hypersurfaces of high degree because of the variation of complex structure of hypersurfaces as the coefficients of their defining functions vary.

\medbreak In his proof of the hyperbolicity of a submanifold $Y$ in an abelian variety $A$ which contains no translate of positive-dimensional abelian subvariety of $A$, Bloch used the vector fields from translations in $A$ to do the differentiation of jet differentials.  That is the reason why when such differentiations cannot yield enough independent holomorphic jet differentials to give hyperbolicity of $Y$, $Y$ must contain a translate of some positive-dimensional abelian subvariety of $A$.  This link of the use of slanted vector fields to Bloch's technique of using vector fields of maps by translation is obscured by the fact that in Bloch's technique the result of differentiation with respect to vector fields of maps by translation is just the same as the use of a different constant-coefficient polynomial in differentials of the coordinates of the universal cover $\tilde A$ of $A$.

\medbreak Let us return to our situation at hand of using slanted vector fields $\xi$ on  $J_k\left(X^{(\alpha)}\right)$ of low vertical pole order for $k=n-1$. Though the complex structure of the hypersurface $X^{(\alpha)}$ changes as $\alpha$ varies in ${\mathbb P}_N$, the slanted vector fields $\xi$ in general do not respect the fibers in the sense that for two distinct points $P_0$ and $P_0^\prime$ on $J_k\left(X^{(\alpha)}\right)$ the two projections $\pi_{k,{\rm vert}}\left(\xi_{P_0}\right)$ and $\pi_{k,{\rm vert}}\left(\xi_{P_0^\prime}\right)$ are in general different vectors in the tangent space of ${\mathbb P}_N$ at $\alpha$.

\medbreak For our situation at hand, the geometric picture is not the pulling back of $k$-jet differential from a neighboring fiber by the slanted vector field $\xi$, even in the infinitesimal setting.  What is relevant is the existence of slanted vector fields pointing in sufficiently many different directions on $J^{\rm{\scriptstyle (vert)}}_k\left({\mathcal X}\right)$ at the prescribed point in question.  The realization of the irrelevancy of the variation of the complex structure $X^{(\alpha)}$ as $\alpha$ varies in ${\mathbb P}_N$, as well as the interpretation of Bloch's technique of differentiation with respect to vector fields of maps by translation, points to the promise of the applicability of our method even to the case of some rigid complex manifolds $Z$ inside some ${\mathbb P}_m$ as a submanifold of possibly high codimension.  In certain cases, though $Z$ may be rigid as a compact complex manifold, yet there is a possibility that appropriate meromorphic vector fields on ${\mathbb P}_n$ applied to pullbacks to $Z$ of low pole-order meromorphic jet differentials on ${\mathbb P}_n$ may yield sufficiently independent holomorphic jet differentials on $Z$ vanishing on an ample divisor.

\begin{nodesignation}\label{necessity_vertical_jet_space}
{\it Necessity of Use of Vertical Jet Space.}\end{nodesignation}
The reason why the more complicated space $J^{\rm{\scriptstyle (vert)}}_{n-1}\left({\mathcal X}\right)$ of vertical $(n-1)$-jets of ${\mathcal X}$ has to be used instead of the simpler $(n-1)$-jet space $J_{n-1}({\mathcal X})$ of ${\mathcal X}$ is that, while it is possible to extend a holomorphic $(n-1)$-jet differential $\omega^{(\hat\alpha)}$ on a hypersurface $X^{(\hat\alpha)}$ for a generic $\hat\alpha\in{\mathbb P}_N$ to a holomorphic family of
$\omega^{(\alpha)}$ on $X^{(\alpha)}$ for $\alpha$ in an open neighborhood $U$ of $\hat\alpha$ in ${\mathbb P}_N$, it is in general impossible to find a holomorphic $(n-1)$-jet differential on the part of $J_{n-1}({\mathcal X})$ above some open neighborhood $U$ of $\hat\alpha$ in ${\mathbb P}_N$ whose pullback to $X^{(\hat\alpha)}$ is equal to $\omega^{(\hat\alpha)}$.

\medbreak Difficulty of the latter kind of extension can be illustrated easily in the case of a holomorphic family of plane curves $C_a$  given by $R(x,y,a)=0$ with $a$ in the open unit disk $\Delta$ of ${\mathbb C}$ as a holomorphic parameter.  A holomorphic $1$-form on a single plane curve $C_a$ can be constructed as
$$
\frac{dx}{R_y(x,y,a)}=-\frac{dy}{R_x(x,y,a)}
$$
from the vanishing of the differential $dR=R_xdx+R_ydy$ on $C_a$ when $a$ is considered as a constant, but in the total space $\bigcup_{a\in\Delta}C_a$ of the family of plane curves it is not easy to carry out a similar construction, because when $a$ is regarded as a variable, the differential $dR$ becomes $R_xdx+R_ydy+R_ada$ and the same method cannot be applied.

\medbreak Furthermore, it is for this kind of difficulty of constructing $(n-1)$-jet differentials on the universal hypersurface ${\mathcal X}$ that the additional condition $(\ref{II}.1)_j$
for $1\leq j\leq n-1$ is introduced into Theorem \ref{solution_slow_varying_coefficient} on entire function solutions of polynomial equations with slowing varying coefficients, so that  families of vertical $(n-1)$-jet differentials on the fibers can be used instead.

\begin{nodesignation}\label{algebraic_geometric_counterpart_of slanted_vector_fields}
{\it Algebraic Geometric Counterpart of Slanted Vector Fields.}\end{nodesignation}  In his 1986 paper \cite{Cl86} Clemens introduced a technique  (later generalized by Ein \cite{Ei88}, and Voisin \cite{Voi96}) to prove the nonexistence of rational and elliptic curves in generic hypersurfaces of high degree by showing that the normal bundle of one such curve in the family of such curves is globally generated by sections with vertical pole order $1$.  His technique can be considered the counterpart of our method of slanted vector fields and as a matter of fact serves as motivation for our method.

\medbreak On its face value Clemens's technique of using normal bundle to estimate the genus of a curve is algebraic in nature and cannot possibly have anything to do with the problem of hyperbolicity of transcendental in nature.  Its relevancy was realized for the first time in \cite{Si02} and \cite{Si04} partly because of our seemingly completely unrelated earlier work on the deformational invariance of the plurigenera \cite{Si98} \cite{Si00}.

\medbreak Like jet differentials, pluricanonical sections can be naturally pulled back by a map and, as a result, their Lie differentiation can be naturally defined without specifying any special connection.  In a hitherto unsuccessful attempt to study deformational invariance of sections of other bundles associated with the tangent bundle (besides pluricanonical sections), we investigated the obstruction of moving jet differentials out of a fiber in a family of compact complex manifolds and considered their Lie derivatives with respect to slanted vector fields.  Such investigations, though unsuccessful so far as its original goal is concerned, serendipitously led to the use of slanted vector fields in the study of hyperbolicity problems and to the realization that Clemens's technique is relevant to, and can serve as motivation for, the differentiation of jet differentials by slanted vector fields to produce new ones.

\begin{nodesignation}\label{linear_algebra_versus_riemann_roch}{\it Linear Algebra versus Theorem of Riemann-Roch.}\end{nodesignation}  As already pointed out in the paragraph straddling p.445 and p.446 in \cite{Si02}, a non identically zero holomorphic $(n-1)$-jet differential on $X^{(\alpha)}$ vanishing on an ample divisor can be constructed from the theorem of Riemann-Roch by using the sufficient positivity of the canonical line bundle of $X^{(\alpha)}$ and the lower bound of the negativity of jet differential bundles of $X^{(\alpha)}$.  Such a jet differential can also be directly obtained by using the linear algebra method of solving a system of linear equations with more unknowns than independent linear equations, which is the method used here in Proposition \ref{construction_jet_differential_as_polynomial} below, as sketched on p.446 of \cite{Si02}.  This direct method of construction by linear algebra has the important advantage of better control over the form of the resulting jet differential so that the application of slanted vector fields can produce sufficiently independent jet differentials vanishing on an ample divisor of $X^{(\alpha)}$ for a generic point $\alpha$ of ${\mathbb P}_N$ (see Proposition \ref{basepoint_freeness_jet_differential}, and the proof of Theorem \ref{main_theorem} given in \ref{proof_main_theorem}
below).

\medbreak Of course, the use of the theorem of Riemann-Roch also uses the linear algebra technique of counting the dimension of sections modules and the dimension of obstructional higher cohomology groups, but the process of going through a labyrinth of exact sequences so obscures the eventual form of the resulting jet differential that not enough control can be retained to get beyond the weaker conclusion that holomorphic maps from ${\mathbb C}$ to $X^{(\alpha)}$ is contained in some proper subvariety of $X^{(\alpha)}$.

\medbreak Recently Diverio, Merker, and Rousseau in \cite{DMR10} used the theorem of Riemann-Roch to construct a holomorphic jet differential on $X^{(\alpha)}$ vanishing on an ample divisor and then used Merker's work \cite{Me09} involving our method of slanted vector fields to arrive at the conclusion that holomorphic maps from ${\mathbb C}$ to $X^{(\alpha)}$ is contained in some proper subvariety of $X^{(\alpha)}$.

\begin{nodesignation}\label{simplification_last_step}{\it Simplified Treatment in Going from Non Zariski Density of Entire Curves to Hyperbolicity.}\end{nodesignation}  In the original sketch of the proof of Theorem \ref{main_theorem} in \cite{Si02}, for the last step discussed on p.447 of \cite{Si02} of going from the non Zariski density of entire curves to hyperbolicity, the method of construction of jet differentials is applied to a hypersurface $\hat X$ in ${\mathbb P}_{\hat n}$ constructed from the hypersurface $X^{(\alpha)}$ in ${\mathbb P}_n$ with a larger $\hat n$ so that more jet differentials on $X^{(\alpha)}$ can be obtained from jet differentials on $\hat X$.  The idea is that the zero-set of the jet differentials constructed on $X^{(\alpha)}$ from linear algebra would be defined by the vanishing of polynomials of low degree in low-order partial derivatives of the polynomial $f^{(\alpha)}$ defining $X^{(\alpha)}$ and, in order to take care of such zero-set, the low-order partial derivatives of $f$ are introduced as additional new variables.  The genericity condition of $f$ enters in a certain form of independence of the low-order partial derivatives of $f$.

\medbreak In this paper we use a simplified treatment of this step which just uses the fact that our construction of holomorphic jet differentials depends on the choice of an affine coordinate system of the affine part ${\mathbb C}^n$ of ${\mathbb P}_n$ so that different choices of the affine coordinate system in the construction would give us sufficient independent holomorphic jet differentials to conclude hyperbolicity of a generic hypersurface of high degree.  An earlier version of this paper uses the alternative argument that a meromorphic $(n-1)$-jet differential on ${\mathbb P}_n$ defined by a low-degree polynomials of the inhomogeneous coordinates of ${\mathbb P}_n$ and their differentials have only low vanishing order at every point of $X^{(\alpha)}$.  The current simplified treatment is used here, because the complete rigorous details of the alternative argument on low vanishing order in the earlier version of this paper turn out to be quite tedious.  Moreover, the current argument is related to the technique of slanted vector fields so that the two arguments of generating sufficient holomorphic vector fields are two just different aspects of the same idea.  The relation with the technique of slanted vector fields is that the technique of slanted vector fields is actually the infinitesimal or differential version of the current argument which can be regarded as using affine coordinate transformations to pull back holomorphic jet differentials on neighboring fibers to reduce the common zero-set of holomorphic jet differentials.

\begin{nodesignation}\label{Techniques Parallel to Those in Gelfond-Schneider-Bombieri Theory}{\it Techniques Parallel to Those in Gelfond-Schneider-Lang-Bombieri Theory.}\end{nodesignation}  Paul Vojta presented in \cite{Voj87} a formal parallelism between the results in diophantine approximation and those in value distribution theory.  Along this line, the techniques presented here for the proof of the hyperbolicity of generic hypersurfaces of sufficiently high degree are, in certain ways, quite parallel to the techniques used for the theory of Gelfond-Schneider-Lang-Bombieri (\cite{Sc34}, \cite{Ge34}, \cite{La62}, \cite{La65}, \cite{La66}, \cite{Bo70}, and \cite{BL70}).

\medbreak\noindent(i) The construction of holomorphic jet differentials in Proposition \ref{construction_jet_differential_as_polynomial} by solving a system of linear equations with more unknowns than equations is parallel to the use of Siegel's lemma in the theory of Gelfond-Schneider-Lang-Bombieri to construct a polynomial with estimates on its degree and the heights of its coefficients.

\medbreak\noindent(ii) The requirement that the constructed jet differential vanishing on an ample divisor of high degree in Proposition \ref{construction_jet_differential_as_polynomial} is parallel to the requirement of the vanishing of the constructed polynomial in the theory of Gelfond-Schneider-Lang-Bombieri to high order at certain points.

\medbreak\noindent(iii) Lemma \ref{injectivity_pullback_map_jet_differential} concerning the injectivity of the pullback map for certain jet differentials is parallel to the constructed polynomial in the theory of Gelfond-Schneider-Lang-Bombieri being not identically zero due to the assumption of the degree of transcendence of the given functions.

\medbreak\noindent(iv) The use of Nevanlinna's logarithmic derivative lemma and the use of logarithms of global meromorphic functions as local coordinates in the Schwarz lemma to estimate the contribution from the differentials to be of lower order is parallel to the use of the differential equations in the theory of Gelfond-Schneider-Lang-Bombieri.

\medbreak Such a parallelism between the techniques used in this paper and those in
theory of Gelfond-Schneider-Lang-Bombieri lends support to the preferability of the approach used in this paper for the hyperbolicity problem of hypersurfaces.

\begin{nodesignation}\label{notations_and_terminology}{\it Notations and Terminology.}\end{nodesignation}  For $r>0$ we use $\Delta_r$ to denote the open unit disk in ${\mathbb C}$ of radius $r$ centered at the origin.  When $r=1$, we simply use $\Delta$ to denote $\Delta_1$ when there is no confusion.

\medbreak For a real number $\lambda$ denote by $\lfloor\lambda\rfloor$ the round-down of $\lambda$ which means the largest integer $\leq\lambda$ and denote by $\lceil\lambda\rceil$ be the round-up of $\lambda$ which means the smallest integer $\geq\lambda$.

\medbreak We use $[z_0,\cdots,z_n]$ to
denote the homogeneous coordinates of ${\mathbb P}_n$ and we use $(x_1,\cdots,x_n)$ to denote
the inhomogeneous coordinates of ${\mathbb P}_n$ with $x_j=\frac{z_j}{z_0}$ for $1\leq j\leq n$.  Sometimes we also go to the inhomogeneous coordinates by fixing $z_0\equiv 1$ in the homogeneous coordinates when notationally it is more advantageous to do so.

\medbreak Denote by ${\mathbb N}$ the set of all positive integers.  For an $(n+1)$-tuple
$\nu\in\left({\mathbb N}\cup\{0\}\right)^{n+1}$ of nonnegative integers, we write
$\nu=(\nu_0,\nu_1,\cdots,\nu_n)$ and
$|\nu|=\nu_0+\nu_1+\cdots+\nu_n$ and let
$$
z^\nu=z_0^{\nu_0}z_1^{\nu_1}\cdots z_n^{\nu_n}.
$$
For $0\leq p\leq n$, let $e_p$ denote the unit vector in ${\mathbb C}^{n+1}$ such that all components are zero except that the
component in the $p$-th place is $1$. We use
${\bf\delta}_{\nu,\mu}$ to denote the Kronecker delta for the
indices $\nu,\mu\in\left({\mathbb N}\cup\{0\}\right)^{n+1}$, which assumes the value $1$ for
$\nu=\mu$ and assumes the value $0$ when $\nu\not=\mu$.

\medbreak If from the context there is no risk of confusion, we use $N$ to denote ${\delta+n\choose n}-1$ so that ${\mathbb P}_N$ is the moduli space for all hypersurface of degree $\delta$, without further explicit mention.  The homogeneous coordinates of ${\mathbb P}_N$ will be denoted by
$\alpha=\left[\alpha_{\nu_0,\cdots,\nu_n}\right]_{\nu_0+\cdots+\nu_n=\delta}$.  The hypersurface defined by
$$
f^{(\alpha)}=\sum_{\nu_0+\nu_1+\cdots+\nu_n=\delta}\alpha_{\nu_0,\cdots,\nu_n}z_0^{\nu_0}z_1^{\nu_1}\cdots z_n^{\nu_n}
$$
is denoted by $X^{(\alpha)}$.  For notational simplicity, sometimes the superscript $(\alpha)$ in $f^{(\alpha)}$ and $X^{(\alpha)}$ is dropped when there is no risk of confusion.  Also sometimes we simply use $f^{(\alpha)}\left(x_1,\cdots,x_n\right)$ or $f\left(x_1,\cdots,x_n\right)$ to mean $\frac{1}{z_0^\delta}f^{(\alpha)}\left(z_0,\cdots,z_n\right)$ when the context makes it clear what is being meant.  This notational simplification by dropping superscript $(\alpha)$ applies also to other symbols such as replacing $Q^{(\alpha)}$ by $Q$ (respectively $\omega^{(\alpha)}$ by $\omega$) when there is no risk of confusion or replacing $Q$ by $Q^{(\alpha)}$ (respectiverly $\omega$ by $\omega^{(\alpha)}$) when there is a need to keep track of the dependence on the parameter $\alpha\in{\mathbb P}_N$.

\medbreak When we present the main ideas of an argument and refer to high vanishing order without explicitly giving a precise number, we mean a quantity of the order of $\delta$.  In such a situation, when we refer to low pole order without explicitly giving a precise number, we mean a quantity of the order of $\delta^{1-\varepsilon}$ for some appropriate $0<\varepsilon<1$.

\medbreak The notation at the end of the inequality
$$A(r)\leq B(r)\ \ \|$$
means that there exist $r_0>0$ and a subset $E$ of ${\mathbb R}\cap\{r>r_0\}$ with finite Lebesgue measure such that the inequality holds for $r>r_0$ and not in $E$.  This is the condition needed for the logarithmic derivative lemma of Nevanlinna as given at the bottom of p.51 of \cite{Ne25}.

\medbreak For a meromorphic function $F$ on ${\mathbb C}$ and $c\in{\mathbb C}\cup\{\infty\}$ with $F(0)\not=c$ the {\it counting function} is
$$
N\left(r,F,c\right)=\int_{\rho=0}^r n(\rho,F,c)\frac{d\rho}{\rho},
$$
where $n(\rho,F,c)$ is the number of roots of $F=c$ in $\Delta_\rho$ with multiplicities counted.  The {\it characteristic function} is
$$
T(r,F)=N(r,F,\infty)+\frac{1}{4\pi}\int_{\theta=0}^{2\pi}\log^+\left|F\left(re^{i\theta}\right)\right|d\theta
$$
under the assumption that $0$ is not a pole of $F$, where $\log^+$ means the maximum of $\log$ and $0$.

\medbreak For a complex manifold $Y$ with a $(1,1)$-form $\eta$ and for a holomorphic map $\varphi:{\mathbb C}\to Y$, the characteristic function of $\varphi$ with respect to $\eta$ is
$$
T(r,\varphi,\eta)=\int_{\rho=0}^r\left(\int_{\Delta_\rho}\varphi^*\eta\right)\frac{d\rho}{\rho}.
$$

\begin{nodesignation}\label{twice_integration}{\it Twice Integration of Laplacian in Nevanlinna Theory.}
\end{nodesignation}The technique of twice integrating the Laplacian of a function introduced by Nevanlinnna for his theory of value distribution will be used a number of times in this paper.  We put it down here for reference later.  For any smooth function $g(\zeta)$, from the divergence theorem
$$
\int _{|\zeta|<r} \Delta g = \int ^{2\pi }_{\theta =0} \left({\partial \over \partial r}
g\left(re^{i\theta }\right)\right) r d\theta
$$
and $\Delta  = 4\partial _\zeta\partial
_{\overline\zeta}$
it follows that
$$
\displaylines{4\int_{\rho=r_1}^r\left(\int _{|\zeta|<\rho} \partial_\zeta\partial
_{\overline\zeta}g\right)\frac{d\rho}{\rho}
=
\int_{\rho=r_1}^r\left(\int _{|z|<\rho} \Delta g\right)\frac{d\rho}{\rho}\cr=\int ^{2\pi }_{\theta =0}g(re^{i\theta }) d\theta-\int ^{2\pi }_{\theta =0}g(r_1e^{i\theta }) d\theta.\cr}
$$

\begin{nodesignation}\label{pullback_function}{\it Function Associated to Pullback of Jet Differential to Part of Complex Line.}
\end{nodesignation}
Let $\omega$ be a holomorphic $k$-jet differential on a complex manifold $Y$ of complex dimension $n$ and $\varphi$ be a holomorphic map from an open subset $U$ of ${\mathbb C}$ (with coordinate $\zeta$) to $Y$.  The map $\varphi$ induces a map $J_{k,\varphi}$ from the space $J_k(U)$ of $k$-jets on $U$ to the space $J_k(Y)$ of $k$-jets on $Y$, which sends a $k$-jet $\eta$ on $U$ at $\zeta_0$ represented by a parametrized complex curve germ $\gamma:\Delta\to U$ with $\gamma(0)=\zeta_0$ to the $k$-jet represented by the parametrized complex curve germ $\varphi\circ\gamma:\Delta\to Y$ at $\varphi(\zeta_0)$.  The pullback $\varphi^*\omega$ of $\omega$ by $\varphi$ means the holomorphic $k$-jet differential on $U$ whose value at a $k$-jet $\eta$ of $U$ at $\zeta_0$ is the value of $\omega$ at the $k$-jet $J_{k,\varphi}\left(\eta\right)$ of $J_k(Y)$ at $\varphi(\zeta_0)$.

\medbreak The crucial tool in the study of the hyperbolicity problem is the result, usually referred to as the Schwarz lemma, of the vanishing of the pullback of a holomorphic jet differential on a compact complex manifold vanishing on an ample divisor by a holomorphic map from ${\mathbb C}$.  Its proof, by Bloch's technique of using the logarithmic derivative lemma of Nevanlinna (p.51 of \cite{Ne25}) with logarithms of global meromorphic functions as local coordinates, first shows the vanishing of a function associated to the pullback of the jet differential and then obtains the vanishing of the pullback of the jet differential by composing the map from ${\mathbb C}$ with appropriate holomorphic maps ${\mathbb C}\to{\mathbb C}$.

\medbreak In the later part of this article, when the analogues of the Big Picard Theorem are introduced for generic hypersurfaces $X$ of high degree to extend holomorphic maps from ${\mathbb C}-\overline{\Delta_{r_0}}\to X$ to holomorphic maps from ${\mathbb C}\cup\{\infty\}-\overline{\Delta_{r_0}}\to X$, appropriate holomorphic maps ${\mathbb C}-\overline{\Delta_{r_0}}$ to itself are unavailable for proof the full Schwarz lemma.  Instead only the vanishing of the function associated to the pullback of the jet differential can be obtained.  We now introduce a notation for this function.  The function on $U$, denoted by ${\rm eval}_{id_{\mathbb C}}(\varphi^*\omega)$, at the point $\zeta_0$ is the value of the $k$-jet $\varphi^*\omega$ evaluated at the $k$-jet of $U$ at $\zeta_0$ represented by the parametrized curve defined by the identity map of ${\mathbb C}$.  In other words, the value of ${\rm eval}_{id_{\mathbb C}}(\varphi^*\omega)$ at $\zeta_0\in U$ is the value of $\omega$ at the $k$-jet on $Y$ represented by the parametized complex curve germ $\varphi:U\to Y$ at $\varphi(\zeta_0)$.

\medbreak When for some local coordinates $y_1,\cdots,y_n$ of $Y$ the $k$-jet differential $\omega$ is written as
$$
\sum_{\boldsymbol\nu}
G_{\boldsymbol\nu}\left(y_1,\cdots,y_n\right)
\prod_{1\leq j\leq n,\,1\leq\ell\leq k}\left(d^\ell y_j\right)^{\nu_{\ell,j}}
$$
where ${\boldsymbol\nu}=\left(\nu_{\ell,j}\right)_{1\leq j\leq n,\,1\leq\ell\leq k}$.  The function ${\rm eval}_{id_{\mathbb C}}(\varphi^*\omega)$ at $\zeta\in U$ is given by
$$
\sum_{\boldsymbol\nu}
G_{\boldsymbol\nu}\left(\varphi_1(\zeta),\cdots,\varphi_n(\zeta)\right)
\prod_{1\leq j\leq n,\,1\leq\ell\leq k}\left(\frac{d^\ell}{d\zeta^\ell}\varphi_j(\zeta)\right)^{\nu_{\ell,j}},
$$
where $\varphi$ is represented by $\left(\varphi_1,\cdots,\varphi_n\right)$ with respect to the local coordinates $y_1,\cdots,y_n$ of $Y$, so that if $y_j$ is locally equal to $\log F_j$ for some global meromorphic function $F_j$ on $Y$ (for $1\leq j\leq n$), the logarithmic derivative lemma of Nevanlinna (p.51 of \cite{Ne25}) can be applied to
$\frac{d^\ell}{d\zeta^\ell}\varphi_j(\zeta)=\frac{d^\ell}{d\zeta^\ell}\log F_j(\varphi(\zeta))$.

\section{\sc Approach of Vector
Fields and Lie Derivatives}\label{vector_field_and_lie_derivatives}

\begin{nodesignation}{\it Moduli Space of Hypersurfaces.}\end{nodesignation}  The moduli space of all hypersurfaces of degree
$\delta$ in ${\mathbb P}_n$ is the same as the complex projective
space ${\mathbb P}_N$ of complex dimension $N={\delta+n\choose n}-1$. The defining
equation for the universal hypersurface ${\mathcal X}$ in
${\mathbb P}_n\times{\mathbb P}_N$ is
$$
f=\sum_{\nu\in\left({\mathbb N}\cup\{0\}\right)^{n+1}\atop|\nu|=\delta} \alpha_\nu z^\nu.
$$
The number of indices $\nu\in\left({\mathbb N}\cup\{0\}\right)^{n+1}$ with $|\nu|=\delta$ is
${\delta+n\choose n}=N+1$. For $\alpha\in{\mathbb P}_N$ we use $X^{(\alpha)}$ to
denote ${\mathcal X}\cap\left({\mathbb P}_n\times\{\alpha\}\right)$.

\begin{lemma}\label{Lemma(1.2)} ${\mathcal X}$ is a
nonsingular hypersurface of ${\mathbb P}_n\times{\mathbb P}_N$ of
bidegree $(\delta,1)$.
\end{lemma}

\begin{proof} Take an arbitrary point
$(y,\alpha)$ of ${\mathcal X}$ with $y\in{\mathbb P}_n$ and $\alpha\in{\mathbb P}_N$.
Choose a homogeneous coordinate system
$[z_0,z_1,\cdots,z_n]$ of ${\mathbb P}_n$ so that $y$ is given
by $[z_0,z_1,\cdots,z_n]=[1,0,\cdots,0]$. In other words, $y$ is
the origin in the inhomogeneous coordinate system $$\left(x_1,\cdots,x_n\right)=\left({z_1\over
z_0},{z_2\over z_0},\cdots,{z_n\over z_0}\right)$$ associated to
the homogeneous coordinate system $[z_0,z_1,\cdots,z_n]$.

\medbreak The hypersurface ${\mathcal X}$ in ${\mathbb P}_n\times{\mathbb P}_N$ is nonsingular if and only if its
pullback $\tilde{\mathcal X}$ to $({\mathbb C}^{N+1}-0)\times({\mathbb C}^{n+1}-0)$ is nonsingular, because locally at points of
${\mathcal X}$ the pullback $\tilde{\mathcal X}$ is simply equal
to the product of ${\mathcal X}$ with $({\mathbb C}-0)\times({\mathbb C}-0)$.

\medbreak To determine whether $\tilde{\mathcal X}$ is
nonsingular, we differentiate the defining function
$$
f=\sum_{\nu_0+\cdots+\nu_n=\delta}\alpha_{\nu_0,\cdots,\nu_n}
z_0^{\nu_0}\cdots z_n^{\nu_n}
$$
with respect to each $\alpha_{\nu_0,\cdots,\nu_n}$ and each $z_j$
and evaluate the results at $z_0=1,\, z_1=\cdots=z_n=0$ to see
whether we get a nonzero
$\left(\left(N+1\right)+\left(n+1\right)\right)$-vector. We choose
$\nu_0=\delta,\, \nu_1=\cdots=\nu_n=0$ and get
$$
\left({\partial f\over\partial\alpha_{\delta,0,\cdots,0}}\right)_{z_0=1,z_1=\cdots=z_n=0}=1
$$
and conclude that the
$\left(\left(N+1\right)+\left(n+1\right)\right)$-vector is
nonzero. Thus ${\mathcal X}$ is nonsingular at every point
$(y,\alpha)\in{\mathbb P}_n\times{\mathbb P}_N$ which belongs to
${\mathcal X}$.
\end{proof}

\begin{remark}\label{remark(1.3)}  Though ${\mathcal X}$ is
nonsingular at every point $(y,\alpha)\in{\mathbb P}_n\times{\mathbb
P}_N$ which belongs to ${\mathcal X}$, the hypersurface $X^{(\alpha)}$ in
${\mathbb P}_n$ which corresponds to $\alpha$ and is equal to
${\mathcal X}\cap\left({\mathbb P}_n\times\{\alpha\}\right)$ may have
singularities.
\end{remark}

\begin{lemma}\label{lemma(1.4)} Let $\ell$ be a
positive integer and let $L$ be a homogeneous polynomial of degree
$\ell$ in the variables $\{\alpha_\nu\}_{|\nu|=\delta}$. Let
$0\leq p\not=q\leq n$ and $\nu,\mu\in\left({\mathbb N}\cup\{0\}\right)^{n+1}$ such that
$\nu+e_q=\mu+e_p$. Then the ${\mathcal O}_{{\mathbb P}_n}(1)\times {\mathcal O}_{{\mathbb
P}_N}(\ell-1)$-valued global
holomorphic vector field
$$
L\left(z_q\left({\partial\over\partial\alpha_\nu}\right)
-z_p\left({\partial\over\partial\alpha_\mu}\right)\right)
$$
on ${\mathbb P}_n\times{\mathbb P}_N$ is tangential to ${\mathcal
X}$.
\end{lemma}

\begin{proof}
The expression
$$
L\left(z_q\left({\partial\over\partial\alpha_\nu}\right)
-z_p\left({\partial\over\partial\alpha_\mu}\right)\right)
$$
is a ${\mathcal
O}_{{\mathbb P}_n}(1)\times {\mathcal O}_{{\mathbb P}_N}(\ell-1)$-valued global holomorphic vector field on
${\mathbb P}_n\times{\mathbb P}_N$, because the tangent bundle of
${\mathbb P}_N$ is generated by global holomorphic vector fields
of the form
$$
\sum_{|\mu|=|\nu|=\delta}A_{\mu,\nu}\alpha_\mu\frac{\partial}{\partial\alpha_\nu}
$$
with $A_{\mu,\nu}\in{\mathbb C}$.

\medbreak The hypersurface ${\mathcal X}$ is defined by
$f=\sum_{|\nu|=\delta}\alpha_\nu z^\nu$.  From
$$
\frac{\partial f}{\partial\alpha_\nu}=z^\nu
$$
and $\nu+e_q=\mu+e_p$ it follows that
$$
\left(L\left(z_q\left({\partial\over\partial\alpha_\nu}\right)
-z_p\left({\partial\over\partial\alpha_\mu}\right)\right)\right)
f= L\left(z_q z^\nu-z_p z^\mu\right)=0.
$$  Hence the ${\mathcal O}_{{\mathbb
P}_n}(1)\times{\mathcal O}_{{\mathbb P}_N}(\ell-1)$-valued global holomorphic vector field
$$
L\left(z_q\left({\partial\over\partial\alpha_\nu}\right)
-z_p\left({\partial\over\partial\alpha_\mu}\right)\right)
$$
on ${\mathbb P}_n\times{\mathbb P}_N$ is tangential to the
hypersurface ${\mathcal X}$ of ${\mathbb P}_n\times{\mathbb P}_N$
\end{proof}

\begin{remark}\label{remark(1.5)}  The use of $L$ is to
make sure that we have a line-bundle-valued global holomorphic
tangent vector field on ${\mathbb P}_n\times{\mathbb P}_N$. We
will only use the case $\ell=1$.
\end{remark}

\begin{lemma}\label{lemma(1.6)} For any global
holomorphic vector field $\xi$ on ${\mathbb P}_n$ there exists a
global holomorphic vector field $\tilde\xi$ on ${\mathbb P}_n\times{\mathbb P}_N$ such that

\medbreak \noindent (i) $\tilde\xi$ is tangential to ${\mathcal
X}$, and

\medbreak \noindent (ii) $\tilde\xi$ is projected to $\xi$ under
the natural projection from ${\mathbb P}_n\times{\mathbb P}_N$
onto the second factor ${\mathbb P}_n$.
\end{lemma}

\begin{proof} Consider the Euler sequence
$$
0\rightarrow{\mathcal O}_{{\mathbb
P}_n}\stackrel{\phi}{\rightarrow} {\mathcal O}_{{\mathbb
P}_n}(1)^{\oplus(n+1)}\stackrel{\psi}{\rightarrow} T_{{\mathbb
P}_n}\rightarrow 0,
$$
where
$$
\displaylines{
\psi\left(\sum_{j=0}^n a_{j,0}z_j,\sum_{j=0}^n a_{j,1}z_j,\cdots,
\sum_{j=0}^n a_{j,n}z_j\right)
=\sum_{j,k=0}^n a_{j,k}z_j{\partial\over\partial z_k},\cr
\phi(1)=(z_0,\cdots,z_n).\cr
}
$$
Since $H^1({\mathbb P}_n,{\mathcal O}_{{\mathbb P}_n})$ vanishes,
it follows from the exact cohomology sequence of the Euler
sequence that $\xi$ is of the form
$\sum_{j,k=0}^na_{j,k}z_j{\partial\over\partial z_k}$ for some
complex numbers $a_{j,k}$.

\medbreak For $0\leq j,k\leq n$ with $j\not=k$ we define
$$
\Phi_{j,k}:\left({\mathbb N}\cup\{0\}\right)^{n+1}\rightarrow\left({\mathbb N}\cup\{0\}\right)^{n+1}
$$
as follows. For $\nu\in\left({\mathbb N}\cup\{0\}\right)^{n+1}$ we set
$$
\displaylines{
\left(\Phi_{j,k}(\nu)\right)_\ell=\nu_\ell\quad\forall\ell\not=j,k,\cr
\left(\Phi_{j,k}(\nu)\right)_j=\nu_j-1,\cr
\left(\Phi_{j,k}(\nu)\right)_k=\nu_k+1.\cr }
$$
For $\{\alpha_\nu\}_{\nu\in\left({\mathbb N}\cup\{0\}\right)^{n+1}}$
we define
$$
\beta_\nu=
-\sum_{0\leq j,k\leq n,j\not=k}
\alpha_{
\Phi_{j,k}(\nu)}
a_{j,k}(\nu_k+1)
-\sum_{j=0}^n\alpha_{\nu_0,\nu_1,\cdots,\nu_n}a_{j,j}\nu_j.
$$
Then for $f=\sum_{|\nu|=\delta}\alpha_\nu z^\nu$,
we have
$$
\xi(f)=\left(\sum_{j,k=0}^n
a_{j,k}z_j{\partial\over\partial z_k}\right)
f=
-\sum_{|\nu|=\delta}
\beta_\nu z^\nu.
$$
The verification is as follows.
Since
$$
\left({\partial\over\partial z_k}\right)f=
\sum_{|\nu|=\delta}\alpha_\nu\,\nu_k\, z_0^{\nu_0}\cdots
z_{k-1}^{\nu_{k-1}}z_k^{\nu_k-1}z_{k+1}^{\nu_{k+1}}\cdots
z_n^{\nu_n},
$$
it follows that
$$
\qquad\left(\sum_{j,k=0}^n a_{j,k}z_j{\partial\over\partial
z_k}\right) f\leqno{(\ref{lemma(1.6)}.1)}
$$
$$
=\sum_{j,k=0}^n a_{j,k}z_j
\sum_{|\nu|=\delta}\alpha_\nu\nu_k z_jz_0^{\nu_0}\cdots
z_{k-1}^{\nu_{k-1}} z_k^{\nu_k-1}z_{k+1}^{\nu_{k+1}}\cdots
z_n^{\nu_n}.
$$
So, the term on the right-hand side of (\ref{lemma(1.6)}.1) when $j=k$
is
$$
a_{j,j}\alpha_{\nu_0,\nu_1,\cdots,\nu_n}\nu_k
z_0^{\nu_o}z_1^{\nu_1}\cdots z_n^{\nu_n}.
$$
This means that the net effect is multiplication by $\nu_j$ when
$j=k$. The contribution to the term on the right-hand side of
(\ref{lemma(1.6)}.1) with $j\not=k$ is
$$
a_{jk}\nu_k
\alpha_{\nu_0,\nu_1,\cdots,\nu_n}
z_0^{\nu_0}\cdots z_{j-1}^{\nu_{j-1}}z_j^{\nu_j+1}
z_{j+1}^{\nu_{j+1}}\cdots
z_{k-1}^{\nu_{k-1}}
z_k^{\nu_k-1}z_{k+1}^{\nu_{k+1}}\cdots z_n^{\nu_n}.
$$
We now change the dummy indices $\nu_j$ and $\nu_k$
to look at the coefficient of the monomial
$$
z_0^{\nu_0}z_1^{\nu_1}\cdots z_n^{\nu_n}.
$$
We change the dummy index $\nu_j$ to $\nu_j-1$ and change the
dummy index $\nu_k$ to $\nu_k+1$ to get
$$
(\nu_k+1)a_{j,k}
\alpha_{\nu_0,\cdots\nu_{j-1},\nu_j-1,\nu_{j+1},\cdots,
\nu_{k-1},\nu_k+1,\nu_{k+1},\cdots,\nu_n}
z_0^{\nu_0}z_1^{\nu_1}\cdots z_n^{\nu_n}.
$$
This concludes the verification.

\medbreak It now suffices to set
$$
\tilde\xi=\sum_{j,k=0}^na_{j,k}z_j{\partial\over\partial z_k}
+\sum_{|\nu|=\delta}\beta_\nu{\partial\over\partial\alpha_\nu}.
$$
\end{proof}

\begin{lemma}\label{lemma(1.7)} $T_{\mathcal X}\otimes{\mathcal O}_{{\mathbb P}_n}(1)$ is globally generated.
\end{lemma}

\begin{proof} Take an arbitrary point
$(y,\alpha)$ of ${\mathcal X}$ with $y\in{\mathbb P}_n$ and $\alpha\in{\mathbb P}_N$.  Again choose a homogeneous coordinate
system $[z_0,z_1,\cdots,z_n]$ of ${\mathbb P}_n$ so that $y$ is
given by $[z_0,z_1,\cdots,z_n]=[1,0,\cdots,0]$.  We choose a
homogeneous linear polynomial $L$ of the variables
$\left\{\alpha_\nu\right\}_{\left|\nu\right|=\delta}$ such that
$L\left(s\right)\not=0$.

\medbreak It is equivalent to look at the tangent bundle of the
pullback $\tilde{\mathcal X}$ of ${\mathcal X}\subset\left({\mathbb C}^{n+1}-0\right) \times\left({\mathbb C}^{N+1}-0\right)$ under the
natural projection $\left({\mathbb C}^{n+1}-0\right) \times\left({\mathbb C}^{N+1}-0\right) \rightarrow{\mathbb P}_n\times{\mathbb P}_N$.
Take $\nu$ with $\nu_p>0$ for some $1\leq p\leq n$. Then there
exists a unique $\mu\in\left({\mathbb N}\cup\{0\}\right)^{n+1}$ such that $\nu+e_0=\mu+e_p$.
At $(y,\alpha)$ the value of
$$
L\left(z_0\left({\partial\over\partial\alpha_\nu}\right)
-z_p\left({\partial\over\partial\alpha_\mu}\right)\right)
$$
is equal to $L(s)\left({\partial\over\partial\alpha_\nu}\right)$.
Thus we conclude that the global holomorphic sections of
$T_{\mathcal X}\times{\mathcal O}_{{\mathbb P}_n}(1)$ generate
${\partial\over\partial\alpha_\nu}$ for
$\nu\not=(\delta,0,\cdots,0)$.  By Lemma~\ref{lemma(1.6)} the
global holomorphic sections of $T_{\mathcal X}\times{\mathcal
O}_{{\mathbb P}_n}(1)$ also generate ${\partial\over\partial z_j}$
for $0\leq j\leq n$. We thus conclude that global holomorphic
sections of $T_{\mathcal X}\times{\mathcal O}_{{\mathbb P}_n}(1)$
generate a codimension $1$ vector subspace of the tangent space of
${\mathbb P}_n\times{\mathbb P}_N$ at $(y,\alpha)$. Since ${\mathcal
X}$ is nonsingular at $(\alpha,y)$, it follows that global
holomorphic sections of $T_{\mathcal X}\otimes{\mathcal
O}_{{\mathbb P}_n}(1)$ generate $T_{\mathcal X}\otimes{\mathcal
O}_{{\mathbb P}_n}(1)$.
\end{proof}

\begin{lemma}\label{lemma(1.8)} Let $q$ be a
nonnegative integer. Global holomorphic sections of $T_{\mathcal
X}\otimes{\mathcal O}_{{\mathbb P}_n}(q+1)\otimes{\mathcal
O}_{{\mathbb P}_N}(q)$ generate all of its $q$-jets of ${\mathcal
X}$.
\end{lemma}

\begin{proof} Global holomorphic section
of ${\mathcal O}_{{\mathbb P}_n}(q)\otimes{ \cal O}_{{\mathbb
P}_N}(q)$ generate all of its $q$-jets. Thus we can use the
product of a global holomorphic section of ${\mathcal O}_{{\mathbb
P}_n}(q)\otimes{ \cal O}_{{\mathbb P}_N}(q)$ and a global
holomorphic section of $T_{\mathcal X}\otimes{\mathcal
O}_{{\mathbb P}_n}(1)$ to generate any prescribed $q$-jet of
${\mathcal X}$.
\end{proof}

\begin{nodesignation}\label{lie_derivatives} {\it Lie Derivatives}\end{nodesignation} Let $X$
be a complex manifold and $\xi$ be a holomorphic vector field on
$X$. Let $\varphi_{\xi,t}$ be a $1$-parameter local biholomorphism
defined by the vector field $\xi$ so that
$$\left.{\partial\over\partial
t}\varphi_{\xi,t}^*g\right|_{t=0}=\xi(g)$$ for any local
holomorphic function $g$ on $X$.  For any $k$-jet differential
$\omega$ on $X$, we define the Lie derivative ${{\mathcal
L}ie}_\xi(\omega)$ of $\omega$ with respect to $\xi$ by
$$
\left.{{\mathcal L}ie}_\xi(\omega)={\partial\over\partial t}
\varphi_{X,t}^*\omega\right|_{t=0}.
$$
Since
$$
\displaylines{d\left(\left.{\partial\over\partial
t}\varphi_{X,t}^*\omega\right|_{t=0}\right)
=d\left(\lim_{t\rightarrow 0} {1\over
t}\left(\varphi_{X,t}^*\omega - \omega\right)\right)\cr
=\lim_{t\rightarrow 0} {1\over t}\left(d\varphi_{X,t}^*\omega
-d\omega\right) =\lim_{t\rightarrow 0} {1\over
t}\left(\varphi_{X,t}^*d\omega -d\omega\right) \cr
=\left.{\partial\over\partial
t}\left(\varphi_{X,t}^*\left(d\omega\right)\right)\right|_{t=0},\cr
}
$$
it follows that
$$
d\left({{\mathcal L}ie}_\xi(\omega)\right)= {{\mathcal
L}ie}_\xi\left(d\omega\right).
$$
Let $\eta$ be a holomorphic $\ell$-jet differential
on $X$.  The Leibniz product formula holds for the Lie derivatives of
the product of $\omega$ and $\eta$ so that
$$
{{\mathcal L}ie}_\xi(\omega\eta) ={{\mathcal
L}ie}_\xi(\omega)\eta+ \omega{{\mathcal L}ie}_\xi(\eta).
$$
Let $(w_1,\cdots,w_n)$ be a local coordinate system of $X$. Fix
some $1\leq i\leq n$. If $\omega=d^k w_i$ and $\xi=\sum_{j=1}^n
g_j(w){\partial\over\partial w_j}$, then
$$
{{\mathcal L}ie}_\xi(\omega) =d^k\left(\sum_{j=1}^n
g_j(w)\left({\partial\over\partial w_j}\right)w_i\right) =d^k
g_i(w).
$$

\begin{lemma}\label{lemma(1.10)} Let $k$ be a positive
integer. Let $X$ be a complex manifold and $D$ be a complex
hypersurface in $X$.  Let $\omega$ be a holomorphic $k$-jet
differential on $X$ which vanishes at points of $D$ to order $p$.
Let $\xi$ be a meromorphic vector field on $X$ whose only possible
poles are those of order at most $q$ at $D$.  If $p\geq q+k$, then
${{\mathcal L}ie}_\xi(\omega)$ is a holomorphic $k$-jet
differential on $X$ which vanishes at points of $D$ to order
$p-(q+k)$.
\end{lemma}

\begin{proof} Locally we can write
$$
\omega=\sum_ {\lambda_{1,1},\cdots,\lambda_{1,n},\cdots,
\lambda_{k,1},\cdots,\lambda_{k,n}\geq 0}
h_{\lambda_{1,1},\cdots,\lambda_{1,n},\cdots,
\lambda_{k,1},\cdots,\lambda_{k,n}}(w) \prod_{1\leq\ell\leq
k,1\leq j\leq n}\left(d^\ell w_j\right)^{\lambda_{\ell,j}}
$$
with $h_{\lambda_{1,1},\cdots,\lambda_{1,n},\cdots,
\lambda_{k,1},\cdots,\lambda_{k,n}}(w)$ vanishing at points of $D$
to order $p$.

\medbreak Locally we can write $\xi=\sum_{j=1}^n
g_j(w){\partial\over\partial w_j}$ with the pole order of $g_j$ at
most $q$ at $D$. When we apply ${{\mathcal L}ie}_\xi$, by the
Leibniz product rule we apply it only to one factor of each term
separately and sum up. When it is applied to
$h_{\lambda_{1,1},\cdots,\lambda_{1,n},\cdots,
\lambda_{k,1},\cdots,\lambda_{k,n}}$, we end up with
$$
\sum_{j=1}^n g_j(w){\partial\over\partial w_j} \left(
h_{\lambda_{1,1},\cdots,\lambda_{1,n},\cdots,
\lambda_{k,1},\cdots,\lambda_{k,n}}(w) \right)
$$
which vanishes at points of $D$ to order $p-(q+1)\geq p-(q+k)$.
Since the pole order of $d^\ell g_j$ is at most $q+\ell$ at $D$,
when it is applied to $d^\ell w_j$ and then multiplied by
$h_{\lambda_{1,1},\cdots,\lambda_{1,n},\cdots,
\lambda_{k,1},\cdots,\lambda_{k,n}}$, we end up with
$h_{\lambda_{1,1},\cdots,\lambda_{1,n},\cdots,
\lambda_{k,1},\cdots,\lambda_{k,n}} d^\ell g_j$ which vanishes at
points of $D$ to order $p-(q+\ell)\geq p-(q+k)$. \end{proof}

\bigbreak
\section{\sc Construction of Slanted Vector Fields for Jet Space}\label{construction_slanted_vector_fields_jet_space}

\bigbreak In the preceding \S\ref{vector_field_and_lie_derivatives} we constructed global vector fields on ${\mathcal X}$ of low vertical pole order.  The construction in Lemma \ref{lemma(1.4)} involves the use of two indices $\nu,\mu\in\left({\mathbb N}\cup\{0\}\right)^{n+1}$ with
$\nu+e_q=\mu+e_p$ and the construction in Lemma \ref{lemma(1.7)} involves also the two indices $\mu$ and $\nu$ for the case of $q=0$.  In this \S\ref{construction_slanted_vector_fields_jet_space} we are going to carry out a similar construction with ${\mathcal X}$ replaced by the space $J^{\rm{\scriptstyle (vert)}}_k\left({\mathcal X}\right)$ of all vertical $k$-jets on ${\mathcal X}$ for $k\geq 1$ by induction on $k$, with the similarity interpreted from identifying ${\mathcal X}$ with the space
the space $J^{\rm{\scriptstyle (vert)}}_0\left({\mathcal X}\right)$ of all vertical $0$-jets on ${\mathcal X}$. For each step of this construction by induction, just as in the construction in Lemma \ref{lemma(1.4)} and Lemma \ref{lemma(1.7)} a pair of indices will be used, leading to some binary tree of indices whose precise definition will be given in
\ref{binary_tree_of_indices}.  Because of the number of indices involved in the construction given in this \S\ref{construction_slanted_vector_fields_jet_space}, the details seem to be complicated, but the key argument is simply a natural extension of what is used in the preceding \S\ref{vector_field_and_lie_derivatives} and all the steps are straightforward.

\begin{nodesignation}{\it Binary Trees of Indices}.\label{binary_tree_of_indices}\end{nodesignation}
In order to conveniently describe vector fields on the total space
of $k$-jets, we now introduce the binary trees of indices.  A {\it
binary tree of indices of order} $k$, which we denote by
${\mathfrak p}^{(k)}$, is a collection
$$
\left\{p_{\gamma_1,\gamma_2,\cdots,\gamma_j}\,\Bigg| \,1\leq j\leq
k,\hbox{ each of }\gamma_1,\cdots,\gamma_j=0,1\right\}
$$
of indices, where each $p_{\gamma_1,\gamma_2,\cdots,\gamma_j}$ is
an integer satisfying $0\leq
p_{\gamma_1,\gamma_2,\cdots,\gamma_j}\leq n$. We will use the
interpretation of this collection of indices as a tree as follows.
The binary tree starts with two nodes $p_0, p_1$ at its root and
each of these two notes $p_0, p_1$ branches out into a pair of
nodes. On top of the node $p_0$ there are two nodes $p_{0,0}$ and
$p_{0,1}$. On top of the node $p_1$ there are two nodes $p_{1,0}$
and $p_{1,1}$.  Each of the four nodes $p_{0,0}, p_{0,1}, p_{1,0},
p_{1,1}$ again branches out into a pair of nodes.  On top of the
node $p_{\gamma_1,\gamma_2}$ ($\gamma_1=0,1;\,\gamma_2=0,1$) there
are two nodes $p_{\gamma_1,\gamma_2,0}$ and
$p_{\gamma_1,\gamma_2,1}$. At the $j$-th branching into a pair of
two nodes for each node, we have two nodes
$$
p_{\gamma_1,\gamma_2,\cdots,\gamma_j,0},\,
p_{\gamma_1,\gamma_2,\cdots,\gamma_j,1}$$ on top of the node
$p_{\gamma_1,\gamma_2,\cdots,\gamma_j}$ for
$$\gamma_1=0,1;\,\gamma_2=0,1;\cdots,\gamma_j=0,1.$$
At the top the tree, after the $(k-1)$-th branching we have the
notes
$$
p_{\gamma_1,\gamma_2,\cdots,\gamma_k}\qquad\left(
\gamma_1=0,1;\,\gamma_2=0,1;\cdots,\gamma_k=0,1\right).
$$
We will use the convention that the binary tree, as a collection
of indices, will be denoted by a lower case Gothic letter and its
indices are denoted by the corresponding lower case Latin letter.
When $k=0$, we use the convention that ${\mathfrak p}^{(0)}$ is
just the empty set.

\medbreak We now introduce the {\it truncation} of a binary tree
of order $k$ to form a binary subtree of order $k-j$. We denote by
${\mathfrak p}^{(k;\tilde\gamma_1,\cdots,\tilde\gamma_j)}$ the
binary tree
$$
\left\{
p_{\tilde\gamma_1,\tilde\gamma_2,\cdots,\tilde\gamma_j,\gamma_{j+1},\cdots,\gamma_i}\,\Bigg|
\,j+1\leq i\leq k,\hbox{ each of
}\gamma_{j+1},\cdots,\gamma_i=0,1\right\}
$$
of order $k-j$.  We call ${\mathfrak
p}^{(k;\tilde\gamma_1,\cdots,\tilde\gamma_j)}$ the {\it
truncation} of ${\mathfrak p}^{(k)}$ at its node
$p_{\tilde\gamma_1,\tilde\gamma_2,\cdots,\tilde\gamma_j}$.

\medbreak In this paper we work only with the following special
kind of binary trees.  A binary tree ${\mathfrak p}^{(k)}$ of
order $k$ is said to {\it have level-wise homogeneous branches} if
for every $1\leq j<l$ and for any pairs
$$\left(\tilde\gamma_1,\tilde\gamma_2,\cdots,\tilde\gamma_j\right)
{\rm \ \ and\ \ }
\left(\hat\gamma_1,\hat\gamma_2,\cdots,\hat\gamma_j\right)$$ of
$j$-tuples of $0$'s and $1$'s, the two truncations $${\mathfrak
p}^{(k;\tilde\gamma_1,\tilde\gamma_2,\cdots,\tilde\gamma_j)},\ \
{\mathfrak
p}^{(k;\hat\gamma_1,\hat\gamma_2,\cdots,\hat\gamma_j)}$$ of
${\mathfrak p}^{(k)}$ are identical binary trees of order $k-j$.

\setcounter{subsection}{0}
\begin{lemma}\label{lemma(2.2)} Let
$$
\xi^{(i)}_j=d^i\log z_j\quad(i\geq 1,\ 0\leq j\leq n)
$$
and
$$
\xi^{(i)}=\left(\xi^{(i)}_0,\xi^{(i)}_1,\cdots,\xi^{(i)}_n\right).
$$
Let $f=\sum_\nu\alpha_\nu z^\nu$ be a polynomial of homogeneous
degree $\delta$. Let $\Phi^{(0)}_\nu\equiv 1$ and inductively for
$k\geq 0$,
$$
\Phi^{(k+1)}_\nu\left(\xi^{(1)},\cdots,\xi^{(k+1)}\right)
=\left(\sum_{\ell=0}^n \nu_\ell\xi^{(1)}_\ell\right)\Phi^{(k)}_\nu
+\sum_{\ell=0}^n\sum_{i=1}^k \xi^{(i+1)}_\ell {\partial
\Phi^{(k)}_\nu\over\partial\xi^{(i)}_\ell}.\leqno{(\ref{lemma(2.2)}.1)_{k+1}}
$$
Then
$$
d^k f=\sum_\nu\alpha_\nu
\Phi^{(k)}_\nu\left(\xi^{(1)},\cdots,\xi^{(k)}\right)z^\nu,\leqno{(\ref{lemma(2.2)}.2)}
$$
where for the differentiation $\alpha_\nu$ is regarded as a
constant and $\Phi_\nu^{(k)}$ is of homogeneous weight $k$ when
$\xi^{(j)}_\ell$ is given the weight $j$ and is independent of
$\left(z_0,\cdots,z_n\right)$. Moreover, the coefficients of
$\Phi_\nu^{(k)}$ in $\xi^{(1)},\cdots,\xi^{(k)}$ are polynomials
in $\nu=\left(\nu_0,\cdots,\nu_n\right)$ of degree at most $k$
with universal coefficients. As a polynomial in
$\nu=\left(\nu_0,\cdots,\nu_n\right)$, the degree of
$$\Phi_\nu^{(k)}-\left(\sum_{\ell=0}^n\nu_\ell\xi^{(1)}_\ell\right)^k$$
is at most $k-1$.
\end{lemma}

\begin{proof} To prove the Lemma, we
define $\Phi_\nu^{(k)}$ by (\ref{lemma(2.2)}.2) and verify
$(\ref{lemma(2.2)}.1)_1$ and $(\ref{lemma(2.2)}.1)_{k+1}$ for $k\geq 1$. The
verification is as follows. Clearly,
$$
\Phi^{(1)}_\nu=\sum_{\ell=0}^n\nu_\ell\xi^{(1)}_\ell
$$
and is homogeneous of weight $1$. To verify
(\ref{lemma(2.2)}.2) when $k$ is replaced by $k+1$, we apply
$d$ to both sides of (\ref{lemma(2.2)}.2).  The effect of
applying $d$ to $z^\nu$ is to replace the factor $z^\nu$ by
$\Phi^{(1)}_\nu$ which is $\sum_{\ell=0}^n
\nu_\ell\xi_\ell^{(1)}$. The effect of applying $d$ to the other
factor $\Phi^{(k)}_\nu\left(\xi^{(1)},\cdots,\xi^{(k)}\right)$ is
$$\sum_{\ell=0}^n\sum_{i=1}^k \xi^{(i+1)}_\ell {\partial
\Phi^{(k)}_\nu\over\partial\xi^{(i)}_\ell}$$ by the chain rule. We
now consider the question of weights and homogeneous degrees. For
$k\geq 1$ the first term on the right-hand side of
$(\ref{lemma(2.2)}.1)_{k+1}$ is the product of the factors
$\Phi^{(1)}_\nu\Phi^{(k)}_\nu$ which are respectively homogeneous
of weights $1$ and $k$ by induction hypothesis.  The second term
on the right-hand side of $(\ref{lemma(2.2)}.1)_{k+1}$ is the sum of a
product of two factors
$$\xi^{(i+1)}{\partial
\Phi^{(k)}_\nu\over\partial\xi^{(i)}_\ell}$$ which are
respectively homogeneous of weights $i+1$ and $k-i$ by induction
hypothesis. This finishes the verification of $(\ref{lemma(2.2)}.1)_{k+1}$
and the homogeneity of $\Phi^{(k+1)}_\nu$ of weight $k+1$.  From
$(\ref{lemma(2.2)}.1)_{k+1}$ it is clear by induction on $k$ that the
coefficients of $\Phi_\nu^{(k)}$ in $\xi^{(1)},\cdots,\xi^{(k)}$
are polynomials in $\nu$ of degree at most $k$ with universal
coefficients.

\medbreak Finally, in $(\ref{lemma(2.2)}.1)_{k+1}$ the term
$$\sum_{\ell=0}^n\sum_{i=1}^k \xi^{(i+1)}_\ell {\partial
\Phi^{(k)}_\nu\over\partial\xi^{(i)}_\ell}$$ on the right-hand
side as a polynomial in $\nu$ is of degree no higher than that of
$\Phi^{(k)}_\nu$ which is no higher than $k$.  Thus for the
induction process of going from Step $k$ to Step $k+1$, if the
degree of
$$\Phi_\nu^{(k)}-\left(\sum_{\ell=0}^n\nu_\ell\xi^{(1)}_\ell\right)^k$$
is at most $k-1$ in $\nu=\left(\nu_0,\cdots,\nu_n\right)$, then by
$(\ref{lemma(2.2)}.1)_{k+1}$ the degree of
$$\Phi_\nu^{(k+1)}-\left(\sum_{\ell=0}^n\nu_\ell\xi^{(1)}_\ell\right)^{k+1}$$
is at most $k$ in
$\nu=\left(\nu_0,\cdots,\nu_n\right)$.\end{proof}

\medbreak\noindent\begin{nodesignation}\label{2.3}
{\it Construction by Induction.}
\end{nodesignation} Let ${\mathfrak p}^{(k)}$ be a binary
tree of indices of order $k$. We denote by $\lambda^{(k)}$ a
multi-index of $n+1$ components with total degree $\delta-k$. For
$1\leq j\leq k$ and for the choice of each
$\gamma_1,\cdots,\gamma_j$ being $0$ or $1$, we denote by
$\lambda^{(k;\gamma_1,\cdots,\gamma_j)}$ the multi-index
$$
\lambda^{(k)}+\sum_{i=1}^j e_{p_{\gamma_1,\cdots,\gamma_i}}
$$
with total degree
$\left|\lambda^{\left(k;\gamma_1,\cdots,\gamma_j\right)}\right|=\delta-k+j$.
Recall that $e_{p_{\gamma_1,\cdots,\gamma_i}}$ is the index of
$n+1$ components whose only nonzero component is the
$p_{\gamma_1,\cdots,\gamma_i}$-th component which is $1$.

\bigbreak For $0\leq k\leq n-1$, for any multi-index
$\lambda^{(k)}$ of $n+1$ components with total degree $\delta-k$,
and for any binary tree ${\mathfrak p}^{(k)}$ of order $k$ which
has level-wise homogeneous branches, we are going to explicitly
construct by induction on $k$,
$$
\Theta^{(j)}_{\lambda^{(k)},{\mathfrak p}^{(k)}},\,
\Psi^{(j)}_{\lambda^{(k)},{\mathfrak p}^{(k)}}\qquad\left(k\leq
j\leq n\right)
$$
such that

\medbreak\noindent (1) $\Psi^{(j)}_{\lambda^{(k)},{\mathfrak
p}^{(k)}}$ is a rational function of the entries of
$\lambda^{(k)}$ and $\xi_q^{(\ell)}$ with $q$ equal to some
$p_{\gamma_1,\gamma_2,\cdots,\gamma_\ell}$ for $1\leq\ell\leq k$,

\medbreak\noindent (2) $\Theta^{(k)}_{\lambda^{(k)},{\mathfrak
p}^{(k)}}$ is a meromorphic vector field on the parameter space
with coordinates $\alpha_\nu$ (for multi-indices $\nu$ of $n+1$
components of total degree $\delta$) which is a linear combination
of $\partial\over\partial\alpha_\nu$ (for
$\left|\nu\right|=\delta$) whose coefficients are rational
functions of
$$z_0,\cdots,z_n,\ \xi^{(j)}_\ell\quad
(1\leq j\leq k,\ 0\leq\ell\leq n)$$ and which satisfies
$$ \displaylines{
\Theta^{(k)}_{\lambda^{(k)},{\mathfrak p}^{(k)}}
\left(d^jf\right)=0\quad{\rm for \ }0\leq j\leq k-1,\cr
\Theta^{(k)}_{\lambda^{(k)},{\mathfrak p}^{(k)}} \left(d^j
f\right) =z^{\lambda^{(k)}}\Psi^{(j)}_{\lambda^{(k)},{\mathfrak
p}^{(k)}} \quad{\rm for \ }k\leq j\leq n-1.\cr }
$$
Here we regard $\Theta^{(k)}_{\lambda^{(k)},{\mathfrak p}^{(k)}}$
as a vector field on the space with variables $\alpha_\nu$ for
$|\nu|=\delta$ while the variables $z_0,\cdots,z_n$ and
$\xi^{(j)}_\ell$\ ($1\leq j\leq k,\,0\leq\ell\leq n$) are regarded
as constants.  It is the same as regarding
$\Theta^{(k)}_{\lambda^{(k)},{\mathfrak p}^{(k)}}$ as a vector
field on the space with variables $\alpha_\nu$ for $|\nu|=\delta$
and the variables $z_0,\cdots,z_n$ and $\xi^{(j)}_\ell$\ ($1\leq
j\leq k,\,0\leq\ell\leq n$) when the coefficients for
$\frac{\partial}{\partial z_\ell}$ and for
$\frac{\partial}{\partial\xi^{(j)}_\ell}$ are all $0$ for $1\leq
j\leq k,\,0\leq\ell\leq n$.  The construction is as follows.

\medbreak For $k=0$ with the convention that ${\mathfrak p}^{(0)}$
is the empty set, we simply set
$$
\Theta^{(0)}_{\lambda^{(0)},{\mathfrak
p}^{(0)}}={\partial\over\partial\alpha_{\lambda^{(0)}}}
$$
and
$$
\Psi^{(j)}_{\lambda^{(0)},{\mathfrak
p}^{(0)}}=\Phi^{(j)}_{\lambda^{(0)}}\leqno{(\ref{2.3}.1)}
$$
for $0\leq j\leq n-1$.  It is clear that
$\Psi^{(j)}_{\lambda^{(0)},{\mathfrak p}^{(0)}}$ is of homogeneous
weight $j$ in $\xi^{(\ell)}_t$ ($1\leq\ell\leq j$, $0\leq t\leq
n$) and is independent of $\left(z_0,\cdots,z_n\right)$ and is a
polynomial in the $n+1$ components of $\lambda^{(0)}$ of degree
$\leq j$.  Moreover, it follows from $f=\sum_\nu\alpha_\nu z^\nu$
and the definition of $\Phi^{(j)}_\nu$ that
$$
\Theta^{(0)}_{\lambda^{(0)},{\mathfrak
p}^{(0)}}(d^jf)=z^{\lambda^{(0)}}\Psi^{(j)}_{\lambda^{(0)},{\mathfrak
p}^{(0)}}\leqno{(\ref{2.3}.2)}
$$
for $0\leq j\leq
n-1$. Suppose the construction has been done for the step $k$ and
we are going to construct for the step $k+1$. Define
$$
\Theta^{(k+1)}_{\lambda^{(k+1)},{\mathfrak p}^{(k+1)}}
={\Theta^{(k)}_{\lambda^{(k+1;p_0)},{\mathfrak
p}^{(k+1;p_0)}}\over z_{p_0}
\Psi^{(k)}_{\lambda^{(k+1;p_0)},{\mathfrak p}^{(k+1;p_0)}}} -
{\Theta^{(k)}_{\lambda^{(k+1;p_1)},{\mathfrak p}^{(k+1;p_1)}}\over
z_{p_1} \Psi^{(k)}_{\lambda^{(k+1;p_1)},{\mathfrak
p}^{(k+1;p_1)}}}
$$
and
$$
\Psi^{(j)}_{\lambda^{(k+1)},{\mathfrak p}^{(k+1)}}
={\Psi^{(j)}_{\lambda^{(k+1;p_0)},{\mathfrak p}^{(k+1;p_0)}}\over
\Psi^{(k)}_{\lambda^{(k+1;p_0)},{\mathfrak p}^{(k+1;p_0)}}} -
{\Psi^{(j)}_{\lambda^{(k+1;p_1)},{\mathfrak p}^{(k+1;p_1)}}\over
\Psi^{(k)}_{\lambda^{(k+1;p_1)},{\mathfrak
p}^{(k+1;p_1)}}}\leqno{(\ref{2.3}.3)}
$$
for $j\geq
k+1$.

\begin{lemma}\label{lemma(2.4)} For any
integer $k$ with $0\leq k\leq\delta$ the following two identities
hold.
$$
\displaylines{ \Theta^{(k)}_{\lambda^{(k)},{\mathfrak p}^{(k)}}
\left(d^jf\right)=0\quad{\rm for \ }0\leq j\leq k-1,\cr
\Theta^{(k)}_{\lambda^{(k)},{\mathfrak p}^{(k)}} \left(d^j
f\right) =z^{\lambda^{(k)}}\Psi^{(j)}_{\lambda^{(k)},{\mathfrak
p}^{(k)}} \quad{\rm for \ }k\leq j\leq n.\cr }
$$
\end{lemma}

\begin{proof}  We prove by induction on
$k\geq 0$.  Since the positive integer $\delta$ is fixed once for
all, the induction on $k\geq 0$ is the same as descending
induction on the total degree $\delta-k$ of $\lambda^{(k)}$.  In
the case $k=0$ the statement is simply (\ref{2.3}.2).
To
go from Step $k$ to Step $k+1$, we have
$$
\Theta^{(k+1)}_{\lambda^{(k+1)},{\mathfrak p}^{(k+1)}}
\left(d^jf\right)=0\quad{\rm for \ }0\leq j\leq k-1,
$$
because $$\Theta^{(k)}_{\lambda^{(k)},{\mathfrak p}^{(k)}}
\left(d^jf\right)=0\quad{\rm for \ }0\leq j\leq k-1$$ and
$$
\Theta^{(k+1)}_{\lambda^{(k+1)},{\mathfrak p}^{(k+1)}}
={\Theta^{(k)}_{\lambda^{(k+1;p_0)},{\mathfrak
p}^{(k+1;p_0)}}\over z_{p_0}
\Psi^{(k)}_{\lambda^{(k+1;p_0)},{\mathfrak p}^{(k+1;p_0)}}} -
{\Theta^{(k)}_{\lambda^{(k+1;p_1)},{\mathfrak p}^{(k+1;p_1)}}\over
z_{p_1} \Psi^{(k)}_{\lambda^{(k+1;p_1)},{\mathfrak
p}^{(k+1;p_1)}}}.
$$
We have
$$
\Theta^{(k+1)}_{\lambda^{(k+1)},{\mathfrak p}^{(k+1)}}
\left(d^kf\right)=0,
$$
because
$$
\displaylines{\Theta^{(k+1)}_{\lambda^{(k+1)},{\mathfrak
p}^{(k+1)}}\left(d^kf\right)\cr
={\Theta^{(k)}_{\lambda^{(k+1;p_0)},{\mathfrak
p}^{(k+1;p_0)}}\left(d^kf\right)\over z_{p_0}
\Psi^{(k)}_{\lambda^{(k+1;p_0)},{\mathfrak p}^{(k+1;p_0)}}} -
{\Theta^{(k)}_{\lambda^{(k+1;p_1)},{\mathfrak
p}^{(k+1;p_1)}}\left(d^kf\right)\over z_{p_1}
\Psi^{(k)}_{\lambda^{(k+1;p_1)},{\mathfrak p}^{(k+1;p_1)}}}\cr ={
z^{\lambda^{(k+1;p_0)}}\Psi^{(k+1)}_{\lambda^{(k+1;p_0)},{\mathfrak
p}^{(k+1,p_0)}} \over z_{p_0}
\Psi^{(k)}_{\lambda^{(k+1;p_0)},{\mathfrak p}^{(k+1;p_0)}}} - {
z^{\lambda^{(k+1;p_1)}}\Psi^{(j)}_{\lambda^{(k+1;p_1)},{\mathfrak
p}^{(k+1;p_1)}} \over z_{p_1}
\Psi^{(k)}_{\lambda^{(k+1;p_1)},{\mathfrak p}^{(k+1;p_1)}}}\cr
=\lambda^{(k+1)}-\lambda^{(k+1)}=0.\cr}
$$
Moreover, for $k+1\leq j\leq n-1$,
$$
\displaylines{\Theta^{(k+1)}_{\lambda^{(k+1)},{\mathfrak
p}^{(k+1)}}\left(d^jf\right)\cr
={\Theta^{(k)}_{\lambda^{(k+1;p_0)},{\mathfrak
p}^{(k+1;p_0)}}\left(d^jf\right)\over z_{p_0}
\Psi^{(k)}_{\lambda^{(k+1;p_0)},{\mathfrak p}^{(k+1;p_0)}}} -
{\Theta^{(k)}_{\lambda^{(k+1;p_1)},{\mathfrak
p}^{(k+1;p_1)}}\left(d^jf\right)\over z_{p_1}
\Psi^{(k)}_{\lambda^{(k+1;p_1)},{\mathfrak p}^{(k+1;p_1)}}}\cr = {
z^{\lambda^{(k+1;p_0)}}\Psi^{(j)}_{\lambda^{(k+1;p_0)},{\mathfrak
p}^{(k+1,p_0)}} \over z_{p_0}
\Psi^{(k)}_{\lambda^{(k+1;p_0)},{\mathfrak p}^{(k+1;p_0)}}} - {
z^{\lambda^{(k+1;p_1)}}\Psi^{(j)}_{\lambda^{(k+1;p_1)},{\mathfrak
p}^{(k+1;p_1)}} \over z_{p_1}
\Psi^{(k)}_{\lambda^{(k+1;p_1)},{\mathfrak p}^{(k+1;p_1)}}}\cr
=z^{\lambda^{(k+1)}} \left({
\Psi^{(j)}_{\lambda^{(k+1;p_0)},{\mathfrak p}^{(k+1,p_0)}} \over
\Psi^{(k)}_{\lambda^{(k+1;p_0)},{\mathfrak p}^{(k+1;p_0)}}} -
{\Psi^{(j)}_{\lambda^{(k+1;p_1)},{\mathfrak p}^{(k+1;p_1)}} \over
\Psi^{(k)}_{\lambda^{(k+1;p_1)},{\mathfrak
p}^{(k+1;p_1)}}}\right)\cr =
z^{\lambda^{(k+1)}}\Psi^{(j)}_{\lambda^{(k+1)},{\mathfrak
p}^{(k+1)}}.\cr}
$$
\end{proof}

\begin{lemma}\label{lemma(2.5)} For $k\leq j\leq n-1$
the function $\Psi^{(j)}_{\lambda^{(k)},{\mathfrak p}^{(k)}}$ is
homogeneous of weight $j-k+1$ in the variables $\xi_q^{(\ell)}$
($0\leq q\leq n,\,1\leq\ell\leq n$) when $\xi_q^{(\ell)}$ is
assigned the weight $\ell$.  Moreover, for $k\leq j\leq n$, as a
function of the $n+1$ components of the multi-index
$\lambda^{(k)}$, the function
$\Psi^{(j)}_{\lambda^{(k)},{\mathfrak p}^{(k)}}$ is a polynomial
of degree at most $j-k$. In particular,
$\Psi^{(k)}_{\lambda^{(k)},{\mathfrak p}^{(k)}}$ is independent of
the multi-index $\lambda^{(k)}$ and
$$
\Psi^{(j)}_{\lambda^{(k)},{\mathfrak p}^{(k)}} ={1\over
\Psi^{(k-1)}_{\lambda^{(k;p_0)},{\mathfrak p}^{(k;p_0)}}}
\left[\Psi^{(j)}_{\lambda^{(k;p_0)},{\mathfrak p}^{(k;p_0)}}-
\Psi^{(j)}_{\lambda^{(k;p_1)},{\mathfrak p}^{(k;p_1)}}\right].
$$
\end{lemma}

\begin{proof} We prove the Lemma by
induction on $0\leq k\leq n-1$.

\medbreak First we look at the weight of
$\Psi^{(j)}_{\lambda^{(k)},{\mathfrak p}^{(k)}}$ and show that
$\Psi^{(j)}_{\lambda^{(k)},{\mathfrak p}^{(k)}}$ is homogeneous of
weight $j-k+1$ in the variables $\xi_q^{(\ell)}$ ($0\leq q\leq
n,\,1\leq\ell\leq n$).  Again, since the positive integer $\delta$
is fixed once for all, the induction on $0\leq k\leq n-1$ is the
same as the descending induction on the total degree $\delta-k$ of
$\lambda^{(k)}$.  For $k=0$ the conclusion clearly follows from
(\ref{2.3}.1). The derivation of Step $k+1$ from Step
$k$ simply follows from (\ref{2.3}.3), because, for
$\gamma_1=0,1$, with the weight of
$\Psi^{(j)}_{\lambda^{(k+1;p_{\gamma_1})},{\mathfrak
p}^{(k+1;p_{\gamma_1})}}$ being $j-k$ and the weight of
$\Psi^{(k)}_{\lambda^{(k+1;p_{\gamma_1})},{\mathfrak
p}^{(k+1;p_{\gamma_1})}}$ being $1$, the weight
$$\Psi^{(j)}_{\lambda^{(k+1;p_{\gamma_1})},{\mathfrak
p}^{(k+1;p_{\gamma_1})}}\over\Psi^{(k)}_{\lambda^{(k+1;p_{\gamma_1})},{\mathfrak
p}^{(k+1;p_{\gamma_1})}}$$ is equal to $j-(k+1)$.

\medbreak Now we show that $\Psi^{(k)}_{\lambda^{(k)},{\mathfrak
p}^{(k)}}$ is a polynomial of degree no more than $j-k$ in the
$n+1$ components of $\lambda^{(k)}$. Again for $k=0$ the
conclusion clearly follows from (\ref{2.3}.1).  For the
derivation of Step $k+1$ from Step $k$, since the total degree of
$\lambda^{(k+1;p_{\gamma_1})}$ is $\delta-(k+1)-1=\delta-k$, it
follows from Step $k$ that $\Psi^{(j)}_{
\lambda^{(k+1;p_{\gamma_1})},{\mathfrak p}^{(k+1;p_{\gamma_1})} }$
is a polynomial of degree no more than zero in the $n+1$
components of $\lambda^{(k+1;p_{\gamma_1})}$ and is therefore
independent of $\lambda^{(k+1;p_{\gamma_1})}$ for $\gamma_1=0,1$.
Since the binary tree ${\mathfrak p}^{(k+1)}$ is assumed to have
level-wise homogeneous branches (see the paragraph preceding
Lemma(\ref{lemma(2.2)})), it follows that the two truncations
${\mathfrak p}^{(k+1;0)}$ and ${\mathfrak p}^{(k+1;1)}$ are
identical binary trees and
$$
\Psi^{(k)}_{\lambda^{(k+1;p_0)},{\mathfrak p}^{(k+1;p_0)}}=
\Psi^{(k)}_{\lambda^{(k+1;p_1)},{\mathfrak p}^{(k+1;p_1)}}.
$$
Hence by (\ref{2.3}.3),
$$
\displaylines{\Psi^{(j)}_{\lambda^{(k+1)},{\mathfrak p}^{(k+1)}}
={\Psi^{(j)}_{\lambda^{(k+1;p_0)},{\mathfrak p}^{(k+1;p_0)}}\over
\Psi^{(k)}_{\lambda^{(k+1;p_0)},{\mathfrak p}^{(k+1;p_0)}}} -
{\Psi^{(j)}_{\lambda^{(k+1;p_1)},{\mathfrak p}^{(k+1;p_1)}}\over
\Psi^{(k)}_{\lambda^{(k+1;p_1)},{\mathfrak p}^{(k+1;p_1)}}}\cr
={1\over\Psi^{(k)}_{\lambda^{(k+1;p_0)},{\mathfrak
p}^{(k+1;p_0)}}}\left[\Psi^{(j)}_{\lambda^{(k+1;p_0)},{\mathfrak
p}^{(k+1;p_0)}}-\Psi^{(j)}_{\lambda^{(k+1;p_1)},{\mathfrak
p}^{(k+1;p_1)}}\right].\cr }
$$
By induction assumption $ \Psi^{(k)}_{\lambda^{(k+1;p_0)},{\mathfrak
p}^{(k+1;p_0)}}$ is a polynomial in the $n+1$ components of the
multi-index $\lambda^{(k+1;p_0)}$ of degree $j-k$ and $
\Psi^{(k)}_{\lambda^{(k+1;p_1)},{\mathfrak p}^{(k+1;p_1)}}$ is a
polynomial of the $n+1$ components of the multi-index
$\lambda^{(k+1;p_1)}$ of degree at most $j-k$.  Now
$$\lambda^{(k+1;p_0)}-\lambda^{(k+1;p_1)}=e_{p_0}-e_{p_1}$$
which means that the multi-index $\lambda^{(k+1;p_0)}$ is a
translate of the multi-index $\lambda^{(k+1;p_1)}$ by the
multi-index $e_{p_0}-e_{p_1}$ and is independent of the $n+1$
components of the multi-index $\lambda^{(k+1)}$. Thus the difference
$$\Psi^{(j)}_{\lambda^{(k+1;p_0)},{\mathfrak
p}^{(k+1;p_0)}}-\Psi^{(j)}_{\lambda^{(k+1;p_1)},{\mathfrak
p}^{(k+1;p_1)}}$$ is a polynomial of the $n+1$ components of the
multi-index $\lambda^{(k+1)}$ of degree at most $j-k-1$. The last
statement follows from Lemma~\ref{lemma(2.4)}.\end{proof}

\begin{lemma}\label{lemma(2.6)} If $\lambda^{(k)}$ is a
multi-index of $n+1$ components with total degree $\delta-k$ and
${\mathfrak p}^{(k)}$ is any binary tree of order $k$ which has
level-wise homogeneous branching, then
$$
\Psi^{(k)}_{\lambda^{(k)},{\mathfrak
p}^{(k)}}=k\left(\xi^{(1)}_{p_0}-\xi^{(1)}_{p_1}\right),
$$
where the nodes of ${\mathfrak p}^{(k)}$ are denoted by
$p_{\gamma_1,\cdots,\gamma_j}$ with $1\leq j\leq k$ and each
$\gamma_\ell$ taking on the value $0$ or $1$ for $1\leq\ell\leq
j$.  As a consequence,
$$
\Psi^{(j)}_{\lambda^{(k)},{\mathfrak p}^{(k)}} ={1\over
k\left(\xi^{(1)}_{p_0}-\xi^{(1)}_{p_1}\right)}
\left[\Psi^{(j)}_{\lambda^{(k;p_0)},{\mathfrak p}^{(k;p_0)}}-
\Psi^{(j)}_{\lambda^{(k;p_1)},{\mathfrak p}^{(k;p_1)}}\right].
$$
\end{lemma}

\begin{proof}  First we make the following
simple observation.  Let $G\left(\nu_0,\cdots,\nu_n\right)$ be a
polynomial in $\nu_0,\cdots,\nu_n$ of degree no more than $M$. For
any $p\not=q$ define
$$
\displaylines{
\left(\Delta_{p,q}G\right)\left(\nu_0,\cdots,\nu_n\right)
=G\left(\nu_0,\cdots,\nu_{p-1},\nu_p+1,\nu_{p+1},\cdots,\nu_n\right)\cr
-G\left(\nu_0,\cdots,\nu_{q-1},\nu_q+1,\nu_{q+1},\cdots,\nu_n\right).\cr
}
$$
Then $\left(\Delta_{p,q}G\right)\left(\nu_0,\cdots,\nu_n\right)$ is
a polynomial in $\nu_0,\cdots,\nu_n$ of degree no more than $M-1$,
because we can write
$$
\displaylines{
\left(\Delta_{p,q}G\right)\left(\nu_0,\cdots,\nu_n\right)\cr
=\left[G\left(\nu_0,\cdots,\nu_{p-1},\nu_p+1,\nu_{p+1},\cdots,\nu_n\right)-
G\left(\nu_0,\cdots,\nu_n\right)\right]\cr
-\left[G\left(\nu_0,\cdots,\nu_{q-1},\nu_q+1,\nu_{q+1},\cdots,\nu_n\right)-
G\left(\nu_0,\cdots,\nu_n\right)\right]\cr }
$$
and clearly each of the two terms
$$
G\left(\nu_0,\cdots,\nu_{p-1},\nu_p+1,\nu_{p+1},\cdots,\nu_n\right)-
G\left(\nu_0,\cdots,\nu_n\right)
$$
and
$$
G\left(\nu_0,\cdots,\nu_{q-1},\nu_q+1,\nu_{q+1},\cdots,\nu_n\right)-
G\left(\nu_0,\cdots,\nu_n\right)
$$
is a a polynomial in $\nu_0,\cdots,\nu_n$ of degree no more than
$M-1$. As a consequence,

\medbreak\noindent(\ref{lemma(2.6)}.1) if
$$
r_1\not=s_1,\cdots,r_{M+1}\not=s_{M+1},
$$
then $$\Delta_{r_1,s_1}\cdots\Delta_{r_M,s_M}G$$ is of degree zero
in $\nu_0,\cdots,\nu_n$ and $
\Delta_{r_1,s_1}\cdots\Delta_{r_{N+1},s_{N+1}}G$ is identically zero
for any polynomial $G\left(\nu_0,\cdots,\nu_n\right)$ in
$\nu_0,\cdots,\nu_n$ of degree no more than $N$.

\addtocounter{equation}{1}

\medbreak Let $r_j=p_{\gamma_1,\cdots,\gamma_{j-1},0}$ and
$s_j=p_{\gamma_1,\cdots,\gamma_{j-1},1}$.  Since the binary tree
${\mathfrak p}^{(k)}$ of order $k$ has level-wise homogeneous
branches, the values of $r_j=p_{\gamma_1,\cdots,\gamma_{j-1},0}$ and
$s_j=p_{\gamma_1,\cdots,\gamma_{j-1},1}$ are independent of the
choices of the values $0$ or $1$ for $\gamma_1,\cdots,\gamma_{j-1}$.
Let ${\mathfrak p}^{(k-j)}={\mathfrak
p}^{(k;\gamma_1,\cdots,\gamma_{j-1})}$.  Again we know that
${\mathfrak p}^{(k-j)}$ is independent of the choices of the values
$0$ or $1$ for $\gamma_1,\cdots,\gamma_{j-1}$ because the binary
tree ${\mathfrak p}^{(k)}$ of order $k$ has level-wise homogeneous
branches. By the last statement of Lemma~\ref{lemma(2.5)}, we have
$$\Psi^{(j)}_{\lambda^{(k)},{\mathfrak p}^{(k)}}
={1\over \Psi^{(k-1)}_{\lambda^{(k;p_0)},{\mathfrak p}^{(k-1)}}}
\left[\Psi^{(j)}_{\lambda^{(k;p_0)},{\mathfrak p}^{(k;p_0)}}-
\Psi^{(j)}_{\lambda^{(k;p_1)},{\mathfrak
p}^{(k;p_1)}}\right]\leqno{(\ref{lemma(2.6)}.2)}
$$
$$
={1\over \Psi^{(k-1)}_{{\mathfrak p}^{(k-1)}}}
\left[\Psi^{(j)}_{\lambda^{(k;p_0)},{\mathfrak p}^{(k-1)}}-
\Psi^{(j)}_{\lambda^{(k;p_1)},{\mathfrak
p}^{(k-1)}}\right]
$$
for any multi-index
$\lambda^{(k)}$ of $n+1$ components and total degree $\delta-k$.
Here, because of the independence of
$\Psi^{(k-1)}_{\lambda^{(k;p_0)},{\mathfrak p}^{(k;p_0)}}$ of the
$n+1$ components of the multi-index $\lambda^{(k;p_0)}$ by
Lemma~\ref{lemma(2.5)}, we drop $\lambda^{(k;p_0)}$ from the
subscript of $\Psi^{(k-1)}_{\lambda^{(k;p_0)},{\mathfrak
p}^{(k;p_0)}}$ and simply write
$\Psi^{(k-1)}_{\lambda^{(k;p_0)},{\mathfrak p}^{(k;p_0)}}$ as
$\Psi^{(k-1)}_{{\mathfrak p}^{(k;p_0)}}$. From
$$
\Psi^{(j)}_{\lambda^{(0)},{\mathfrak
p}^{(0)}}=\Phi^{(j)}_{\lambda^{(0)}}
$$
in the formula (\ref{2.3}.1) and from
(\ref{lemma(2.6)}.2) it follows that
$$
\Psi^{(j)}_{\lambda^{\left(k;\gamma_1,\cdots,\gamma_{k-1}\right)},
{\mathfrak p}^{(1)}} =\left[{1\over \Psi^{(0)}_{{\mathfrak
p}^{(0)}}}\, \Delta_{
r_k,s_k}\Phi^{(j)}_{\nu}\right]_{\nu=\lambda^{\left(k;\gamma_1,\cdots,\gamma_{k-1}\right)}}
$$
for any $\lambda^{(k)}$ of $n+1$ components and total degree
$\delta-k$, because the total degree of the multi-index
$\lambda^{\left(k;\gamma_1,\cdots,\gamma_{k-1}\right)}$ is
$\delta-1$, which corresponds to the situation of $k=1$ in
(\ref{lemma(2.6)}.2).  Inductively for $1\leq\ell\leq k$ we are
going to verify that
$$
\Psi^{(j)}_{\lambda^{\left(k;\gamma_1,\cdots,\gamma_{k-\ell}\right)},
{\mathfrak p}^{(\ell)}}\leqno{(\ref{lemma(2.6)}.3)}
$$
$$=\left[{1\over \Psi^{(0)}_{{\mathfrak
p}^{(0)}}\cdots\Psi^{(\ell-1)}_{{\mathfrak p}^{(\ell-1)}}}\,
\Delta_{ r_{k-\ell+1},s_{k-\ell+1}}\cdots\Delta_{
r_k,s_k}\Phi^{(j)}_{\nu}\right]
_{\nu=\lambda^{\left(k;\gamma_1,\cdots,\gamma_{k-\ell}\right)}}
$$
for any $\lambda^{(k)}$ of $n+1$ components and total degree
$\delta-k$.  To go from Step $\ell$ to Step $\ell+1$, by
(\ref{lemma(2.6)}.2) we have
$$
\displaylines{
\Psi^{(j)}_{\lambda^{\left(k;\gamma_1,\cdots,\gamma_{k-\ell-1}\right)},
{\mathfrak p}^{(\ell+1)}}\cr={1\over\Psi^{(\ell)}_{{\mathfrak
p}^{(\ell)}}}\left\{
\Psi^{(j)}_{\lambda^{\left(k;\gamma_1,\cdots,\gamma_{k-\ell-1},0\right)},
{\mathfrak p}^{(\ell)}}
-\Psi^{(j)}_{\lambda^{\left(k;\gamma_1,\cdots,\gamma_{k-\ell-1},1\right)},
{\mathfrak p}^{(\ell)}}\right\}\cr
={1\over\Psi^{(\ell)}_{{\mathfrak p}^{(\ell)}}}\left\{
\Psi^{(j)}_{\lambda^{\left(k;\gamma_1,\cdots,\gamma_{k-\ell-1},0\right)},
{\mathfrak p}^{(\ell)}}
-\Psi^{(j)}_{\lambda^{\left(k;\gamma_1,\cdots,\gamma_{k-\ell-1},1\right)},
{\mathfrak p}^{(\ell)}}\right\}\cr
={1\over\Psi^{(\ell)}_{{\mathfrak p}^{(\ell)}}} \Bigg\{
 \left[{1\over \Psi^{(0)}_{{\mathfrak
p}^{(0)}}\cdots\Psi^{(\ell-1)}_{{\mathfrak p}^{(\ell-1)}}}\,
\Delta_{ r_{k-\ell+1},s_{k-\ell+1}}\cdots\Delta_{
r_k,s_k}\Phi^{(j)}_{\nu}\right]
_{\nu=\lambda^{\left(k;\gamma_1,\cdots,\gamma_{k-\ell-1},0\right)}}\cr
- \left[{1\over \Psi^{(0)}_{{\mathfrak
p}^{(0)}}\cdots\Psi^{(\ell-1)}_{{\mathfrak p}^{(\ell-1)}}}\,
\Delta_{ r_{k-\ell+1},s_{k-\ell+1}}\cdots\Delta_{
r_k,s_k}\Phi^{(j)}_{\nu}\right]
_{\nu=\lambda^{\left(k;\gamma_1,\cdots,\gamma_{k-\ell-1},0\right)}}
\Bigg\}\cr =\left[{1\over \Psi^{(0)}_{{\mathfrak
p}^{(0)}}\cdots\Psi^{(\ell)}_{{\mathfrak p}^{(\ell)}}}\, \Delta_{
r_{k-\ell},s_{k-\ell}}\cdots\Delta_{
r_k,s_k}\Phi^{(j)}_{\nu}\right]
_{\nu=\lambda^{\left(k;\gamma_1,\cdots,\gamma_{k-\ell}\right)}}.\cr}
$$
This finishes the verification of (\ref{lemma(2.6)}.3) by
induction. Setting $\ell=j=k$ in (\ref{lemma(2.6)}.3) yields
$$
\Psi^{(k)}_{{\mathfrak p}^{(k)}}= \left[{1\over
\Psi^{(0)}_{{\mathfrak p}^{(0)}}\cdots\Psi^{(k-1)}_{{\mathfrak
p}^{(k-1)}}}\, \Delta_{ r_1,s_1}\cdots\Delta_{
r_k,s_k}\Phi^{(k)}_{\nu}\right] _{\nu=\lambda^{(k)}}
$$
for any $\lambda^{(k)}$ of $n+1$ components and total degree
$\delta-k$.  We rewrite it as
$$\Delta_{r_1,s_1}\cdots\Delta_{r_k,s_k}\Phi^{(k)}_\nu
=\Psi^{(1)}_{\lambda^{(1)},{\mathfrak p}^{(1)}}
\Psi^{(2)}_{\lambda^{(2)},{\mathfrak p}^{(2)}}\cdots
\Psi^{(k)}_{\lambda^{(k)},{\mathfrak
p}^{(k)}}\leqno{(\ref{lemma(2.6)}.4)}
$$
when
$\Phi^{(k)}_\nu$ is regarded as a polynomial in the $n+1$
components of $\nu$. Note that by (\ref{lemma(2.6)}.1) the
left-hand side
$$\Delta_{r_1,s_1}\cdots\Delta_{r_k,s_k}\Phi^{(k)}_\nu$$ is
independent of the value of $\nu$ because the degree of
$\Phi^{(k)}_\nu$ as a polynomial in $\nu$ is no more than $k$.

\medbreak Since by Lemma A as a polynomial in $\nu_1,\cdots,\nu_n$
the degree of
$$\Phi_\nu^{(k)}-\left(\sum_{\ell=0}^n\nu_\ell\xi^{(1)}_\ell\right)^k$$
is at most $k-1$, it follows from (\ref{lemma(2.6)}.1) that
$$
\Delta_{r_1,s_1}\cdots\Delta_{r_k,s_k}
\left(\Phi_\nu^{(k)}-\left(\sum_{\ell=0}^n\nu_\ell\xi^{(1)}_\ell\right)^k\right)
=0
$$
and
$$
\Delta_{r_1,s_1}\cdots\Delta_{r_k,s_k}
\left(\sum_{\ell=0}^n\nu_\ell\xi^{(1)}_\ell\right)^k=\Psi^{(1)}_{\lambda^{(1)},{\mathfrak
p}^{(1)}} \Psi^{(2)}_{\lambda^{(2)},{\mathfrak p}^{(2)}}\cdots
\Psi^{(k)}_{\lambda^{(k)},{\mathfrak p}^{(k)}}.$$ We have
$$
\displaylines{\Delta_{r_k,s_k}
\left(\sum_{\ell=0}^n\nu_\ell\xi^{(1)}_\ell\right)^k\cr
=\left(\sum_{\ell=0}^n\nu_\ell\xi^{(1)}_\ell+\xi^{(1)}_{r_k}\right)^k
-\left(\sum_{\ell=0}^n\nu_\ell\xi^{(1)}_\ell+\xi^{(1)}_{s_k}\right)^k\cr
=\left(\xi^{(1)}_{r_k}-\xi^{(1)}_{s_k}\right)\left\{\sum_{j=0}^{k-1}
\left(\sum_{\ell=0}^n\nu_\ell\xi^{(1)}_\ell+\xi^{(1)}_{r_k}\right)^j
\left(\sum_{\ell=0}^n\nu_\ell\xi^{(1)}_\ell+\xi^{(1)}_{s_k}\right)^{k-1-j}\right\}.\cr
}
$$
Now for $j\geq 1$,
$$
\displaylines{
\left(\sum_{\ell=0}^n\nu_\ell\xi^{(1)}_\ell+\xi^{(1)}_{r_k}\right)^j
\left(\sum_{\ell=0}^n\nu_\ell\xi^{(1)}_\ell+\xi^{(1)}_{s_k}\right)^{k-1-j}
-\left(\sum_{\ell=0}^n\nu_\ell\xi^{(1)}_\ell+\xi^{(1)}_{s_k}\right)^{k-1}=\cr
\left(\sum_{\ell=0}^n\nu_\ell\xi^{(1)}_\ell+\xi^{(1)}_{s_k}\right)^{k-1-j}
\left\{\left(\sum_{\ell=0}^n\nu_\ell\xi^{(1)}_\ell+\xi^{(1)}_{r_k}\right)^j
-\left(\sum_{\ell=0}^n\nu_\ell\xi^{(1)}_\ell+\xi^{(1)}_{r_k}\right)^j\right\}\cr
=\left(\sum_{\ell=0}^n\nu_\ell\xi^{(1)}_\ell+\xi^{(1)}_{s_k}\right)^{k-1-j}
\left(\xi^{(1)}_{r_k}-\xi^{(1)}_{s_k}\right)\times\cr\times
\left\{
\sum_{\ell=0}^{j-1}\left(\sum_{\ell=0}^n\nu_\ell\xi^{(1)}_\ell+\xi^{(1)}_{r_k}\right)^\ell
\left(\sum_{\ell=0}^n\nu_\ell\xi^{(1)}_\ell+\xi^{(1)}_{s_k}\right)^{j-1-\ell}\right\}\cr
}
$$
is a polynomial of degree no more than $k-2$ in the $n+1$
components of $\nu$.  Since
$$
\displaylines{
\left(\sum_{\ell=0}^n\nu_\ell\xi^{(1)}_\ell+\xi^{(1)}_{s_k}\right)^{k-1}-
\left(\sum_{\ell=0}^n\nu_\ell\xi^{(1)}_\ell\right)^{k-1}\cr
=\xi^{(1)}_{s_k}\left\{\sum_{\ell=0}^{k-2}
\left(\sum_{\ell=0}^n\nu_\ell\xi^{(1)}_\ell+\xi^{(1)}_{s_k}\right)^\ell
\left(\sum_{\ell=0}^n\nu_\ell\xi^{(1)}_\ell\right)^{k-2-\ell}\right\}\cr
}
$$
is a polynomial of degree no more than $k-2$ in the $n+1$
components of $\nu$, it follows that
$$
\Delta_{r_k,s_k}
\left(\sum_{\ell=0}^n\nu_\ell\xi^{(1)}_\ell\right)^k-
k\left(\sum_{\ell=0}^n\nu_\ell\xi^{(1)}_\ell\right)^{k-1}
$$
a polynomial of degree no more than $k-2$ in the $n+1$ components
of $\nu$ and
$$
\Delta_{r_1,s_1}\cdots\Delta_{r_{k-1},s_{k-1}} \left\{
\Delta_{r_k,s_k}
\left(\sum_{\ell=0}^n\nu_\ell\xi^{(1)}_\ell\right)^k-
k\left(\sum_{\ell=0}^n\nu_\ell\xi^{(1)}_\ell\right)^{k-1}\right\}
$$
is zero. By induction on $k$, we conclude that
$$
\Delta_{r_1,s_1}\cdots\Delta_{r_k,s_k}
\left(\sum_{\ell=0}^n\nu_\ell\xi^{(1)}_\ell\right)^k
=k!\prod_{\ell=1}^k\left(\xi^{(1)}_{r_\ell}-\xi^{(1)}_{s_\ell}\right).
$$
and
$$
\Delta_{r_1,s_1}\cdots\Delta_{r_k,s_k}\Phi^{(k)}_\nu
=k!\prod_{\ell=1}^k\left(\xi^{(1)}_{r_\ell}-\xi^{(1)}_{s_\ell}\right).
$$
It follows from (\ref{lemma(2.6)}.4) that
$$
\Psi^{(1)}_{\lambda^{(1)},{\mathfrak p}^{(1)}}
\Psi^{(2)}_{\lambda^{(2)},{\mathfrak p}^{(2)}}\cdots
\Psi^{(k)}_{\lambda^{(k)},{\mathfrak p}^{(k)}}
=k!\prod_{\ell=1}^k\left(\xi^{(1)}_{r_\ell}-\xi^{(1)}_{s_\ell}\right).\leqno{(\ref{lemma(2.6)}.5)}
$$
Finally
(\ref{lemma(2.6)}.5)
yields
$$
\Psi^{(k)}_{\lambda^{(k)},{\mathfrak
p}^{(k)}}=k\left(\xi^{(1)}_{p_0}-\xi^{(1)}_{p_1}\right)
$$
by induction on $k$ for any multi-index $\lambda^{(k)}$ of $n+1$
components with total degree $\delta-k$ and for any binary tree
${\mathfrak p}^{(k)}$ of order $k$ which has level-wise
homogeneous branching.\end{proof}

\begin{remark}\label{remark(2.7)} The reason for
explicitly computing the function
$\Psi^{(k)}_{\lambda^{(k)},{\mathfrak p}^{(k)}}$ is to determine
the pole set of the vector field
$\Theta^{(k)}_{\lambda^{(k)},{\mathfrak p}^{(k)}}$.
\end{remark}

\begin{lemma}\label{lemma(2.8)} (a) The function
$$
\Psi^{(j)}_{\lambda^{(k)},{\mathfrak
p}^{(k)}}\prod_{\ell=1}^{k-1}\left(\xi^{(1)}_{r_\ell}-\xi^{(1)}_{s_\ell}\right)
$$
is of homogeneous weight $j$ in $\xi^{(\ell)}_t$ ($1\leq\ell\leq
j$, $0\leq t\leq n$) and is independent of
$\left(z_0,\cdots,z_n\right)$ and is a polynomial in the $n+1$
components of $\lambda^{(k)}$ of degree $\leq j$.

\medbreak\noindent (b) The vector field
$$
\Theta^{(j)}_{\lambda^{(k)},{\mathfrak
p}^{(k)}}\left(\prod_{\ell=1}^k\left(z_{r_\ell}z_{s_\ell}\right)\right)
\prod_{\ell=1}^{k-1}
\left(\xi^{(1)}_{r_\ell}-\xi^{(1)}_{s_\ell}\right)
$$
is a polynomial in $z_0,\cdots,z_n$ of degree $\leq k$ and is
independent of $\xi^{(\ell)}_t$ ($1\leq\ell\leq n-1$, $0\leq t\leq
n$). Moreover, the dependence of
$\Theta^{(j)}_{\lambda^{(k)},{\mathfrak p}^{(k)}}$ on
$\lambda^{(k)}$ is only through the partial differentiation with
respect to $\alpha_\mu$ with $\mu$ depending on $\lambda^{(k)}$.

\medbreak\noindent (c) The vector field
$\Theta^{(k)}_{\lambda^{(k)},{\mathfrak p}^{(k)}}$ is equal to
$$
{z^{\lambda^{(k)}}\over \prod_{\ell=1}^{k-1}
\left(\xi^{(1)}_{r_\ell}-\xi^{(1)}_{s_\ell}\right)}\left[\Delta_{r_1,s_1}\cdots
\Delta_{r_k,s_k}\left(z^{-\nu}{\partial\over\partial\alpha_\nu}
\right)\right]_{\nu=\lambda^{(k)}},
$$
where $\Delta_{r_1,s_1}\cdots \Delta_{r_k,s_k}$ is applied to
$$z^{-\nu}{\partial\over\partial\alpha_\nu}$$ as a function of
$\nu$.
\end{lemma}

\bigbreak

\begin{nodesignation}{\it Generation of Vector Fields in
the Parameter
Direction.}\label{nodesignation(2.9)}\end{nodesignation} We now
look at the special case of Lemma~\ref{lemma(2.8)}(c) with $k=n$,
then
$$
\tilde\Theta^{(n)}_{\lambda^{(n)},{\mathfrak
p}^{(n)}}=z^{\lambda^{(n)}+\sum_{\ell=1}^n\left(e_{r_\ell}+e_{s_\ell}\right)}
\left[\Delta_{r_1,s_1}\cdots
\Delta_{r_n,s_n}\left(z^{-\nu}{\partial\over\partial\alpha_\nu}
\right)\right]_{\nu=\lambda^{(n)}}
$$
satisfies
$$
\tilde\Theta^{(n)}_{\lambda^{(n)},{\mathfrak
p}^{(n)}}\left(d^jf\right)=0
$$
for $0\leq j\leq n-1$.  We fix a point $y$ in $X_{\alpha}$,
where $\alpha=\left\{\alpha_\nu\right\}_{|\nu|=\delta}$.  We can
choose homogeneous coordinates in ${\mathbb P}_n$ so that
$$\left(z_0,z_1,\cdots,z_n\right)(y)=(1,0,\cdots,0).$$
We choose also $s_1=\cdots=s_n=0$ and $r_j\not=0$ for $1\leq j\leq
n$.  Then at $y$ we end up with
$$
\tilde\Theta^{(n)}_{\lambda^{(n)},{\mathfrak
p}^{(n)}}={\partial\over\partial\alpha_{\lambda^{(n)}+\sum_{\ell=1}^n
e_{r_\ell}}}.
$$
For the choice of $\lambda^{(n)}$ we can choose any multi-index of
total degree $\delta-n$.  When we worry about the generation of
the vector fields by global sections, for differentiations in the
direction of the parameters $\alpha=\{\alpha_\nu\}_{|\nu|=\delta}$
at the origin, we can capture in inhomogeneous coordinates the
differentiation with respect to all coefficients for monomials of
degree at least $n$, because we must include $\sum_{\ell=1}^n
e_{r_\ell}$ with $r_j\not=0$ for $1\leq j\leq n$ in $\nu$ which is
equal to $\lambda^{(n)}+\sum_{\ell=1}^n e_{r_\ell}$.

\bigbreak

\begin{nodesignation}{\it Example of Vector Fields on Jet
Spaces of Low
Order.}\label{nodesignation(2.9b)}\end{nodesignation} Let
$f=\sum_\nu\alpha_\nu z^\nu$. We introduce
$$
\displaylines{ \xi_j=dz_j,\cr \xi^{(1)}_j={dz_j\over
z_j}={\xi_j\over z_j}.\cr }
$$
Then
$$
df=\sum_\nu\alpha_\nu\left(\sum_j\nu_j\xi^{(1)}_j\right)z^\nu.
$$

\begin{proposition}\label{proposition(2.10)} Let $0\leq
p\not=q\leq n$ and $0\leq r\not=s\leq n$. Let $\mu$ be a
multi-index of total weight $\delta-2$. Then
$$
\displaylines{\left\{{1\over z_r}\left(\left({1\over
z_p}{\partial\over\partial\alpha_{\mu+e_r+e_p}} -{1\over
z_q}{\partial\over\partial\alpha_{\mu+e_r+e_q}}
\right)\right)\right.\cr -\left.{1\over z_r}\left(\left({1\over
z_p}{\partial\over\partial\alpha_{\mu+e_r+e_p}} -{1\over
z_q}{\partial\over\partial\alpha_{\mu+e_r+e_q}}
\right)\right)\right\}(d^jf)=0\cr }
$$
for $j=0,1$.
\end{proposition}

\begin{proof} We have
$$
\left({1\over z_p}{\partial\over\partial\alpha_{\lambda+e_p}}
-{1\over z_q}{\partial\over\partial\alpha_{\lambda+e_q}}
\right)(df)=\left(\xi^{(1)}_q-\xi^{(1)}_p \right)z^\lambda
$$
for any $\lambda$ with $|\lambda|=\delta-1$.  Choose $\mu$ with
$|\mu|=\delta-2$.  Apply the above equation to $\lambda=\mu+e_r$
and get
$$
\left({1\over z_p}{\partial\over\partial\alpha_{\mu+e_r+e_p}}
-{1\over z_q}{\partial\over\partial\alpha_{\mu+e_r+e_q}}
\right)(df)=\left(\xi^{(1)}_q-\xi^{(1)}_p \right)z^{\mu+e_r}
$$
and
$$
{1\over z_r}\left(\left({1\over
z_p}{\partial\over\partial\alpha_{\mu+e_r+e_p}} -{1\over
z_q}{\partial\over\partial\alpha_{\mu+e_r+e_q}}
\right)\right)(df)=\left(\xi^{(1)}_q-\xi^{(1)}_p \right)z^\mu.
$$
Since the right-hand side is independent of $r$, we can replace
$r$ by $s$ and take the difference to get
$$
\displaylines{\left\{{1\over z_r}\left(\left({1\over
z_p}{\partial\over\partial\alpha_{\mu+e_r+e_p}} -{1\over
z_q}{\partial\over\partial\alpha_{\mu+e_r+e_q}}
\right)\right)\right.\cr -\left.{1\over z_r}\left(\left({1\over
z_p}{\partial\over\partial\alpha_{\mu+e_r+e_p}} -{1\over
z_q}{\partial\over\partial\alpha_{\mu+e_r+e_q}}
\right)\right)\right\}(df)=0.\cr }
$$
\end{proof}

\begin{remark}\label{remark(2.11)} We can rewrite the
vector field
$$
\displaylines{\left\{{1\over z_r}\left(\left({1\over
z_p}{\partial\over\partial\alpha_{\mu+e_r+e_p}} -{1\over
z_q}{\partial\over\partial\alpha_{\mu+e_r+e_q}}
\right)\right)\right.\cr -\left.{1\over z_r}\left(\left({1\over
z_p}{\partial\over\partial\alpha_{\mu+e_r+e_p}} -{1\over
z_q}{\partial\over\partial\alpha_{\mu+e_r+e_q}}
\right)\right)\right\}\cr }
$$
as
$$
z^\mu\left[\Delta_{r,s}\Delta_{p,q}\left({1\over
z^\nu}{\partial\over
\partial\alpha_\nu}\right)\right]_{\nu=\mu},
$$
where
$$
\Delta_{r,s}F(\nu)=F\left(\nu+e_r\right)-F\left(\nu+e_s\right).
$$
\end{remark}

\bigbreak To illustrate the situation of vector fields on jet
spaces of low order, we do the case of the next order.
$$
\displaylines{d^2f=\sum_\nu\alpha_\nu\left[
\left(\sum_j\nu_j\xi^{(1)}_j\right)^2+\sum_\nu\alpha^{(2)}_j\right]z^\nu,\cr
{1\over z^\nu}{\partial\over\partial\alpha_\nu}d^2f=
\left(\sum_j\nu_j\xi^{(1)}_j\right)^2+\sum_\nu\xi^{(2)}_j,\cr
\Delta_{p,q}\left({1\over
z^\nu}{\partial\over\partial\alpha_\nu}d^2f\right)=
\left(\sum_j\nu_j\xi^{(1)}_j+\xi^{(1)}_p\right)^2\cr
-\left(\sum_j\nu_j\xi^{(1)}_j+\xi^{(1)}_q\right)^2
+\left(\xi^{(2)}_p-\xi^{(2)}_q\right)\cr =
\left(\xi^{(1)}_p-\xi^{(1)}_q\right)\left(2\sum_j\nu_j\xi^{(1)}_j+\xi^{(1)}_p+\xi^{(2)}_q\right)
+\left(\xi^{(2)}_p-\xi^{(2)}_q\right),\cr
\Delta_{r,s}\Delta_{p,q}\left({1\over
z^\nu}{\partial\over\partial\alpha_\nu}d^2f\right)=2
\left(\xi^{(1)}_r-\xi^{(1)}_s\right)\left(\xi^{(1)}_p-\xi^{(1)}_q\right),\cr
\Delta_{u,v}\Delta_{r,s}\Delta_{p,q}\left({1\over
z^\nu}{\partial\over\partial\alpha_\nu}d^2f\right)=0.\cr }
$$
Thus for any multi-index $\mu$ of total degree $\delta-3$, the
vector field
$$
\left[\Delta_{u,v}\Delta_{r,s}\Delta_{p,q}\left({1\over
z^\nu}{\partial\over\partial\alpha_\nu}\right)\right]_{\nu=\mu}
$$
annihilates $d^jf$ for $j=0,1,2$, because
$$
\left[{\partial\over\partial\alpha_\nu}(d^jf)\right]_{\nu=\lambda}=
\left[{\partial\over\partial\alpha_\nu}\right]_{\nu=\lambda}(d^jf).
$$

\bigbreak We can now formulate the case of higher-order jets.

\begin{proposition}\label{proposition(2.12)}  Let $0\leq
r_\ell\not=s_\ell\leq n$ for $1\leq\ell\leq k$. Let $\mu$ be a
multi-index of total weight $\delta-k$. Let
$\Theta_{\mu;r_1,\cdots,r_k;s_1,\cdots,s_k}$ denote the vector
field
$$
z^\mu\left[\Delta_{r_1,s_1}\cdots\Delta_{r_k,s_k}\left({1\over
z^\nu}{\partial\over
\partial\alpha_\nu}\right)\right]_{\nu=\mu}.
$$
Then
$\Theta_{\mu;r_1,\cdots,r_k;s_1,\cdots,s_k}\left(d^jf\right)=0$
for $0\leq j\leq k-1$.
\end{proposition}

\bigbreak In the above Proposition the vector field
$\Theta_{\mu;r_1,\cdots,r_k;s_1,\cdots,s_k}$ is a linear
combination of the partial differentiation operators
$$
{\partial\over\partial\alpha_{\mu+e_{r_{i_1}}+\cdots+e_{r_{i_p}}+
e_{s_{j_1}}+\cdots+e_{s_{j_{k-p}}}}}
$$
for $0\leq j\leq k$.  The process of generating such vector fields
is not independent of coordinate transformations from the general
linear group $GL\left(n+1,{\mathbb C}\right)$.  Suppose we have the
coordinate transformation
$$
z_j=\sum_{\ell=0}^n a_{j\ell}w_\ell\quad(0\leq j\leq\ell)
$$
from the element ${\bf a}=\left(a_{j\ell}\right)_{0\leq j,\ell\leq
n}$ of the general linear group $GL\left(n+1,{\mathbb C}\right)$. Then
$$
z^\nu=z_0^{\nu_0}\cdots z_n^{\nu_n}=\left(\sum_{\ell=0}^n
a_{0\ell}w_\ell\right)^{\nu_0}\cdots \left(\sum_{\ell=0}^n
a_{n\ell}w_\ell\right)^{\nu_n}
=\sum_{|\mu|=\delta}A_{\nu,\mu}w^\mu.
$$
Let $\left(B_{\mu,\nu}\right)_{|\mu|=|\nu|=\delta}$ be the inverse
matrix of the matrix
$\left(A_{\nu,\mu}\right)_{|\mu|=|\nu|=\delta}$.  Write
$$
f=\sum_{|\nu|=\delta}\alpha_\nu z^\nu=\sum_{|\mu|=\delta}\beta_\mu
w^\mu.
$$
Then
$$
\displaylines{\beta_\mu=\sum_{|\nu|=\delta}\alpha_\nu
A_{\nu,\mu},\cr \alpha_\mu=\sum_{|\mu|=\delta}\beta_\mu
B_{\mu,\nu}.\cr }
$$
When the generation of the vector field $\Theta$ in the coordinate
system $(z_0,\cdots,z_n)$ gives
$$
\Theta_z=\sum_\nu g_\nu(z){\partial\over\partial\alpha_\nu},
$$
the procedure applied to the coordinate system $(w_0,\cdots,w_n)$
gives
$$
\Theta_w=\sum_\mu g_\mu(w){\partial\over\partial\beta_\mu}.
$$
When we transform back to the coordinate system
$(z_0,\cdots,z_n)$, we get
$$
\Theta_{{\bf a};\mu;r_1,\cdots,r_n;s_1,\cdots,s_n}=\sum_{\mu,\nu}
g_\mu(w(z)){\partial\alpha_\nu\over\partial\beta_\mu}{\partial\over\partial\alpha_\nu}.
$$
Now We would like to show that when $k=n$, the dimension of the
quotient space
$$
\left(\bigoplus_{|\nu|=\delta}{\mathbb C}{\partial\over\partial\alpha_\nu}\right)\Bigg/\left(\sum_{{{{\bf
a}\in GL(n+1,{\mathbb C});|\mu|=\delta-n;}\atop{
r_1,\cdots,r_n;s_1,\cdots,s_n}}}{\mathbb C}\Theta_{{\bf a};\mu;
r_1,\cdots,r_n;s_1,\cdots,s_n}\right)
$$
is no more than $n$ over ${\mathbb C}$.  For this we need only show
that modulo the linear space generated by all such $\Theta_{{\bf
a};\mu;r_1,\cdots,r_n;s_1,\cdots,s_n}$, every generator
${\partial\over\partial\alpha_\nu}$ can be expressed as a linear
combination of $n$ fixed
$${\partial\over\partial\alpha_{\nu^{[\ell]}}}\quad(1\leq\ell\leq
n),$$ where $\nu^{[\ell]}$ is a multi-index of total weight
$\delta$.  We use the linear transformations defined by ${\bf
a}\in GL\left(n+1,{\mathbb C}\right)$ simply to make sure that, for
any given point, we are free to do the checking in an appropriate
coordinate system which depends on the point.

\medbreak For the convenience of bookkeeping we let $M$ be an
integer $>\delta$ and introduce a new weight $\|\nu\|_M$ for any
multi-index $\nu$ of total degree $\delta$ which is defined as
follows.
$$
\|\nu\|_M=\sum_{\ell=0}^n\nu_\ell M^\ell.
$$
We single out the $n$ multi-index $\nu$ of total degree $\delta$
which has the $n$ lowest weight $\|\nu\|_M$ possible, namely,
$$
\delta-\ell+(n-1-\ell)M\quad(0\leq\ell\leq n-1).
$$
These $n-1$ multi-indices are
$$
\nu^{[\ell]}=\left(\delta-\ell,n-1-\ell,0,\cdots,0\right)\quad(0\leq\ell\leq
n-1).
$$
Fix a point $P_0$ in the space $J^{\rm{\scriptstyle
vert}}_{n-1}\left({\mathcal X}\right)$ of vertical $(n-1)$-jets.
Choose a coordinate system $\left(z_0,\cdots,z_n\right)$ so that
all the coefficients of
$$
{\partial\over\partial\alpha_{\mu+e_{r_{i_1}}+\cdots+e_{r_{i_p}}+
e_{s_{j_1}}+\cdots+e_{s_{j_{k-p}}}}}
$$
occurring in
$$
\Theta_{\mu;r_1,\cdots,r_n;s_1,\cdots,s_n}=
z^\mu\left[\Delta_{r_1,s_1}\cdots\Delta_{r_k,s_k}\left({1\over
z^\nu}{\partial\over
\partial\alpha_\nu}\right)\right]_{\nu=\mu}
$$
are all nonzero.  Then modulo
$\Theta_{\mu;r_1,\cdots,r_n;s_1,\cdots,s_n}$ we can express
$$
{\partial\over\partial\alpha_{\mu+e_{r_1}+\cdots+e_{r_n}}}
$$
as a linear combination of
$$
{\partial\over\partial\alpha_{\mu+e_{r_{i_1}}+\cdots+e_{r_{i_p}}+
e_{s_{j_1}}+\cdots+e_{s_{j_{k-p}}}}}
$$
for $p<n$.  Now take any multi-index $\nu$ with total degree
$\delta$ which is different from any one of $\nu^{[0]},\cdots,
\nu^{[n-1]}$.  In other words,
$$
\|\nu\|_M>\delta-n+1+(n-1)M.
$$
Then for some $1\leq r_1,\cdots,r_n\leq n$, all the $n+1$
components of
$$
\nu-\sum_{\ell=1}^n e_{r_\ell}
$$
are nonnegative.  Let
$$
\displaylines{\mu=\nu-\sum_{\ell=1}^n e_{r_\ell},\cr
s_1=\cdots=s_n=0.\cr }
$$
Then modulo $\Theta{\mu;r_1,\cdots,r_n;s_1,\cdots,s_n}$ we can
express
$$
{\partial\over\partial\alpha_\nu}
$$
in terms of
$$
{\partial\over\partial\alpha_{\mu+e_{r_{i_1}}+\cdots+e_{r_{i_p}}+
e_{s_{j_1}}+\cdots+e_{s_{j_{k-p}}}}}
$$
for $p<n$ with
$$
\left\|\mu+e_{r_{i_1}}+\cdots+e_{r_{i_p}}+
e_{s_{j_1}}+\cdots+e_{s_{j_{k-p}}}\right\|_M<\|\nu\|_M.
$$
We thus conclude that modulo $$\sum_{{{{\bf a}\in GL(n+1,{\mathbb C});|\mu|=\delta-n;}\atop{ r_1,\cdots,r_n;s_1,\cdots,s_n}}}{\mathbb C}\Theta_{{\bf a};\mu; r_1,\cdots,r_n;s_1,\cdots,s_n},$$ the space
$$ \bigoplus_{|\nu|=\delta}{\mathbb C}{\partial\over\partial\alpha_\nu}$$ is generated by
$$
{\partial\over\partial\alpha_{\nu^{[0]}}},\cdots,{\partial\over\partial\alpha_{\nu^{[n-1]}}}
$$
and we conclude that the dimension of the quotient space
$$
\left(\bigoplus_{|\nu|=\delta}{\mathbb C}{\partial\over\partial\alpha_\nu}\right)\Bigg/\left(\sum_{{{{\bf
a}\in GL(n+1,{\mathbb C});|\mu|=\delta-n;}\atop{
r_1,\cdots,r_n;s_1,\cdots,s_n}}}{\mathbb C}\Theta_{{\bf a};\mu;
r_1,\cdots,r_n;s_1,\cdots,s_n}\right)
$$
is no more than $n$ over ${\mathbb C}$.  Note that the pole order of
each of the meromorphic vector field
$$
\Theta_{{\bf a};\mu; r_1,\cdots,r_n;s_1,\cdots,s_n}
$$
is no more than $2n$ along the infinity hyperplane of ${\mathbb
P}_n$.

\begin{remark}\label{remark(2.13)} The reason why in the
above argument we fail to get generation of all vectors in
parameter space is that we can only expect to get generation all
vectors in parameter space up to codimension $n$ for $(n-1)$-jets.
The vector fields have to be tangential to the space
$J^{\rm{\scriptstyle (vert)}}_{n-1}\left({\mathcal X}\right)$ of
vertical $(n-1)$-jets of ${\mathcal X}$ which is of codimension
$n$ in the product
$$J_{n-1}\left({\mathbb P}_n\right)\times {\mathbb P}_N$$ of the space $J_{n-1}\left({\mathbb P}_n\right)$ of
$(n-1)$-jets of ${\mathbb P}_n$ and the parameter space ${\mathbb
P}_N$.
\end{remark}

\bigbreak

\begin{nodesignation}\label{construct_generation_vertical_direction}{\it Generation of Vectors in Vertical
Directions}\label{nodesignation(2.14)}\end{nodesignation} Now we
construct holomorphic vector fields which generate the vertical
directions modulo the horizontal directions.
$$
\displaylines{z_j{\partial\over\partial z_j}\left(d^k f\right) =
\sum_\nu
\alpha_\nu\nu_j\Phi^{(k)}_\nu\left(\xi^{(1)},\cdots,\xi^{(k)}\right)z^\nu,\cr
{\partial\over\partial\xi_j^\ell}\left(d^k f\right) = \sum_\nu
\alpha_\nu\left[{\partial\over\partial\xi_j^{(\ell)}}\,
\Phi^{(k)}_\nu\left(\xi^{(1)},\cdots,\xi^{(k)}\right)\right]z^\nu.\cr
}
$$
To unify the notations, we use $T$ to denote any one of
$$
z_j{\partial\over\partial z_j},\ \
{\partial\over\partial\xi_j^\ell}\quad(0\leq j\leq
n,\,1\leq\ell\leq n)
$$
and write
$$
T\left(d^k f\right)=-\sum_\nu\Xi^{(k)}_\nu z^\nu.
$$
This means that
$$
\Xi^{(k)}_\nu=-\alpha_\nu\nu_j\Phi^{(k)}_\nu\left(\xi^{(1)},\cdots,\xi^{(k)}\right)
\quad{\rm when\ }T=z_j{\partial\over\partial z_j}
$$
and
$$
\Xi^{(k)}_\nu=-\alpha_\nu\left[{\partial\over\partial\xi_j^{(\ell)}}\,
\Phi^{(k)}_\nu\left(\xi^{(1)},\cdots,\xi^{(k)}\right)\right]
\quad{\rm when\ }T={\partial\over\partial\xi_j^\ell}.
$$
The function $\Xi^{(k)}_\nu$ is homogeneous of weight $k-\ell$ and
is independent of $z_0,\cdots,z_n$.

\medbreak We also unify the notations for the vector fields
$\Theta^{(k)}_{\lambda^{(k)},{\mathfrak p}^{(k)}}$ and write
$\Theta^{(k)}_\nu$ as the vector field with effective low pole
order such that
$$
\displaylines{ \Theta^{(k)}_\nu\left(d^jf\right)=0{\rm\ for \
}0\leq j<k,\cr\Theta^{(k)}_\nu\left(d^kf\right)= z^\nu,\cr
\Theta^{(k)}_\nu\left(d^jf\right)=-\Psi^{(k,j)}_\nu z^\nu{\rm\ for
\ }j>k.\cr}
$$
This means the following. Choose a binary tree ${\mathfrak
p}^{(n-1)}$ of order $n-1$ which has level-wise homogeneous
branches.  Let $r_j=p_{\gamma_1,\cdots,\gamma_{j-1},0}$ and
$s_j=p_{\gamma_1,\cdots,\gamma_{j-1},1}$ for $1\leq j\leq n-1$.
Since the binary tree ${\mathfrak p}^{(n-1)}$ of order $n-1$ has
level-wise homogeneous branches, the values of
$r_j=p_{\gamma_1,\cdots,\gamma_{j-1},0}$ and
$s_j=p_{\gamma_1,\cdots,\gamma_{j-1},1}$ are independent of the
choices of the values $0$ or $1$ for
$\gamma_1,\cdots,\gamma_{j-1}$ and $1\leq j\leq n-1$. Let
${\mathfrak p}^{(n-1-j)}={\mathfrak
p}^{(n-1;\gamma_1,\cdots,\gamma_j)}$ for $1\leq j\leq n-1$. We
know that ${\mathfrak p}^{(n-1-j)}$ is independent of the choices
of the values $0$ or $1$ for $\gamma_1,\cdots,\gamma_j$ and $0\leq
j\leq n$ because the binary tree ${\mathfrak p}^{(n-1)}$ of order
$n-1$ has level-wise homogeneous branches.

\medbreak Given any multi-index $\nu$ of $n+1$ components and
total degree $\delta$, we choose a multi-index
$\lambda_\nu^{(n-1)}$ of $n+1$ components and total degree
$\delta-n+1$ such that $\lambda_\nu^{(n-1)}\leq\nu$ in the sense
that the $j$-th component of $\lambda_\nu^{(n-1)}\leq\nu$ is no
more than $\nu_j$ for $0\leq j\leq n$.  Though the binary tree
${\mathfrak p}^{(n-1)}$ of order $n-1$ has level-wise homogeneous
branches, yet $\lambda_\nu^{(n-1;\gamma_1,\cdots,\gamma_j)}$ does
depend on the choices of the values $0$ or $1$ for
$\gamma_1,\cdots,\gamma_j$ and $1\leq j\leq n$.  The dependence is
as follows.  If we denote $r_j$ by $r_{j,0}$ and $s_j=r_{j,1}$,
then
$$
\lambda_\nu^{(n-1;\gamma_1,\cdots,\gamma_j)}=
\lambda_\nu^{(n-1-j)}=\lambda_\nu^{(n-1)}
+\sum_{\ell=1}^je_{r_{\ell,\gamma_\ell}}
$$
By Lemma~\ref{lemma(2.4)},
$$
\Theta^{(k)}_{\lambda^{(k)},{\mathfrak p}^{(k)}} \left(d^j
f\right) =z^{\lambda^{(k)}}\Psi^{(j)}_{\lambda^{(k)},{\mathfrak
p}^{(k)}} \quad{\rm for \ }k\leq j\leq n
$$
for any multi-index $\lambda^{(k)}$ of $n+1$ components and total
degree $\delta-k$. So we can set
$$
\displaylines{\Theta^{(k)}_\nu={-z^{\nu-\lambda_\nu^{(n-1-k)}}
\over\Psi^{(k)}_{{\mathfrak p}^{(k)}}}\,\Theta^{(k)}
_{\lambda_\nu^{(n-1-k;\gamma_1,\cdots,\gamma_k)},{\mathfrak
p}^{(k)}},\cr \Psi^{(k,j)}_\nu= {-z^{\nu-\lambda_\nu^{(n-1-k)}}
\over\Psi^{(k)}_{{\mathfrak
p}^{(k)}}}\,\Psi^{(j)}_{\lambda_\nu^{(n-1-k;\gamma_1,\cdots,\gamma_k)},
{\mathfrak p}^{(k)}}.\cr }
$$
By Lemma~\ref{lemma(2.8)}, The function
$$
\Psi^{(k,j)}_\nu\prod_{\ell=1}^k\left(\xi^{(1)}_{r_\ell}-\xi^{(1)}_{s_\ell}\right)
$$
is of homogeneous weight $j$ in $\xi^{(\ell)}_t$ ($1\leq\ell\leq
j$, $0\leq t\leq n$) and is a polynomial in the variables
$z_0,\cdots,z_n$ of degree $\leq n-1-k$ and in the $n+1$
components of $\nu$ of degree $\leq j$. The vector field
$$
\Theta^{(k)}_\nu \prod_{\ell=1}^k\left(z_{r_\ell}z_{s_\ell}
\left(\xi^{(1)}_{r_\ell}-\xi^{(1)}_{s_\ell}\right)\right)
$$
is a polynomial in $z_0,\cdots,z_n$ of degree $\leq n-1$ and is
independent of $\xi^{(\ell)}_t$ ($1\leq\ell\leq n-1$, $0\leq t\leq
n$). Moreover, the dependence of
$\Theta^{(j)}_{\lambda^{(k)},{\mathfrak p}^{(k)}}$ on $\nu$ is
only through the partial differentiation with respect to
$\alpha_\mu$ with $\mu$ depending on $\nu$.

\medbreak To make sure that a modification of $T$ annihilates $f$,
we modify $T$ to $T+\sum_\nu\Xi^{(0)}_\nu\Theta^{(0)}_\nu$. To
make sure that our constructed vector field annihilates $df$, we
use
$$
\left(T+\sum_\nu\Xi^{(0)}_\nu\Theta^{(0)}_\nu\right)(df) =
-\sum_\nu\Xi^{(0)}_\nu\Psi^{(0,1)}_\nu z^\nu-\sum_\nu\Xi^{(1)}_\nu
z^\nu.
$$
This means that we have to modify
$T+\sum_\nu\Xi^{(0)}_\nu\Theta^{(0)}_\nu$ to
$$T+\sum_\nu\Xi^{(0)}_\nu\Theta^{(0)}_\nu+\sum_\nu\Xi^{(0)}_\nu\Psi^{(0,1)}_\nu\Theta^{(1)}_\nu
+\sum_\nu\Xi^{(1)}_\nu \Theta^{(1)}_\nu.
$$
To go one step further to make sure that our constructed vector
field annihilates $d^2f$, we use
$$
\displaylines{
\left(T+\sum_\nu\Xi^{(0)}_\nu\Theta^{(0)}_\nu+\sum_\nu\Xi^{(0)}_\nu\Psi^{(0,1)}_\nu\Theta^{(1)}_\nu
+\sum_\nu\Xi^{(1)}_\nu \Theta^{(1)}_\nu\right)\left(d^2f\right)\cr
=-\sum_\nu\Xi^{(0)}_\nu\Psi^{(0,2)}_\nu
z^\nu-\sum_\nu\Xi^{(0)}_\nu\Psi^{(0,1)}_\nu\Psi^{(1,2)}_\nu
z^\nu-\sum_\nu\Xi^{(1)}_\nu\Psi^{(1,2)}_\nu
z^\nu-\sum_\nu\Xi^{(2)}_\nu z^\nu.\cr }
$$
Thus we have to modify
$$T+\sum_\nu\Xi^{(0)}_\nu\Theta^{(0)}_\nu+\sum_\nu\Xi^{(0)}_\nu\Psi^{(0,1)}_\nu\Theta^{(1)}_\nu
+\sum_\nu\Xi^{(1)}_\nu \Theta^{(1)}_\nu
$$
to
$$
\displaylines{T+\sum_\nu\Xi^{(0)}_\nu\Theta^{(0)}_\nu+\sum_\nu\Xi^{(0)}_\nu\Psi^{(0,1)}_\nu\Theta^{(1)}_\nu
\cr+\sum_\nu\Xi^{(0)}_\nu\Psi^{(0,2)}_\nu\Theta^{(2)}_\nu
+\sum_\nu\Xi^{(0)}_\nu\Psi^{(0,1)}_\nu\Psi^{(1,2)}_\nu\Theta^{(2)}_\nu\cr
+\sum_\nu\Xi^{(1)}_\nu \Theta^{(1)}_\nu
+\sum_\nu\Xi^{(1)}_\nu\Psi^{(1,2)}_\nu\Theta^{(2)}_\nu
+\sum_\nu\Xi^{(2)}_\nu\Theta^{(2)}_\nu.\cr}
$$
In general, to make sure that we have the annihilation of all
$d^jf$ for $0\leq j\leq n-1$, we need to write
$$
T+\sum_\nu\sum_{0\leq j_0\leq k\leq n-1} \Xi^{(j_0)}_\nu
\left(\sum_{\ell=0}^{k-j_0-1}\sum_{j_0<\cdots<j_\ell<k}
\Psi^{(j_\ell,k)}_\nu\prod_{q=0}^{\ell-1}\Psi^{(j_q,j_{q+1})}_\nu\right)
\Theta^{(k)}_\nu.
$$
The main point is to control the pole order of the vector fields
and make the pole order bounded and independent of $\delta$.  That
is the reason why we want to remove $z^\nu$ by using the vector
fields $\Theta^{(k)}_{\lambda^{(k)},{\mathfrak p}^{(k)}}$. We now
count the degree in $z_0,\cdots,z_n$ and the weight in
$\xi^{(\ell)}_j$ after we clear the denominators.  We need to
multiply
$$\Psi^{(j_\ell,k)}_\nu\left(\prod_{q=0}^{\ell-1}\Psi^{(j_q,j_{q+1})}_\nu
\right)\Theta^{(k)}_\nu$$ by
$$
\left[\prod_{q=0}^\ell\left(\prod_{i=1}^{j_q}\left(\xi^{(1)}_{r_i}-\xi^{(1)}_{s_i}\right)\right)
\right]\prod_{i=1}^k\left(\xi^{(1)}_{r_i}-\xi^{(1)}_{s_i}\right)
$$
to get rid of the denominator involving $\xi^{(\ell)}_j$.  The
worst that can occur is
$$
\prod_{\ell=1}^k\left(\xi^{(1)}_{r_\ell}-\xi^{(1)}_{s_\ell}\right)^{n-\ell},
$$
whose weight is $n(n-1)\over 2$.  Since
$$
\Psi^{(k,j)}_\nu\prod_{\ell=1}^k\left(\xi^{(1)}_{r_\ell}-\xi^{(1)}_{s_\ell}\right)
$$
is of homogeneous weight $j$ in $\xi^{(\ell)}_t$ ($1\leq\ell\leq
j$, $0\leq t\leq n$), it follows that the worst situation is that
after multiplication by the above factor to clear the denominator
we end up with a weight of $j_0+j_1+\cdots+j_\ell$ plus the weight
of the factor which is no greater than the weight
$$
j_0+j_1+\cdots+j_\ell+{n(n-1)\over 2}\leq n(n-1).
$$
When it comes to the degree in $z_0,\cdots,z_n$, we have degree
$1$ from $T$, multiplication by the factor
$$
\prod_{\ell=1}^k\left(z_{r_\ell}z_{s_\ell}\right)
$$
to clear the denominator of $\Theta^{(k)}_\nu$ to yield degree
$\leq n-1$, and the degree of
$$
\Psi^{(k,j)}_\nu\prod_{\ell=1}^k\left(\xi^{(1)}_{r_\ell}-\xi^{(1)}_{s_\ell}\right)
$$
no more than $n-1-k$.  So after clearing the denominators, we have
no than $2(n-1)+1$ in degree for $T$ and no more than
$$
\displaylines{(n-1-j_0)+\cdots+(n-1-j_k)+(n-1)+2n \cr\leq
n(n-1)+(n-1)+2n\leq n^2-1+2n\cr}
$$
for
$$\Psi^{(j_\ell,k)}_\nu\left(\prod_{q=0}^{\ell-1}\Psi^{(j_q,j_{q+1})}_\nu\right)
\Theta^{(k)}_\nu.$$ Finally we conclude that, after clearing the
denominators, we end up with weight no more than $n(n-1)$ in
$\xi^{(\ell)}_t$ ($1\leq\ell\leq j$, $0\leq t\leq n$) and degree
no more than $n^2-1+2n$ in $z_0,\cdots,z_n$.

\bigbreak

\begin{nodesignation}{\it Vector Fields in Terms of
Differentiation with Respect to Inhomogeneous
Coordinates}.\label{nodesignation(2.15)}\end{nodesignation}  The
introduction of homogeneous coordinates is simply for the
notational convenience of our discussion.  We now return to
inhomogeneous coordinates by specializing to $z_0\equiv 1$.  First
of all we would like to go back to the coordinates
$$
d^jz_\ell\quad(0\leq j\leq n-1,\,0\leq\ell\leq n)
$$
from the coordinates
$$
\displaylines{z_0,z_1,\cdots,z_n,\cr d^j\log z_\ell\quad(1\leq
j\leq n-1,\,0\leq\ell\leq n)\cr }
$$
which is the same as
$$
\displaylines{z_0,z_1,\cdots,z_n,\cr \xi^{(j)}_\ell\quad(1\leq
j\leq n-1,\,0\leq\ell\leq n)\cr}
$$
(because $\xi^{(k)}_j=d^k\log z_j$). We are going to use the chain
rule for the transformation of vector fields.
$$
{\partial\over\partial\left(d^\ell
z_p\right)}=\sum_{p,k}{\partial\xi^{(k)}\over\partial\left(d^\ell
z_p\right)}{\partial\over\partial\xi^{(k)}_j}.
$$
Since $\xi^{(k)}_j=d^{k-1}\left({d z_j\over z_j}\right)$, it
follows that
$$
\xi^{(k)}_j=\xi^{(k)}_j\left(z_0,\cdots,z_n,dz_0,\cdots,dz_n,\cdots,d^nz_0,\cdots,d^nz_n\right)
$$
is a rational function which is homogeneous of weight $0$ when
$d^\ell z_p$ is assigned weight $1$ and is homogeneous of weight
$k$ when $d^\ell z_p$ is assigned weight $\ell$. Thus
$$
{\partial\xi_j^{(\ell)}\over\partial\left(d^\ell z_p\right)}
$$
is of weight $-1$ when $d^\ell z_p$ is assigned weight $1$ and is
homogeneous of weight $k-\ell$ when $d^\ell z_p$ is assigned
weight $\ell$.  It follows from weight considerations that
$$
{\partial\xi_j^{(\ell)}\over\partial\left(d^\ell
z_p\right)}=\delta_{p,j}\left({1\over z_j}\right),
$$
where $\delta_{p,j}$ is the Kronecker delta.  We conclude that, so
far as the independence of the constructed vector fields are
concerned, it makes no difference whether we are using the
coordinate system
$$
\displaylines{z_0,z_1,\cdots,z_n,\cr d^j\log z_\ell\quad(1\leq
j\leq n-1,\,0\leq\ell\leq n)\cr }
$$
or the coordinate system
$$
\displaylines{z_0,z_1,\cdots,z_n,\cr \xi^{(j)}_\ell\quad(1\leq
j\leq n-1,\,0\leq\ell\leq n).\cr}
$$
Now we pass from the homogeneous coordinates to the inhomogeneous
coordinates.  It is equivalent to restricting all the objects to
the linear subspace
$$
z_0=1,dz_0=d^2z_0=\cdots=d^n z_0=0.
$$
So far as
$$
{\partial\over\partial\xi_j^{(\ell)}},\ \
{\partial\over\partial\left(d^\ell z_p\right)}
$$
are concerned, the linear subspace is part of the line defined by
setting some coordinates equal to constant and the argument is not
affected. Since the pole order of $d^\ell z_j$ is no more than
$\ell+1$, we have the following proposition.  A point
of the space $J^{\rm{\scriptstyle vert}}_{n-1}\left({\mathcal
X}\right)$ of vertical $(n-1)$-jets is {\it represented by a nonsingular complex curve
germ in ${\mathbb P}_n$} precisely when the value of $z_1dz_0-z_0dz_1$ is nonzero at it
for some homogeneous coordinate system $z_0,\cdots,z_n$ of ${\mathbb P}_n$.

\begin{proposition}\label{sufficient_slanted_vector_fields} {\rm (Global Generation on Jet Space by Slanted Vector Fields at Points Representable by Regular Curve Germs)}\ \ Let $P_0$ be a
point of the space $J^{\rm{\scriptstyle
vert}}_{n-1}\left({\mathcal X}\right)$ of vertical $(n-1)$-jets
such that $P_0$ can be represented by a nonsingular complex curve germ in ${\mathbb P}_n$.
Then the meromorphic vector
fields on ${\mathbb P}_n\times{\mathbb P}_N$ tangential to
$$
\left\{f=df=\cdots=d^nf=0\right\}
$$
of pole order $\leq n^2+2n+n(n-1)=n(2n+1)$ (along the infinity
hyperplane of ${\mathbb P}_n$) generate at $P_0$ the tangent space
of the total space of fiber-direction $(n-1)$-jets, where $d^k f$
is taken with $\alpha_\nu$ regarded as constants. In terms of
inhomogeneous coordinates, the statement is equivalent to that of
the formulation in terms of homogeneous coordinates on the
restriction to
$$
\left\{z_j=1, dz_j=d^2z_j=\cdots=d^nz_j=0\right\}
$$
for each $0\leq j\leq n$.
\end{proposition}

\begin{proof}  Since $P_0$ can be represented by a nonsingular complex curve germ in ${\mathbb P}_n$, there exists some homogeneous
coordinate system $z_0,\cdots,z_n$ such that $z_0\not=0$ and
$z_1dz_0-z_0dz_1\not=0$ at it.  Since one of $z_0$ and $z_1$ must
be nonzero at $P_0$, we assume without loss of generality that $z_0$
is nonzero at $P_0$.  Let $x_j={z_j\over z_0}$ for $1\leq
j\leq n$.  Then $dx_1={z_0dz_1-z_1dz_0\over z_0^2}\not=0$ at
$P_0$.  We then apply a translation to the the affine coordinates
$\left(x_1,\cdots,x_n\right)$ to make sure that $x_\ell\not=0$ at
$P_0$ for $1\leq\ell\leq n$.  Then we apply a linear
transformation to the affine coordinates
$\left(x_1,\cdots,x_n\right)$ that $x_\ell\not=0$ at $P_0$ for
$1\leq\ell\leq n$ and $dx_1\not=dx_2$ at $P_0$. We are going to
set $r_j=1$ and $s_j=2$ for $1\leq j\leq n$.  Also we will
restrict the vector fields to $z_0=1$ and $z_\ell=x_\ell$ for
$1\leq\ell\leq n$ so that
$$\xi^{(1)}_{r_j}-\xi^{(1)}_{s_j}=\xi^{(1)}_1-\xi^{(1)}_2\not=0
$$
at $P_0$ and $z_\ell\not=0$ at $P_0$ for $0\leq\ell\leq n$.  The
above construction now gives the generation of the tangent bundle
of $J^{\rm{\scriptstyle vert}}_{n-1}\left({\mathcal X}\right)$ at
$P_0$.\end{proof}

\begin{remark}\label{remark(2.16.1)}  In the global generation of the tangent
bundle of ${\mathcal X}$ in Lemma~\ref{lemma(1.7)} there is no reference to the
tangent vector being representable by nonsingular complex curve germ, because a tangent
vector which is not representable by a nonsingular complex curve germ must be zero and
the identically zero global vector field already generates the zero tangent vector.
However, a higher-order jet which cannot be represented by a nonsingular complex curve germ
need not be zero.  The condition of representability by a nonsingular complex curve germ can
be technically suppressed by formulating global generation over some suitably defined projectivization of the jet space which includes only those jets which have well-defined images in the projectivization of the tangent bundle.

\medbreak For the hyperbolicity of generic hypersurface $X$ of sufficiently high degree, the generation by slanted vector fields of low vertical pole order only at jets representable by nonsingular complex curve germs offers no difficulty, because any nonconstant holomorphic map $\varphi$ from the affine complex line ${\mathbb C}$ to $X$ must have a nonzero tangent vector at some point $\zeta_0$ of ${\mathbb C}$ and we need only use slanted vector fields of low vertical pole order at the jet represented by $\varphi$ at $\varphi\left(\zeta_0\right)$.
\end{remark}

\bigbreak The slanted vector fields on $J^{\rm{\scriptstyle vert}}_{n-1}\left({\mathcal X}\right)$ of low vertical pole order constructed in (\ref{construct_generation_vertical_direction}) for the the proof of Proposition \ref{sufficient_slanted_vector_fields} start out with the vector field $T$ which is any one of
$$
z_j{\partial\over\partial z_j},\ \
{\partial\over\partial\xi_j^\ell}\quad(0\leq j\leq
n,\,1\leq\ell\leq n).
$$
The constructed slanted vector fields on $J^{\rm{\scriptstyle vert}}_{n-1}\left({\mathcal X}\right)$ of low vertical pole order are actually restrictions to
$J^{\rm{\scriptstyle vert}}_{n-1}\left({\mathcal X}\right)$ of vector fields on $J_{n-1}\left({\mathbb P}_n\right)\times{\mathbb P}_N$ of low vertical pole order which are tangential to $J^{\rm{\scriptstyle vert}}_{n-1}\left({\mathcal X}\right)$.  We formulate below a proposition about the slanted vector fields on
$J_{n-1}\left({\mathbb P}_n\right)\times{\mathbb P}_N$ of low vertical pole order which are tangential to $J^{\rm{\scriptstyle vert}}_{n-1}\left({\mathcal X}\right)$ before their restrictions to $J^{\rm{\scriptstyle vert}}_{n-1}\left({\mathcal X}\right)$.  This formulation is needed in \ref{proof_hyperbolicity_complement_hypersurface} where the proof of Theorem \ref{main_theorem} is modified to give a proof of Theorem \ref{hyperbolicity_complement_hypersurface}.

\begin{proposition}\label{vector_field_projective_space} Let $\alpha\in{\mathbb P}_N$ and  $P_0$ be a
point of the space $J_{n-1}\left({\mathbb P}_n\right)\times{\mathbb P}_N$
such that $P_0$ can be represented by a nonsingular complex curve germ in ${\mathbb P}_n\times\{\alpha\}$ which lies in ${\mathcal X}$.
Then the meromorphic vector
fields on $J_{n-1}\left({\mathbb P}_n\right)\times{\mathbb P}_N$
of pole order $\leq n^2+2n+n(n-1)=n(2n+1)$ (along the infinity
hyperplane of ${\mathbb P}_n$) which are tangential to
$J^{\rm{\scriptstyle vert}}_{n-1}\left({\mathcal X}\right)$ generate at $P_0$ the tangent space
of $J_{n-1}\left({\mathbb P}_n\right)\times{\mathbb P}_N$.
\end{proposition}

\bigbreak\begin{nodesignation}{\it Use of Slanted Vector Fields to Lower Vanishing Order of Jet Differentials and to Generate Linearly Independent Jet Differentials.}\label{nodesignation(2.17)}\end{nodesignation}  First we would
like to make a remark about the weight of a jet differential
after the application of the vector fields which we have
constructed. The weight of $d^jx_\ell$ is $j$.  The coordinates
$x_\ell$ has weight zero and does not contribute at all to the
computation of weights. When we consider the vector field which
starts with
$${\partial\over\partial\left(d^kx_\ell\right)},$$ to clear the
denominator we have to multiply the result by the factor
$$\left({dx_1\over x_1}-{dx_2\over x_2}\right)^k$$ so that one
ends up with
$$
\left({dx_1\over x_1}-{dx_2\over x_2}\right)^k
{\partial\over\partial\left(d^kx_\ell\right)}
$$
which means that the action of the vector field after clearing out
the denominator preserves the weight of the jet differential.  Moreover, by
the explicit construction of the slanted vector fields of low vertical pole order, we cannot apply them
to $(n-1)$-jet differentials to lower their orders to get
$k$-jet differentials for some $k<n-1$.

\medbreak For the hyperbolicity problem there are two ways to apply the constructed slanted vector fields of low vertical pole order.  One is to lower the vanishing order of an $(n-1)$-jet differential on a generic hypersurface at a prescribed point by applying slanted vector fields to the extensions of the $(n-1)$-jet differential on neighboring hypersurfaces.

\medbreak The other is, from a given $(n-1)$-jet differential on a generic hypersurface of a given weight which is nonzero at a prescribed point, to generate more $(n-1)$-jet differentials so that the resulting $(n-1)$-jet differentials at the prescribed point span the finite-dimensional vector space of all $(n-1)$-jet differentials of that particular weight defined only at the prescribed point.  Again the slanted vector fields have to be applied to the extensions of the given $(n-1)$-jet differential to neighboring hypersurfaces.  For the hyperbolicity problem, in both applications the given $(n-1)$-jet differential to which slanted vector fields are applied, as well as its extension on neighboring hypersurfaces, is assumed to be holomorphic and vanish to sufficiently high order on some ample divisor in order that the resulting $(n-1)$-jet differential is holomorphic and still vanishes on some ample divisor, after part of the vanishing order on the ample divisor of the given $(n-1)$-jet differential (as well as its extensions to neighboring hypersurfaces) is used to cancel the low vertical pole orders of the constructed slanted vector fields.

\medbreak Note that for a generic hypersurface a holomorphic $(n-1)$-jet differential vanishing to a sufficiently high order on an ample divisor is automatically extendible to a holomorphic $(n-1)$-jet differential on a neighboring hypersurface vanishing also to a sufficiently high order on an ample divisor (see Proposition \ref{extendibility_to_neighborhing_fiber} below).

\medbreak The following proposition is a precise formulation of the applications of the constructed slanted vector fields. The jet differentials to which the slanted vector fields are applied will be constructed in \S\ref{section_construction_jet_differential} below.  The polynomials $g^{(\alpha)}$ are introduced in the proposition in order to use the coefficients of $g^{(\alpha)}$ to control the linear independence of the resulting jet differentials, because of other Lie differentiations coming after the Lie differentiation by $g^{(\alpha)}$ times a slanted vector field.

\begin{proposition}\label{lower_vanishing_order_generate_independent_jet_differential} {\it( Slanted Vector Fields to Reduce Vanishing Order and to Generate Independent Jet Differentials).} Let $\hat\alpha\in{\mathbb P}_N$ and $U$ be an open neighborhood of $\hat\alpha$ such that $X^{(\alpha)}$ is nonsingular for $\alpha\in U$. Let $\hat y\in X^{(\hat\alpha)}$ and $C$ be a nonsingular curve germ in $X^{(\hat\alpha)}$ at $\hat y$.  Let $q_0$ and $m$ be positive integers.  For $\alpha\in U$ let $\omega^{(\alpha)}$ be a holomorphic $(n-1)$-jet differential on $X^{(\alpha)}$ of weight $m$ which vanishes to order $\geq q_0$ at the intersection of $X^{(\hat\alpha)}$ and some hyperplane section of ${\mathbb P}_n$ and which varies holomorphically as $\alpha$ varies in $U$.  Assume that the pullback of $\omega^{(\hat\alpha)}$ to $C$ as an $(n-1)$-jet differential on $C$ has a coefficient with vanishing order $\leq r_0$ at $\hat y$ for some nonnegative integer $r_0$ and assume also that $q_0>(r_0+1)(m+1)(n-1)n(2n+1)$. Then for some $J\in{\mathbb N}$ there exist holomorphic $(n-1)$-jet differentials $\omega_j^{(\hat\alpha)}$ on $X^{(\hat\alpha)}$ for $1\leq j\leq J$ which vanish on an ample divisor of $X^{(\hat\alpha)}$ and which have no common zeroes, as functions of homogeneous weight, on $J_{n-1}\left(X^{(\hat\alpha)}\right)_{\hat y}$ other than the zero $(n-1)$-jet of $X^{(\hat\alpha)}$ at $\hat y$, where $J_{n-1}\left(X^{(\hat\alpha)}\right)_{\hat y}$ is the finite-dimensional ${\mathbb C}$-vector space of all $(n-1)$-jets of $X^{(\hat\alpha)}$ at $\hat y$.  Moreover, each of the $(n-1)$-jet differentials $\omega_j^{(\hat\alpha)}$ for $1\leq j\leq J$ can be given as the restriction to $X^{(\hat\alpha)}$ of the $r_0$-times iterated Lie derivative of $\omega_j^{(\alpha)}$ with respect to $r_0$ slanted vector fields, each of which is a slanted vector field of the kind constructed in Proposition \ref{sufficient_slanted_vector_fields} and Proposition \ref{vector_field_projective_space} multiplied by some polynomial $g_j^{(\alpha)}$ of degree $\leq r_0$ in the inhomogeneous coordinates of ${\mathbb P}_n$ whose coefficients are holomorphic in $\alpha\in U$.
\end{proposition}

\bigbreak As long as $\hat\alpha$ is a generic point of the parameter space ${\mathbb P}_N$, in Proposition \ref{lower_vanishing_order_generate_independent_jet_differential} it suffices to assume the existence of one single $\omega^{(\hat\alpha)}$ on $X^{(\hat\alpha)}$ instead of a family of $\omega^{(\alpha)}$ on $X^{(\alpha)}$ for $\alpha\in U$ which is holomorphic in $\alpha\in U$, because of the following general abstract statement.

\begin{proposition}\label{extendibility_to_neighborhing_fiber}{\it (Extendibility of Jet Differentials on Generic Fiber to Neighborhing Fibers).}  Let $\tilde\pi:{\mathcal Y}\to S$ be a flat
holomorphic family of compact complex spaces and ${\mathcal
L}\to{\mathcal Y}$ be a holomorphic vector bundle.  Then there exists a
proper subvariety $Z$ of $S$ such that for $s\in S-Z$ the
restriction map
$$
\Gamma\left(U_s,{\mathcal L}\right)\to
\Gamma\left(\tilde\pi^{-1}(s),{\mathcal
L}|_{\tilde\pi^{-1}(s)}\right)
$$
is surjective for some open neighborhood $U_s$ of $s$ in $S$.
\end{proposition}

\bigbreak In Proposition \ref{lower_vanishing_order_generate_independent_jet_differential}, if for some point $\hat\alpha$ of ${\mathbb P}_N$ and for every point $\hat y$ of $X^{(\hat\alpha)}$ the assumption of Proposition \ref{lower_vanishing_order_generate_independent_jet_differential} is satisfied, then the hypersurface $X^{(\hat\alpha)}$ is hyperbolic in the sense that there is no nonconstant holomorphic map from ${\mathbb C}$ to $X^{(\hat\alpha)}$.  The following proposition formulates precisely this result and will be applied to prove Theorem \ref{main_theorem} after the construction of the required holomorphic jet differentials in \S\ref{section_construction_jet_differential} below and after the analysis in Proposition \ref{basepoint_freeness_jet_differential} of the effect on them from the change of inhomogeneous coordinates of ${\mathbb P}_n$ used in the construction.

\begin{proposition}
\label{hyperbolicity_after_construction_jet_differential} {\it (\it Hyperbolicity from Existence of Appropriate Jet Differentials).} Let $\hat\alpha\in{\mathbb P}_N$ such that the hypersurface $X^{(\hat\alpha)}$ is nonsingular.  Suppose, for every $\hat y\in X^{(\hat\alpha)}$ and for every $C$ whose $(n-1)$-jet at $\hat y$ is generic, the assumption of Proposition \ref{lower_vanishing_order_generate_independent_jet_differential} is satisfied for some $C$, $U$, $\omega_1^{(\alpha)},\cdots,\omega_J^{(\alpha)}$, $q_0$, $r_0$ which may depend on the point $\hat y$ of $X^{(\hat\alpha)}$.  Then the hypersurface $X^{(\hat\alpha)}$ is hyperbolic in the sense that there is no nonconstant holomorphic map from ${\mathbb C}$ to $X^{(\hat\alpha)}$.
\end{proposition}

\begin{proof} Suppose the contrary and there is some nonconstant holomorphic map $\varphi$ from ${\mathbb C}$ to $X^{(\hat\alpha)}$.  For some $\zeta_0\in{\mathbb C}$ where the differential $d\varphi$ of $\varphi$ at $\zeta_0$ is nonzero.  We let $\hat y=\varphi(\zeta_0)$.  By the argument of Proposition \ref{lower_vanishing_order_generate_independent_jet_differential} there exist holomorphic $(n-1)$-jet differentials $\omega_j^{(\hat\alpha)}$ on $X^{(\hat\alpha)}$ for $1\leq j\leq J$ which vanish on an ample divisor of $X^{(\hat\alpha)}$ and which, as $J$ functions of homogeneous weight on the finite-dimensional Euclidean space $J_{n-1}\left(X^{(\hat\alpha)}\right)_{\hat y}$ of all $(n-1)$-jets of $X^{(\hat\alpha)}$ at the point $\hat y$, have no common zeroes other than the zero $(n-1)$-jet of $X^{(\hat\alpha)}$ at $\hat y$.  By applying the Schwarz lemma of the vanishing of pullbacks, by a holomorphic map from ${\mathbb C}$ to a compact algebraic manifold, of jet differentials vanishing on an ample divisor of $X^{(\hat\alpha)}$, we conclude that the pullbacks of each of $\omega_1^{(\hat\alpha)},\cdots,\omega_J^{(\hat\alpha)}$ by $\varphi$ is identically zero on  ${\mathbb C}$.  This means that the nonzero $(n-1)$-jet of $X^{(\hat\alpha)}$ at $\hat y$ defined by the map $\varphi$ is a common zero of the $J$ functions of homogeneous weight on $J_{n-1}\left(X^{(\hat\alpha)}\right)_{\hat y}$ defined by $\omega_1^{(\hat\alpha)},\cdots,\omega_J^{(\hat\alpha)}$ at $\hat y$.  This is a contradiction.\end{proof}

\begin{remark}\label{remark(2.24)} In the application of Proposition \ref{hyperbolicity_after_construction_jet_differential} for the proof of Theorem \ref{main_theorem} given in \ref{proof_main_theorem}, only the special case of $r_0=0$ is used.
\end{remark}

\section{\sc Construction
of Holomorphic Jet Differentials}\label{section_construction_jet_differential}

\bigbreak We are going to construct holomorphic jet differentials.
One crucial ingredient is the use of the Koszul complex to show
that a homogeneous polynomial of low degree in $n+1$ homogeneous
coordinates and their differentials up to order $n-1$ cannot
locally belong to the ideal generated by a second homogeneous
polynomial and its differentials up to order $n-1$ when the second
homogeneous polynomial is a homogeneous polynomial of high degree
in the $n+1$ homogeneous coordinates (see Lemma \ref{injectivity_pullback_map_jet_differential} below).  The jet differentials are constructed by using the linear algebra method of solving a system of linear equations with more unknowns than independent linear equations (see Proposition \ref{construction_jet_differential_as_polynomial} below).

\begin{lemma}\label{lemma(3.1)} Let $Y$ be a compact
complex manifold and $Z$ be a subvariety of pure codimension at
least $2$ in $Y$. Let ${\mathcal F}$ be a locally free sheaf on
$Y$. Then the restriction map
$$
H^q\left(Y,{\mathcal F}\right)\to H^q\left(Y-Z,{\mathcal F}\right)
$$
is an isomorphism for $0\leq q\leq{\rm codim}_Y Z-2$.
\end{lemma}

\begin{proof} This is a standard
removability result for cohomology groups.  Let $\{U_j\}_j$ be a
finite cover of $Y$ by Stein open subsets $U_j$.  Since
$$
H^r\left(\bigcap_{\ell=0}^p\left(U_{j_\ell}-Z\right),{\mathcal
F}\right)=0
$$
for $1\leq r\leq{\rm codim}_Y Z-2$ and for any $j_0,\cdots,j_p$,
by Leray's theorem the following natural isomorphism
$$
H^q\left(Y-Z,{\mathcal F}\right)\approx
H^q\left(\left\{U_j-Z\right\}_j,{\mathcal F}\right)
$$
gives the computation of the sheaf cohomology by Cech cohomology.
Since the restriction map
$$
\Gamma\left(\bigcap_{\ell=0}^p U_{j_\ell},{\mathcal
F}\right)\to\Gamma\left(\bigcap_{\ell=0}^p\left(U_{j_\ell}-Z\right),{\mathcal
F}\right)
$$
is bijective for any $j_0,\cdots,j_p$, it follows that the map
$$
H^q\left(\left\{U_j\right\}_j,{\mathcal F}\right)\approx
H^q\left(\left\{U_j-Z\right\}_j,{\mathcal F}\right)
$$
defined by restriction is an isomorphism. The lemma follows from
the following natural isomorphism
$$
H^q\left(Y,{\mathcal F}\right)\approx
H^q\left(\left\{U_j\right\}_j,{\mathcal F}\right)
$$
gives the computation of the sheaf cohomology by Cech cohomology.
\end{proof}

\begin{lemma}\label{lemma(3.2)} Let $\ell$ and $a\leq N$
be positive integers. Let $Z$ be a linear subspace of ${\mathbb
P}_N$ and let
$$
F_1,\cdots,F_a\in\Gamma({\mathbb P}_N,{\mathcal O}_{{\mathbb
P}_N}(\ell))
$$
such that the zero-set of $F_1,\cdots,F_a$ in ${\mathbb P}_N-Z$ is
a submanifold of codimension $a$ in ${\mathbb P}_N-Z$ which is a
complete intersection. Assume that
$$
H^q\left({\mathbb P}_N-Z,{\mathcal O}_{{\mathbb
P}_N}(r)\right)=0\quad\hbox{for } 1\leq q<a$$ for any integer $r$.
Then
$$
\Gamma\left({\mathbb P}_N-Z,{\mathcal O}_{{\mathbb P}_N}^{\oplus
a}(-\ell+p)\right) \rightarrow\Gamma\left({\mathbb
P}_N-Z,\sum_{j=1}^a{\mathcal O}_{{\mathbb P}_N}(p)F_j\right)
$$
induced by
$$
\left(g_1,\cdots,g_a\right)\mapsto\sum_{j=1}^a g_j F_j
$$ is
surjective.
\end{lemma}

\begin{proof} Consider the Koszul complex
$$
\displaylines{ 0\rightarrow{\mathcal O}_{{\mathbb
P}_N}^{\oplus{a\choose a}}(-a\ell+p) \rightarrow\cdots\rightarrow
{\mathcal O}_{{\mathbb P}_N}^{\oplus{a\choose
k}}(-k\ell+p)\stackrel{\phi_k}{\longrightarrow}{\mathcal
O}_{{\mathbb P}_N}^{\oplus{a\choose
k-1}}(-(k-1)\ell+p)\cr\rightarrow \cdots\rightarrow{\mathcal
O}_{{\mathbb P}_N}^{\oplus{a\choose 2}}(-2\ell+p)\rightarrow
{\mathcal O}_{{\mathbb P}_N}^{\oplus{a\choose
1}}(-\ell+p)\stackrel{\phi_1}{\longrightarrow}{\mathcal
O}_{{\mathbb P}_N}^{\oplus{a\choose 0}}(p).\cr }
$$
The homomorphisms in the Koszul complex is defined as follows.
Take symbols $e_1,\cdots,e_a$. We use
$$
e_{i_1}\wedge\cdots\wedge e_{i_k}\quad(1\leq i_1<\cdots<i_k\leq a)
$$
as a local basis for ${\mathcal O}_{{\mathbb
P}_N}^{\oplus{a\choose k}}(-k\ell+p)$ to represent an element
$$
\left(g_{i_1,\cdots,i_k}\right)_{1\leq i_1<\cdots<i_k\leq a}
$$
of ${\mathcal O}_{{\mathbb P}_N}^{\oplus{a\choose k}}(-k\ell+p)$
as
$$
\sum_{1\leq i_1<\cdots<i_k\leq a}g_{i_1,\cdots,i_k}
\left(e_{i_1}\wedge\cdots\wedge e_{i_k}\right) $$ and define
$$\varphi_k: {\mathcal O}_{{\mathbb P}_N}^{\oplus{a\choose
k}}(-k\ell+p) \rightarrow{\mathcal O}_{{\mathbb
P}_N}^{\oplus{a\choose k-1}}(-(k-1)\ell+p)$$ by
$$
\varphi_k\left(e_{i_1}\wedge\cdots\wedge
e_{i_k}\right)=\sum^k_{\nu =1} (-1)^{\nu
-1}F_{i_\nu}\left(e_{i_1}\wedge\cdots\wedge e_{i_{\nu -1}}\wedge
e_{i_{\nu +1}}\cdots\wedge e_{i_k}\right)
$$
in such a representation. Since the zero-set of $F_1,\cdots,F_a$
in ${\mathbb P}_N-Z$ is a submanifold of codimension $a$ in
${\mathbb P}_N-Z$ which is a complete intersection, it follows
that the Koszul complex is exact on ${\mathbb P}_N-Z$.

\medbreak We are going to prove by descending induction on $b$ for
$1\leq b\leq a-1$ that
$$
H^q\left({\mathbb P}_N-Z,{\rm Ker\,}\phi_b\right)=0
$$
for $1\leq q\leq b$.  The case $b=a-1$ follows from the assumption
of the lemma and
$$
{\rm Ker\,}\phi_{a-1}={\mathcal O}_{{\mathbb
P}_N}^{\oplus{a\choose a}}(-a\ell+p).
$$
For $1\leq b<a-1$ the exact sequence
$$
0\to {\rm Ker\,}\phi_b\to {\mathcal O}_{{\mathbb
P}_N}^{\oplus{a\choose b}}(-b\ell+p)\to {\rm Ker\,}\phi_{b-1}\to 0
$$
yields the exactness of
$$
\displaylines{H^q\left({\mathbb P}_N-Z,{\mathcal O}_{{\mathbb
P}_N}^{\oplus{a\choose b}}(-b\ell+p)\right)\to H^q\left({\mathbb
P}_N-Z,{\rm Ker\,}\phi_{b-1}\right)\cr\to H^{q+1}\left({\mathbb
P}_N-Z,{\rm Ker\,}\phi_b\right)\cr }
$$
and for $1\leq q\leq b-1$ we conclude, from
$$
H^q\left({\mathbb P}_N-Z,{\mathcal O}_{{\mathbb
P}_N}^{\oplus{a\choose b}}(-b\ell+p)\right)=0
$$
in the assumption of the lemma and
$$
H^{q+1}\left({\mathbb P}_N-Z,{\rm Ker\,}\phi_b\right)=0
$$
in the induction hypothesis that
$$
H^q\left({\mathbb P}_N-Z,{\rm Ker\,}\phi_{b-1}\right)=0,
$$
which completes the induction argument.

\medbreak For $b=1$ we have
$$
H^1\left( {\mathbb P}_N-Z,{\rm Ker\,}\phi_1\right)=0,
$$
and the short exact sequence
$$
0\to {\rm Ker\,}\phi_1\to {\mathcal O}_{{\mathbb P}_N}^{\oplus a
}(-\ell+p)\to {\rm Im\,}\phi_1\to 0
$$
yields the surjectivity of
$$
\Gamma\left({\mathbb P}_N-Z,{\mathcal O}_{{\mathbb P}_N}^{\oplus a
}(-\ell+p)\right)\to \Gamma\left({\mathbb P}_N-Z,{\rm
Im\,}\phi_1\right).
$$
Hence
$$
\Gamma\left({\mathbb P}_N-Z,{\mathcal O}_{{\mathbb P}_N}^{\oplus a
}(-\ell+p)\right) \rightarrow\Gamma\left({\mathbb
P}_N-Z,\sum_{j=1}^a{\mathcal O}_{{\mathbb P}_N}(p)F_j\right)
$$
induced by $\varphi_1$ is surjective.\end{proof}

\begin{lemma}\label{lemma(3.3)} Let $w_0,w_1$ be two
transcendental variables (representing two local holomorphic
functions). Then
$$
w_0^{j+1}d^j\left({w_1\over w_0}\right)
$$
is a polynomial in the variables
$$d^\ell w_k\quad(0\leq\ell\leq j,\,k=0,1)$$ which is homogeneous of degree $j+1$ in all the
variables and of total weight $j$ in the differentials $d^\ell
w_k$ for $0\leq\ell\leq j$ and $k=0,1$ when the weight of $d^\ell
w_k$ is assigned to be $\ell$.
\end{lemma}

\begin{proof} The case $j=0$ of the claim is
clear. The induction process of the claim going from Step $j$ to
Step $j+1$ simply follows from
$$
\displaylines{w_0^{j+2}d^{j+1}\left({w_1\over w_0}\right)\cr=
w_0\left(d\left(w_0^{j+1}d^j\left({w_1\over
w_0}\right)\right)\right)-
(j+1)\left(dw_0\right)\left(w_0^{j+1}d^j\left({w_1\over
w_0}\right)\right)\cr }
$$
and the observations that

\medbreak\noindent (i) the differential of a homogeneous
polynomial in the variables $$d^\ell w_k\quad(0\leq\ell\leq
j,\,k=0,1)$$ is a homogeneous polynomial in the variables
$$d^\ell w_k\quad(0\leq\ell\leq
j,\,k=0,1)$$ of the same degree, and

\medbreak\noindent (ii) the differential of a polynomial in the
variables
$$d^\ell w_k\quad(0\leq\ell\leq
j,\,k=0,1)$$ which is of homogeneous weight $a$ is a polynomial in
the variables
$$d^\ell w_k\quad(0\leq\ell\leq
j,\,k=0,1)$$ which is of homogeneous weight $a+1$ when the weight
of $d^\ell w_k$ is assigned to be $\ell$.\end{proof}

\begin{lemma}\label{injectivity_pullback_map_jet_differential} {\rm (Injectivity of Pullback Map for Jet Differentials)}\ \  Let $1\leq k\leq n-1$ and let $f$ be a
polynomial of degree $\delta$ in inhomogeneous coordinates
$x_1,\cdots,x_n$ of ${\mathbb P}_n$ so that the zero-set of $f$
defines a complex manifold $X$ in ${\mathbb P}_n$. Let $Q$ be a
non identically zero polynomial in the variables $d^jx_\ell$
($0\leq j\leq k$,\,$1\leq\ell\leq n$). Assume that $Q$ is of
degree $m_0$ in $x_1,\cdots,x_n$ is $m_0$ and is of homogeneous
weight $m$ in the variables $d^jx_\ell$ ($1\leq j\leq
k$,\,$1\leq\ell\leq n$) when the weight of $d^jx_\ell$ is assigned
to be $j$. If $m_0+2m<\delta$, then $Q$ is not identically zero on
the space of $k$-jets of $X$. \end{lemma}

\begin{proof} Suppose $Q$ is identically
zero on the space of $k$-jets of $X$. We are going to derive a
contradiction.

\medbreak Since $Q$ is of homogeneous weight $m$ in the variables
$d^jx_\ell$ ($1\leq j\leq k$,\,$1\leq\ell\leq n$) when the weight
of $d^jx_\ell$ is assigned to be $j$, it follows that the degree
of $Q$ in the variables $d^jx_\ell$ ($1\leq j\leq
k$,\,$1\leq\ell\leq n$) is at most $m$. We introduce the
homogeneous coordinates
$$z_0,z_1,\cdots,z_n$$ of ${\mathbb P}_n$ so that
$$
x_j={z_j\over z_0}\quad(1\leq j\leq n).$$  Let $N=(k+1)(n+1)-1$
and relabel the variables $$d^j z_\ell\quad(0\leq j\leq
k,\,0\leq\ell\leq n)$$ as the $N+1$ homogeneous coordinates
$w_0,\cdots,w_N$ of ${\mathbb P}_N$. Let $P=z_0^{m_0+2m}Q$. Since
the degree of $Q$ in the variables $d^jx_\ell$ ($1\leq j\leq
k$,\,$1\leq\ell\leq n$) is at most $m$, by Lemma~\ref{lemma(3.3)}
we conclude that $P$ is a polynomial in the variables
$w_0,\cdots,w_N$ and is homogeneous of degree $m_0+2m$.

\medbreak We are going to apply Lemma~\ref{lemma(3.2)}. In our
application We set $m=n$.  The homogeneous polynomials
$F_1,\cdots,F_m$ of degree $\delta$ in the $N+1$ homogeneous
coordinates $w_0,\cdots,w_N$ of ${\mathbb P}_N$ are
$$z_0^\delta f, d\left(z_0^\delta
f\right),\cdots,d^k\left(z_0^\delta f\right).$$ The linear
subspace $Z$ in ${\mathbb P}_N$ is defined by
$z_0=z_1=\cdots=z_n=0$ which is of complex codimension $n+1$ in
${\mathbb P}_N$ and is therefore of complex dimension
$N-(n+1)=k(n+1)-1$.

\medbreak We know that, if $\hat Z$ is a subvariety of ${\mathbb P}_N$,
then for any Stein open subset $U$ of ${\mathbb P}_N$ the
cohomology group $H^q\left(U-\hat Z,{\mathcal O}_{{\mathbb
P}_N}\right)$ vanishes for $0\leq q\leq{\rm codim}_{{\mathbb
P}_N}\hat Z-2=n-1$, where ${\rm codim}_{{\mathbb
P}_N}\hat Z$ means the complex codimension of $Z$ in ${\mathbb P}_N$. Thus,
$$
H^q\left({\mathbb P}_N-Z,{\mathcal O}_{{\mathbb
P}_N}(\ell)\right)=0\quad\hbox{for }0\leq q\leq n-1.
$$
Since $Q$ is identically zero on the space of $k$-jets of $X$, it
follows that $P$ locally belongs to the ideal generated by
$$z_0^\delta f,
d\left(z_0^\delta f\right),\cdots,d^k\left(z_0^\delta f\right).$$
By Lemma~\ref{lemma(3.2)} with $a=k+1$, since
$$
{\rm codim}_{{\mathbb
P}_N}Z-2=n-1\geq k=a-1,
$$
we can write
$$
P=\sum_{j=0}^kg_j\,d^j\left(z_0^\delta f\right)
$$
for some homogeneous polynomials $g_0,\cdots,g_k$ of the variables
$$w_0,\cdots,w_N,$$ where the total degree of $g_j$ is
$m_0+2m-\delta$. We arrive at a contradiction, because
$m_0+2m-\delta$ is negative and a polynomial cannot have a
negative degree.\end{proof}

\bigbreak Now we count the number of unknowns and the number of
equations.

\begin{lemma}\label{lemma(3.5)} Let $X$ be a
hypersurface of degree $\delta$ in in ${\mathbb P}_n$. Let $S$ be
a hypersurface in $X$ defined by a homogeneous polynomial $g$ of
degree $s$ in the homogeneous coordinates of ${\mathbb P}_n$. Then
for $q\geq\delta+s+n$,
$$
{\rm dim\,}\Gamma\left(S,{\mathcal O}_S(q)\right)=
\sum_{j=1}^\delta\sum_{k=1}^s{n+q-j-k\choose n-2}.
$$
In particular,
$$
{\rm dim\,}\Gamma\left(S,{\mathcal O}_S(q)\right)\leq
{s\,\delta\,\left(n+q-2\right)^{n-2}\over(n-2)!}
$$
for $q\geq\delta+s+n$.
\end{lemma}

\begin{proof} First of all, for any
nonnegative integer $\ell$ we have
$$
\hbox{dim\,}\Gamma\left({\mathbb P}_n,{\mathcal O}_{{\mathbb
P}_n}(\ell)\right) ={\ell+n\choose\ell}={\ell+n\choose n},
$$
because it is equal to the number of possibilities of choosing
$\ell$ elements out of $n+1$ elements with repetition allowed
which is the same as choosing $\ell$ elements out of
$n+1+\ell-1=\ell+n$ elements without repetition. From the exact
sequence
$$
0\rightarrow{\mathcal O}_{{\mathbb
P}_n}(\ell){\stackrel{\phi_f}{\longrightarrow}} {\mathcal
O}_{{\mathbb P}_n}(\ell+\delta)\rightarrow{\mathcal
O}_X(\ell+\delta)\rightarrow
0\leqno{(\ref{lemma(3.5)}.1)}
$$
where $\phi_f$ is defined
by multiplication by $f$, it follows that
$$
\displaylines{ \Gamma\left({\mathbb P}_n,{\mathcal O}_{{\mathbb
P}_n}(\ell)\right)\rightarrow \Gamma\left({\mathbb P}_n,{\mathcal
O}_{{\mathbb P}_n}(\ell+\delta)\right)\rightarrow\cr
\Gamma\left(X,{\mathcal O}_X(\ell+\delta)\right)\rightarrow
H^1\left({\mathbb P}_n,{\mathcal O}_{{\mathbb
P}_n}(\ell)\right)=0\cr }
$$
is exact and
$$
\Gamma\left(X,{\mathcal O}_X(\ell+\delta)\right)=
{\Gamma\left({\mathbb P}_n,{\mathcal O}_{{\mathbb
P}_n}(\ell+\delta)\right)\big/ f{\Gamma\left({\mathbb
P}_n,{\mathcal O}_{{\mathbb P}_n}(\ell)\right)}}.
$$
Hence
$$
\hbox{dim\,}\Gamma\left(X,{\mathcal O}_X(\ell+\delta)\right)
={n+\ell+\delta\choose n}-{n+\ell\choose n}
$$
From
(\ref{lemma(3.5)}.1) we have the exact sequence
$$
H^p\left({\mathbb P}_n,{\mathcal O}_{{\mathbb
P}_n}(\ell+\delta)\right)\rightarrow H^p\left(X,{\mathcal
O}_X(\ell+\delta)\right)\rightarrow H^{p+1}\left({\mathbb
P}_n,{\mathcal O}_{{\mathbb P}_n}(\ell)\right).
$$
From the vanishing of $H^p\left({\mathbb P}_n,{\mathcal
O}_{{\mathbb P}_n}(\ell+\delta)\right)$ for $1\leq p<n$ it follows
that $H^p\left(X,{\mathcal O}_X(\ell+\delta)\right)=0$ for $1\leq
p<n-1$. From the exact sequence
$$
0\rightarrow{\mathcal O}_X(\ell)\stackrel{\phi_g}{\longrightarrow}
{\mathcal O}_X(\ell+s)\rightarrow{\mathcal O}_S(\ell+s)\rightarrow
0
$$
where $\phi_g$ is defined by multiplication by $g$ and from
$$
H^1(X,{\mathcal O}_X(\ell))=0
$$
for $n\geq 3$ it follows that
$$
\Gamma\left(S,{\mathcal O}_S(\ell+s)\right)=
{\Gamma\left(X,{\mathcal O}_X(\ell+s)\right)\big/
g{\Gamma\left(X,{\mathcal O}_X(\ell)\right)}}.
$$
Hence
$$
\displaylines{ \hbox{dim\,}\Gamma\left(S,{\mathcal
O}_S(\ell+s)\right)= {\rm dim\,}\Gamma\left(X,{\mathcal
O}_X(\ell+s)\right)-{\rm dim\,}\Gamma\left(X,{\mathcal
O}_X(\ell)\right)\cr = \left[{n+\ell+s\choose
n}-{n+\ell+s-\delta\choose n}\right] -\left[{n+\ell\choose
n}-{n+\ell-\delta\choose n}\right]\cr}
$$
for $\ell\geq\delta$. We are going to use the following identity
for binomial coefficients
$$
{a\choose b}-{a-1\choose b}={a\choose b-1}
$$
for $a-1\geq b\geq 1$. Then for $q\geq\delta+s+n$, we have
$$
\displaylines{ \hbox{dim\,}\Gamma\left(S,{\mathcal
O}_S(q)\right)=\left[{n+q\choose n}-{n+q-\delta\choose
n}\right]\cr -\left[{n+q-s\choose n}-{n+q-s-\delta\choose
n}\right]\cr =\sum_{j=1}^\delta\left[{n+q-j+1\choose
n}-{n+q-j\choose n}\right]\cr-\sum_{j=1}^q\left[{n+q-s-j+1\choose
n}-{n+q-s-j\choose n}\right] \cr=\sum_{j=1}^\delta{n+q-j\choose
n-1}-\sum_{j=1}^q{n+q-s-j\choose n-1} \cr
=\sum_{j=1}^\delta\left[{n+q-j\choose n-1}-{n+q-s-j\choose
n-1}\right]\cr =
\sum_{j=1}^\delta\sum_{k=1}^s\left[{n+q-j-k+1\choose
n-1}-{n+q-j-k\choose n-1}\right] \cr
=\sum_{j=1}^\delta\sum_{k=1}^s{n+q-j-k\choose n-2}.\cr}
$$
\end{proof}

\begin{lemma}\label{lemma(3.6)} Let $y_1,\cdots,y_r$ be
independent transcendental variables.  Let $1=n_1\leq
n_2\leq\cdots\leq n_r$ be integers.  Let $a_m$ be the number of
all monomials $y_1^{k_1}\cdots y_r^{k_r}$ such that $\sum_{j=1}^r
n_j k_j=m$.  Let $a_m$ be the number of elements in $A_m$.  Then
$$
{\left\lfloor{m\over n_r}\right\rfloor+r-1\choose r-1} \leq
a_m\leq {m+r-1\choose r-1},
$$
where $\lfloor u \rfloor$ denotes the largest integer not
exceeding $u$.
\end{lemma}

\begin{proof}  Let $A_m$ be the set of all
monomials $y_2^{k_2}\cdots y_r^{k_r}$ such that $\sum_{j=2}^r n_j
k_j=m$.  Since $n_1=1$, $A_m$ is the same as the set of all
monomials $y_1^{k_1}\cdots y_r^{k_r}$ such that $\sum_{j=1}^r n_j
k_j=m$ and $a_m$ is the number of elements of $A_m$. Let $B_m$ be
the set of all monomials $y_2^{k_2}\cdots y_r^{k_r}$ such that
$\sum_{j=2}^r k_j\leq \left\lfloor{m\over n_1}\right\rfloor$. Let
$C_m$ be the set of all monomials $y_2^{k_2}\cdots y_r^{k_r}$ such
that $\sum_{j=1}^r k_j\leq m$. Since
$$
\displaylines{ \sum_{j=2}^r k_j\leq\left\lfloor{m\over
n_r}\right\rfloor\cr \Longrightarrow\ \sum_{j=2}^r n_j k_j\leq
m\cr\Longrightarrow\ \sum_{j=2}^r k_j\leq m,\cr}
$$
it follows that
$$
B_m\subset A_m\subset C_m.
$$
Since the number of elements in $B_m$ is $${\left\lfloor{m\over
n_r}\right\rfloor+r-1\choose r-1}$$ and the number of elements in
$C_m$ is
$$
{m+r-1\choose r-1},
$$
the conclusion of the Lemma follows. \end{proof}

\begin{lemma}\label{lemma(3.7)} Let $f$ be a
polynomial of degree $\delta$ in the variables $x_1,\cdots,x_n$.
Let $\kappa_\ell$ be the smallest nonnegative integer such that
$\left(f_{x_1}\right)^{\kappa_\ell}d^\ell x_1$ can be expressed as
a polynomial $P_\ell$ of
$$\displaylines{x_1,x_2,\cdots,x_n,\cr
d^j x_r\quad(1\leq j\leq\ell,\,2\leq r\leq n)\cr }
$$
on the space of $\ell$-jets of the zero-set of $f$. Then
$\kappa_1=1$ and
$$
\kappa_\ell\leq 1+ \max\left\{\sum_{j=1}^{\ell-1}\kappa_js_j\
\Bigg| \ s_1+2s_2+\cdots+(\ell-1)s_{\ell-1}\leq\ell\right\}.
$$
Moreover, as a polynomial of the degree of $P_\ell$ in the
variables $x_1,\cdots,x_n$ is at most
$\kappa_\ell\left(\delta-1\right)$.  The integers $\kappa_\ell$ in
can be estimated by $\kappa_\ell\leq\ell!$
\end{lemma}

\begin{proof}  We use induction on $\ell$
for $\ell\geq 1$. The case $\ell=1$ is clear, because
$$
f_{x_1}dx_1=-\sum_{r=2}^n f_{x_r}dx_r
$$
and we can set $$P_1=-\sum_{r=2}^n f_{x_r}dx_r$$ whose degree in
$x_1,\cdots,x_r$ is obviously at most $\delta-1$.  From $d^\ell
f=0$ it follows that on the space of $\ell$-jets of the zero-set
of $f$ the jet differential $f_{x_1}d^\ell x_1$ can be written as
a polynomial $Q_\ell$ in
$$\displaylines{d^jx_r\quad(1\leq j\leq\ell-1,\,1\leq r\leq n),\cr
d^\ell x_2,\cdots,d^\ell x_n\cr}
$$ of weight $\leq\ell$ in
$$\displaylines{d^jx_r\quad(2\leq j\leq\ell-1,\,1\leq r\leq n),\cr
d^\ell x_2,\cdots,d^\ell x_n\cr}
$$ when $d^j x_r$ is given the weight
$j$.  As a polynomial in $dx_1,\cdots,d^{\ell-1}x_1$ the total
weight of $Q_\ell$ is no more than $\ell$.  As a polynomial in
$x_1,\cdots,x_n$ the degree of $Q_\ell$ is at most $\delta-1$.
Thus inductively on $\ell$ we conclude that we need only to
multiply $f_{x_1}d^\ell x_1$ by a power of $f_{x_1}$ not exceeding
$$
\max\left\{\sum_{j=1}^{\ell-1}\kappa_js_j\ \Bigg| \
s_1+2s_2+\cdots+(\ell-1)s_{\ell-1}\leq\ell\right\}
$$
to yield a polynomial $P_\ell$ of
$$\displaylines{x_1,x_2,\cdots,x_n,\cr
d^j x_r\quad(1\leq j\leq\ell,\,2\leq r\leq n)\cr }
$$ on the space of $\ell$-jets of
the zero-set of $f$.  Moreover, the degree of $P_\ell$ in
$x_1,\cdots,x_n$ is at most $\kappa_\ell\left(\delta-1\right)$.

\medbreak The integers $\kappa_\ell$ can be estimated by
$\kappa_\ell\leq\ell!$, because, when
$$r_1+2r_2+\cdots+(\ell-1)r_{\ell-1}\leq\ell,$$
we have $r_j\leq{\ell\over j}$, and from $\kappa_j\leq j!$ for
$1\leq j\leq\ell-1$ it follows that
$$
\sum_{j=1}^{\ell-1}\kappa_jr_j\leq
\sum_{j=1}^{\ell-1}\kappa_j\left({\ell\over
j}\right)\leq\ell\sum_{j=1}^{\ell-1}(j-1)!
<\ell\sum_{j=1}^{\ell-1}(\ell-2)!=\ell !.
$$
\end{proof}

\begin{proposition}\label{construction_jet_differential_as_polynomial} {\rm
(Jet Differential from Polynomial in Differentials of Inhomogeneous Coordinates)}\ \  Let $X$ be a
nonsingular hypersurface of degree $\delta$ in ${\mathbb P}_n$
defined by a polynomial $f(x_1,\cdots,x_n)$ of degree $\delta$ in
the affine coordinates $x_1,\cdots,x_n$ of ${\mathbb P}_n$.
Suppose $\epsilon$, $\epsilon^\prime$, $\theta_0$, $\theta$, and
$\theta^\prime$ are numbers in the open interval $(0,1)$ such that
$n\theta_0+\theta\geq n+\epsilon$ and
$\theta^\prime<1-\epsilon^\prime$. Then there exists an explicit
positive number $A=A(n,\epsilon,\epsilon^\prime)$ depending only
on $n$, $\epsilon$, and $\epsilon^\prime$ such that for
$\delta\geq A$ and any nonsingular hypersurface $X$ in ${\mathbb
P}_n$ of degree $\delta$ there exists a non identically zero
${\mathcal O}_{{\mathbb P}_n}(-q)$-valued holomorphic $(n-1)$-jet
differential $\omega$ on $X$ of total weight $m$ with
$q\geq\delta^{\theta^\prime}$ and $m\leq\delta^\theta$.  Here,
with respect to a local holomorphic coordinate system
$w_1,\cdots,w_{n-1}$ of $X$, the weight of $\omega$ is in the
variables $d^jw_\ell$ ($1\leq j\leq n-1$,\,$1\leq\ell\leq n-1$)
with the weight $j$ assigned to $d^jw_\ell$.  Moreover, for any
affine coordinates $x_1,\cdots,x_n$ of ${\mathbb P}_n$, when
$f_{x_1}=1$ defines in a nonsingular hypersurface in $X$, the
$(n-1)$-jet differential $\omega$ can be chosen to be of the form
${Q\over f_{x_1}-1}$, where $Q$ is a polynomial in
$$d^jx_1,\cdots,d^jx_n\quad(0\leq j\leq n-1)$$ which is of degree
$m_0=\left\lceil\delta^{\theta_0}\right\rceil$ in $x_1,\cdots,x_n$
and is of homogeneous weight
$m=\left\lceil\delta^\theta\right\rceil$ in
$$d^jx_1,\cdots,d^jx_n\quad(1\leq j\leq n-1)$$ when the weight of
$d^j x_\ell$ is assigned to be $j$.
\end{proposition}

\begin{proof} Let $x_1,\cdots,x_n$ and
$z_0,\cdots,z_n$ be respectively the homogeneous and inhomogeneous
coordinates of ${\mathbb P}_n$ so that $x_j={z_j\over z_0}$ for
$1\leq j\leq n$.  Let $f$ be a polynomial of degree $\delta$ in
$x_1,\cdots,x_n$ so that the zero-set of $f$ in ${\mathbb P}_n$ is
$X$.

\medbreak Consider a non identically zero polynomial $Q$ in
$$d^jx_1,\cdots,d^jx_n\quad(0\leq j\leq n-1)$$ which is of degree $m_0$
in $x_1,\cdots,x_n$ and is of homogeneous weight $m$ in
$$d^jx_1,\cdots,d^jx_n\quad(1\leq j\leq n-1)$$ when the weight of
$d^j x_\ell$ is assigned to be $j$. We impose the condition
$$m_0+2m<\delta$$ so that according to Lemma \ref{injectivity_pullback_map_jet_differential}
the pullback of $Q$ to the space of $(n-1)$-jets of $\{f=0\}$ is
not identically zero.  According to Lemma~\ref{lemma(3.6)} the
degree of freedom in the choice of the polynomial $Q$ is at least
$$
{m_0+n\choose n}{\left\lfloor m\over
n-1\right\rfloor+n(n-1)-1\choose n(n-1)-1},
$$
where the first factor
$$
{m_0+n\choose n}
$$
is the number of mononials of degree $\leq m_0$ in $n$ variables
$x_1,\cdots,x_n$ and the second factor
$$
{\left\lfloor m\over n-1\right\rfloor+n(n-1)-1\choose n(n-1)-1},
$$
is the number of mononials of {\it homogeneous} degree
$\left\lfloor m\over n-1\right\rfloor$ in the $n(n-1)$ variables
$$
d^j x_\ell\qquad\left(1\leq j\leq n-1,\, 1\leq\ell\leq n\right).
$$

The key point of this proof is that though we have all the
variables $d^jx_1,\cdots,d^jx_n$ for $0\leq j\leq n-1$, we do not
have to worry about the dependence resulting from the relations
$d^jf=0$ ($0\leq j\leq n-1$).

\medbreak Let $H_{{\mathbb P}_n}$ be the hyperplane of ${\mathbb
P}_n$ defined by $x_n=0$. We now want the meromorphic $(n-1)$-jet
differential defined by $Q$ to be holomorphic on $X$ and and,
moreover, to vanish at $X\cap H_{{\mathbb P}_n}$ of order $q$.
First of all, on the space of $(n-1)$-jets of $X$, we can use the
relation $d^j f=0$ ($1\leq j\leq n-1$) to eliminate the variables
$d^jx_1$ ($1\leq j\leq n-1$) by expressing $d^jx_1$ ($1\leq j\leq
n-1$) in terms of
$$
d^j x_\ell\quad(1\leq j\leq n-1,\,2\leq\ell\leq n).
$$
To do this, according to Lemma~\ref{lemma(3.7)} we can multiply
$Q$ by $\left(f_{x_1}\right)^{\tilde N}$ with $\tilde
N=2m\sum_{j=1}^{n-1}\kappa_j$, because, in a monomial of weight
$m$, the degree of $d^jx_\ell$ ($1\leq j\leq n-1$, $1\leq\ell\leq
n$) is at most $\left\lfloor{m\over j}\right\rfloor$ and
$(j+1)\left\lfloor{m\over j}\right\rfloor\leq 2m$. Since
$\kappa_j\leq j!$, it follows that $\tilde N\leq(n-1)!\,2m$.  Let
$N=(n-1)!\,2m$. The degree of $\left(f_{x_1}\right)^N Q$ in
$x_1,\cdots,x_n$ is now $m_0+N(\delta-1)$ and the weight of
$\left(f_{x_1}\right)^N Q$ in $d^jx_\ell$ ($1\leq j\leq n-1$,\,
$2\leq\ell\leq n$) is homogeneous and equal to $m$ when the weight
of $d^jx_\ell$ is assigned to be $j$.

\medbreak We let $S$ be the divisor in $X$ defined by
$f_{x_1}-1=0$.  The hypersurface in ${\mathbb P}_n$ defined by
$f_{x_1}-1=0$ is of degree $\delta-1$.  We observe that for a
generic polynomial $f$ of degree $\delta$, the divisor $S$ in $X$
is nonsingular, because it is the case when $f$ equals the Fermat
hypersurface
$$
F=\sum_{j=1}^n x_j^\delta-1.
$$
Then
$$
F_{x_1}=\delta x_1^{\delta-1}
$$
and the $2\times n$ matrix
$$
\left[\ \begin{matrix}x_1^{\delta-1}&x_2^{\delta-1}&\cdots&x_n^{\delta-1}\\
x_1^{\delta-2}&0&\cdots&0\\ \end{matrix}\ \right]
$$
whose rows are nonzero multiplies of the gradients of $F$ and
$F_{x_1}$ has rank $2$ unless either $x_1=0$ or
$x_2=\cdots=x_n=0$, which is impossible, because on $S$ one has
$\left|x_1\right|=\delta^{-1\over\delta-1}$ from $F_{x_1}=1$ and
the condition $x_2=\cdots=x_n=0$ implies
$\left|x_1\right|=1\not=\delta^{-1\over\delta-1}$ when $F=0$. To
prove this Lemma we need only prove it for a generic $f$ and then
remove the genericity assumption for $f$ by using the
semi-continuity of the dimension of the space of global
holomorphic sections over a fiber in a holomorphic family of
compact complex manifolds and a holomorphic vector bundle.

\medbreak The polynomial $Q$ defines a meromorphic $(n-1)$-jet
differential on $X$ which we again denote by $Q$.  We now count
the pole order of the jet differential $Q$ on $X$ at $X\cap
H_{{\mathbb P}_n}$. For the counting of this pole order, we
introduce another set of inhomogeneous coordinates
$\zeta_1,\cdots,\zeta_n$ of ${\mathbb P}_n$ defined by
$$
\zeta_1={x_1\over x_n},\, \cdots,\,\zeta_{n-1}={x_{n-1}\over
x_n},\,\zeta_n={1\over x_n}
$$
so that
$$
x_1={\zeta_1\over\zeta_n},\,\cdots,\,x_{n-1}={\zeta_{n-1}\over
\zeta_n},\,x_n={1\over\zeta_n}.
$$
Since by Lemma~\ref{lemma(3.3)}
$$
{\zeta_n}^{j+1}d^jx_\ell={\zeta_n}^{j+1}d^j\left({\zeta_\ell\over\zeta_n}\right)
\quad(1\leq\ell\leq n-1)
$$
and
$$
{\zeta_n}^{j+1}d^jx_n={\zeta_n}^{j+1}d^j\left({1\over\zeta_n}\right)
$$
are polynomials in $d^k\zeta_r$ ($0\leq k\leq j$,\,$1\leq r\leq
n$).  Thus ${\zeta_n}^{2m}Q$ is a polynomial in $d^k\zeta_r$
($0\leq k\leq j$,\,$1\leq r\leq n$).  The pole order of the jet
differential $Q$ at $X\cap H_{{\mathbb P}_n}$ is at most $m_0+2m$.

\medbreak We are going to show that we can choose the coefficients
of the polynomial $Q$ so that the $(n-1)$-jet differential $Q$ is
zero at points of $S$.  This would imply that the $(n-1)$-jet
differential $${1\over f_{x_1}-1}\,Q$$ is holomorphic on $X$ and
vanishes to order $\delta-m_0-2m$ at $X\cap H_{{\mathbb P}_n}$.
The reason is the following.  For some proper subvariety $Z$ of
$X\cap H_{{\mathbb P}_n}$ the function
$\zeta_0^{\delta-1}\left(f_{x_1}-1\right)$ is holomorphic and
nowhere at points of $X\cap H_{{\mathbb P}_n}-Z$.  Thus
$${1\over f_{x_1}-1}\,Q={1\over
\zeta_0^{\delta-1}\left(f_{x_1}-1\right)}\left(\zeta_0^{\delta-1-2m}\right)
\left(\zeta_0^{2m}\,Q\right)
$$
is holomorphic on $X-Z$ and vanishes to order at least
$\delta-1-2m$ along $X\cap H_{{\mathbb P}_n}-Z$.  What we want
follows from Hartogs' extension theorem because $Z$ is of complex
codimension at least $2$ in $X$.

\medbreak Now on $J_{n-1}\left(X\right)\big|\,S$ (which is the
part of the space of $(n-1)$-jets of $X$ lying over $S$) the jet
differential $Q$ equals to the jet differential
$$
\left(f_{x_1}\right)^N Q
$$
because $f_{x_1}=1$ holds on $S$. Since the degree of
$\left(f_{x_1}\right)^N Q$ in $x_1,\cdots,x_n$ is
$m_0+N(\delta-1)$ and the weight of $\left(f_{x_1}\right)^N Q$ in
$d^k z_r$ ($1\leq k\leq j$,\,$2\leq r\leq n$) is homogeneous and
equal to $2m$, it follows from Lemma~\ref{lemma(3.5)} that the
number of linear equations, with the coefficients of $Q$ as
unknowns, needed for $\left(f_{x_1}\right)^N Q$ to vanish at all
points $S$ is no more than the product
$$
{\left(\delta-1\right)\delta\left(m_0+N\left(\delta-1\right)\right)^{n-2}\over(n-2)!}
{m+(n-1)^2-1\choose(n-1)^2-1}.
$$
For the existence of a nontrivial $Q$ with the required vanishing
at all points of $S$, it suffices to have
$$
\displaylines{{m_0+n\choose n}{\left\lfloor m\over
n-1\right\rfloor+n(n-1)-1\choose n(n-1)-1}\cr>
{\left(\delta-1\right)\delta\left(m_0+N\left(\delta-1\right)\right)^{n-2}\over(n-2)!}
{m+(n-1)^2-1\choose(n-1)^2-1}.\cr }$$ In particular, it suffices
that
$$
{\left(m_0+1\right)^n\left({m\over n-1}\right)^{n(n-1)-1}\over
n!\left(n(n-1)-1\right)!} $$ is greater than $$
{\left[\left(\delta-1\right)\delta\left(m_0+(n-1)!\,2m\left(\delta-1\right)\right)^{n-2}\right]
\left[\left(m+(n-1)^2-1\right)^{(n-1)^2-1}\right]\over
(n-2)!\left((n-1)^2-1\right)!}.$$

\medbreak We choose
$m_0=\left\lceil\delta^{\theta_0}\right\rceil$,
$m=\left\lceil\delta^\theta\right\rceil$, and
$q=\left\lfloor\delta^{\theta^\prime}\right\rfloor$. Since the
three positive numbers $\theta_0$, $\theta$, and $\theta^\prime$
are all strictly less than $1$, measured in terms of powers of
$\delta$ as $\delta$ becomes dominantly large, the order of
$$
{\left(m_0+1\right)^n\left({m\over n-1}\right)^{n(n-1)-1}\over
n!\left(n(n-1)-1\right)!} $$ is at least
$$\delta^{n\theta_0+\left(n(n-1)-1\right)\theta}$$ and the order
of
$$
{\left[\left(\delta-1\right)\delta\left(m_0+(n-1)!\,2m\left(\delta-1\right)\right)^{n-2}\right]
\left[\left(m+(n-1)^2-1\right)^{(n-1)^2-1}\right]\over
(n-2)!\left((n-1)^2-1\right)!}$$ is at most
$$\delta^{2+(n-2)(1+\theta)+\left((n-1)^2-1\right)\theta}.$$
Since by assumption $n\theta_0+\theta\geq n+\epsilon$, it follows
that
$$\left[n\theta_0+\left(n(n-1)-1\right)\theta\right]-\left[
2+(n-2)(1+\theta)+\left((n-1)^2-1\right)\theta\right]\geq\epsilon.
$$
So there exists a positive number $A$ depending only on $n$ and
$\epsilon$ such that
$$
{\left(m_0+1\right)^n\left({m\over n-1}\right)^{n(n-1)-1}\over
n!\left(n(n-1)-1\right)!} $$ is greater than $$
{\left[\left(\delta-1\right)\delta\left(m_0+(n-1)!\,2m\left(\delta-1\right)\right)^{n-2}\right]
\left[\left(m+(n-1)^2-1\right)^{(n-1)^2-1}\right]\over
(n-2)!\left((n-1)^2-1\right)!}$$ when $\delta\geq A$. We can also
assume that $A$ is chosen so that
$\delta-\delta^{\theta_0}-2\delta^\theta\geq\delta^{\theta^\prime}$
for $\delta\geq A$ to make sure that the $(n-1)$-jet differential
$${1\over f_{x_1}-1}\,Q$$ is holomorphic on $X$ and vanishes
to order at least $q$ at $X\cap H_{{\mathbb P}_n}$.\end{proof}

\begin{remark} Proposition \ref{construction_jet_differential_as_polynomial} is the same as Proposition 4.6 on p.446 of \cite{Si02} and also the same as Proposition 2 on p.558 of \cite{Si04}.
\end{remark}

\bigbreak

\section{\sc Hyperbolicity from Slanted Vector Fields and No Common Zeroes for Jet Differentials on Generic Hypersurface}

\bigbreak In Proposition \ref{construction_jet_differential_as_polynomial} a holomorphic $(n-1)$-jet differential $\omega$ on a hypersurface $X$ vanishing on an ample divisor of $X$ is constructed as a quotient $\frac{Q}{f_{x_1}-1}$, where $Q$ is a polynomial in
$d^jx_1,\cdots,d^jx_n$ for $0\leq j\leq n-1)$ and actually is a meromorphic $(n-1)$-jet differential on the projective space ${\mathbb P}_n$ where the hypersurface $X$ lies.    The construction depends on the choice of the affine coordinate system $x_1,\cdots,x_n$ of the affine part ${\mathbb C}^n$ of ${\mathbb P}_n$.  Now we apply the construction to the hypersurface $X^{(\alpha)}$ parametrized by $\alpha\in{\mathbb P}_N$ (instead of to $X$) and we denote the polynomial $Q$ of the differentials of affine coordinates by $Q(\alpha,x,dx,\dots,d^{n-1}x)$, where $x$ means $(x_1,\cdots,x_n)$ and $d^jx$ means $(d^jx_1,\cdots,d^jx_n)$ for $1\leq j\leq n-1$.  As a function of $\alpha$, the meromorphic $(n-1)$-jet differential $Q(\alpha,x,dx,\dots,d^{n-1}x)$ on the projective space ${\mathbb P}_n$ is meromorphic in the variable $\alpha\in{\mathbb P}_N$.  We now regard $Q(\alpha,x,dx,\dots,d^{n-1}x)$ as defined over ${\mathbb P}_n\times{\mathbb P}_N$, which for fixed $\alpha\in{\mathbb P}_N$ is a meromorphic $(n-1)$-jet differential on ${\mathbb P}_n\times\{\alpha\}$.  When we replace $X$ by $X^{(\alpha)}$, we denote the function $f_{x_1}-1$ by $F(\alpha,x)$.  We regard $F(\alpha,x)$ as a meromorphic function on ${\mathbb P}_n\times{\mathbb P}_N$.  The quotient $\frac{Q(\alpha,x,dx,\dots,d^{n-1}x)}{F(\alpha,x)}$ on ${\mathcal X}$ defines on every $X^{(\alpha)}$ a holomorphic $(n-1)$-jet differential which vanishes on an ample divisor of $X^{(\alpha)}$, when $\alpha$ is outside some proper subvariety of ${\mathbb P}_N$.

\medbreak The construction of $Q(\alpha,x,dx,\dots,d^{n-1}x)$ depends on the choice of the affine coordinate system $x_1,\cdots,x_n$ of the affine part ${\mathbb C}^n$ of ${\mathbb P}_n$.  We can get different meromorphic $k$-jet differentials $Q(\alpha,x,dx,\dots,d^{n-1}x)$ by using different affine coordinate systems $x_1,\cdots,x_n$ of ${\mathbb C}^n$ in the construction of $Q(\alpha,x,dx,\dots,d^kx)$.  Equivalently, instead of doing a new construction of $Q(\alpha,x,dx,\dots,d^{n-1}x)$ by using a new affine coordinate system, we can use a biholomorphism of ${\mathbb P}_n$ to pull the original $Q(\alpha,x,dx,\dots,d^{n-1}x)$ back in the following way.

\medbreak A choice of a different affine coordinate system $x_1,\cdots,x_n$ of ${\mathbb C}^n$ is the same as choosing a corresponding biholomorphism $\sigma:{\mathbb P}_n\to{\mathbb P}_n$ (which preserves the infinity hyperplane ${\mathbb P}_{n-1}$). This biholomorphism $\sigma:{\mathbb P}_n\to{\mathbb P}_n$ induces a map $\tau_\sigma:{\mathbb P}_N\to{\mathbb P}_N$ such that $(\sigma,\tau_\sigma):{\mathbb P}_n\times{\mathbb P}_N\to{\mathbb P}_n\times{\mathbb P}_N$ maps the universal hypersurface ${\mathcal X}$ to itself with $X^{(\alpha)}$ being mapped to $X^{(\tau_\sigma(\alpha))}$ for $\alpha\in{\mathbb P}_N$.  The new $Q(\alpha,x,dx,\dots,d^{n-1}x)$ constructed by using the new affine coordinate system is the same as the $(n-1)$-jet differential $Q\left(\tau_\sigma(\alpha),\sigma(x), d\sigma(x),\cdots,d^k\sigma(x)\right)$, which is obtained by pulling back the original $Q(\alpha,x,dx,\dots,d^{n-1}x)$ by using $\sigma$ and $\tau_\sigma$.  We use $\sigma^*Q$ to denote $$Q\left(\tau_\sigma(\alpha),\sigma(x), d\sigma(x),\cdots,d^{n-1}\sigma(x)\right)$$ and use $\sigma^*F$ to denote $F\left(\tau_\sigma(\alpha),\sigma(x)\right)$.
Of course, to get a holomorphic $(n-1)$-jet differential on the fiber $X^{(\tau_\sigma(\alpha))}$ of ${\mathcal X}$ (for $\tau_\sigma(\alpha)$ outside some proper subvariety of ${\mathbb P}_N$) we have to use
$$
\frac{\sigma^*Q}{\sigma^*F}=\frac{1}{F\left(\tau_\sigma(\alpha),\sigma(x)\right)}\,\omega\left(\tau_\sigma(\alpha),\sigma(x), d\sigma(x),\cdots,d^{n-1}\sigma(x)\right).
$$In the above discussion we can also use biholomorphisms $\sigma:{\mathbb P}_n\to{\mathbb P}_n$ which may not preserve the infinity hyperplane ${\mathbb P}_{n-1}$.

\begin{proposition}\label{basepoint_freeness_jet_differential} {\it(No Common Zeroes on Generic Hypersurface for Jet Differentials Constructed from Different Affine Coordinates)}.  Let ${\mathcal Z}$ be the set of points $y$ of ${\mathcal X}$ such that $$\sigma^*Q=Q\left(\tau_\sigma(\alpha),\sigma(x), d\sigma(x),\cdots,d^{n-1}\sigma(x)\right)$$ vanishes at $y$ or has a pole at $y$ for every biholomorphism $\sigma:{\mathbb P}_n\to{\mathbb P}_n$.  Let ${\mathcal Z}^\prime$ be the set of points $y$ of ${\mathcal X}$ such that $F\left(\tau_\sigma(\alpha),\sigma(x)\right)$ vanishes at $y$ or has a pole at $y$ for every biholomorphism $\sigma:{\mathbb P}_n\to{\mathbb P}_n$. Then the image ${\rm pr}_2\left({\mathcal Z}\cup{\mathcal Z}^\prime\right)$ of ${\mathcal Z}\cup{\mathcal Z}^\prime$ under the natural projection ${\rm pr}_2:{\mathbb P}_n\times{\mathbb P}_N\to{\mathbb P}_N$ onto the second factor is a proper subvariety of ${\mathbb P}_N$.
\end{proposition}

\begin{proof}  Suppose the contrary and we going to derive a contradiction.  For technical reasons it is easier to present the proof by fixing some point in ${\mathbb P}_n$ as the origin $0$ of some inhomogeneous coordinates $x_1,\cdots,x_n$ of the affine part ${\mathbb C}^n$ of ${\mathbb P}_n$ and consider all hypersurfaces in ${\mathbb P}_n$ of degree $\delta$ which contains the origin $0$ of $x_1,\cdots,x_n$.  This means that we focus only on those hypersurfaces $X^{(\alpha)}$ whose defining functions $f^{(\alpha)}$ have zero constant terms when expressed in terms of the inhomogeneous coordinates $x_1,\cdots,x_n$.  In other words, we focus only on a hyperplane ${\mathbb P}_{N-1}^{(0)}$ of the full moduli space ${\mathbb P}_N$.  The union of all $X^{(\alpha)}$ with $\alpha\in{\mathbb P}_{N-1}^{(0)}$ is a hypersurface ${\mathcal X}_0$ of the universal hypersurface ${\mathcal X}$.

\medbreak Since every hypersurface in ${\mathbb P}_n$ can be transformed by a linear transformation of ${\mathbb P}_n$ to some hypersurface which contains the origin $0$ of $x_1,\cdots,x_n$, from the assumption of the failure of the conclusion of the Proposition it follows that for some $\alpha^{(0)}\in{\mathbb P}_{N-1}^{(0)}$ and some open neighborhood $U$ in ${\mathbb P}_{N-1}^{(0)}$ there exists some local holomorphic section $\rho:U\to{\mathbb P}_n\times U$ of the trivial bundle ${\mathbb P}_n\times{\mathbb P}_{N-1}^{(0)}\to{\mathbb P}_{N-1}^{(0)}$ over $U$ such that the image $\rho(U)$ is contained in $\left({\mathcal Z}\cup{\mathcal Z}^\prime\right)\cap{\mathcal X}_0$ and $\rho(U)$ is contained in the subset ${\mathbb C}^n\times U$ of ${\mathbb P}_n\times U$.   For $\alpha\in U$ let $\rho(\alpha)=(\hat\rho(\alpha),\alpha)$ with $\hat\rho(\alpha)\in{\mathbb P}_n$.

\medbreak Let us first do one reduction to make $\hat\rho(\alpha)$ equal to the origin $0$ of the affine part ${\mathbb C}^n$ of ${\mathbb P}_n$ for all $\alpha\in U$.  We first present this reduction from the viewpoint of analysis involving the Jacobian determinant of a change of variables.  Then we give a more geometric explanation for it in (\ref{basepoint_freeness_jet_differential}.2).

\bigbreak By straightforwardly and explicitly computing the Jacobian determinant of the variable change in the affine part of ${\mathbb C}^{N-1}$ of the moduli space ${\mathbb P}_{N-1}^{(0)}$ at a generic point of ${\mathbb P}_{N-1}^{(0)}$ in $U$ (as explained below after (\ref{basepoint_freeness_jet_differential}.1)), we can find some nonempty open subset $\tilde U$ of $U$ and an $\alpha$-dependent affine coordinate change $x=\hat\rho(\alpha)+A(\alpha)y$ in ${\mathbb C}^n$ for $\alpha\in\tilde U$ which satisfies the following property (\ref{basepoint_freeness_jet_differential}.1), where (i) $x$ denotes the column $n$-vector whose entries are the affine variables $x_1,\cdots,x_n$, (ii) $y$ denotes the column $n$-vector whose entries are the affine variables $y_1,\cdots,y_n$, and (iii) $A(\alpha)$ is a nonsingular $n\times n$ matrix depending holomorphically on $\alpha\in\tilde U$.

\bigbreak\noindent(\ref{basepoint_freeness_jet_differential}.1) If $\alpha\mapsto\beta=\Phi_A(\alpha)$ denotes the holomorphic map from $\tilde U$ to ${\mathbb P}_{N-1}^{(0)}$ (for the choice $\alpha\mapsto A(\alpha)$) such that $f^{(\alpha)}(x_1,\cdots,x_n)=f^{(\beta)}(y_1,\cdots,y_n)$, then $\Phi_A$ gives a biholomorphic map between $\tilde U$ and the open subset $\Phi_A(\tilde U)$ of ${\mathbb C}^{N-1}$.

\bigbreak We now remark on how to choose the $n\times n$ nonsingular matrix $A(\alpha)=\left(A_{jk}(\alpha)\right)_{j,k=1}^n$ with holomorphic functions $A_{jk}(\alpha)$ as entries.  Without loss of generality we can assume that $U$ is an
open subset in the affine part ${\mathbb C}^{N-1}$ of ${\mathbb P}_{N-1}^{(0)}$ given by the coefficient of $x_1^\delta$ in $f^{(\alpha)}\left(x_1,\cdots,x_n\right)$ being nonzero. Write down the coordinates $\alpha_{\nu_1,\cdots,\nu_n}$ of $\alpha\in{\mathbb C}^{N-1}$ for $1\leq\nu_1+\cdots+\nu_n\leq\delta$ with $\nu_1\not=\delta$ from
$$
f^{(\alpha)}\left(x_1,\cdots,x_n\right)=x_1^\delta+\sum_{1\leq\nu_1+\cdots+\nu_n\leq\delta,\atop\nu_1
\not=\delta}
\alpha_{\nu_1,\cdots,\nu_n}x_1^{\nu_1}\cdots x_n^{\nu_n}
$$ and similarly the coordinates $\beta_{\nu_1,\cdots,\nu_n}$ of $\beta\in{\mathbb C}^{N-1}$ for $1\leq\nu_1+\cdots+\nu_n\leq\delta$ with $\nu_1\not=\delta$.  From $f^{(\alpha)}(x_1,\cdots,x_n)=f^{(\beta)}(y_1,\cdots,y_n)$ we can explicitly express $\beta=\Phi_A(\alpha)$ in terms of the single given $n$-vector-valued holomorphic function $\hat\rho(\alpha)$ and the $n^2$ unknown holomorphic functions $A_{jk}(\alpha)$ for $1\leq j,k\leq n$.  Take a point $\alpha^*$ of $U$ such that all its $N-1$ coordinates $\alpha_{\nu_1,\cdots,\nu_n}^*$ are nonzero for $1\leq\nu_1+\cdots+\nu_n\leq\delta$ with $\nu_1\not=\delta$.  When we express the $(N-1)$-form
$$
\bigwedge_{1\leq\nu_1+\cdots+\nu_n\leq\delta,\atop\nu_1\not=\delta}d\beta_{\nu_1,\cdots,\nu_n}
$$
at $\alpha^*$ as a constant $C$ times the $(N-1)$-form
$$
\bigwedge_{1\leq\nu_1+\cdots+\nu_n\leq\delta,\atop\nu_1\not=\delta}d\alpha_{\nu_1,\cdots,\nu_n}
$$
at $\alpha^*$, it is easy to see that we can generically choose the coefficients of the $n^2$ power series $A_{jk}(\alpha)$ in the variables $\alpha_{\nu_1,\cdots,\nu_n}$ (for $1\leq\nu_1+\cdots+\nu_n\leq\delta$ with $\nu_1\not=\delta$) to achieve $C\not=0$.  Then $\tilde U$ can be chosen to be a sufficiently small open neighborhood of $\alpha^*$ in $U$.

\bigbreak\noindent(\ref{basepoint_freeness_jet_differential}.2) We now more geometrically explain the reason why a generic choice of $\alpha\mapsto A(\alpha)$ is possible to yield the statement (\ref{basepoint_freeness_jet_differential}.1).  We introduce the equivalence relation on the parameter space ${\mathbb P}_{N-1}^{(0)}$ such that $\alpha$ is equivalent to $\alpha^\prime$ if some ${\mathbb C}$-linear transformation of the affine part ${\mathbb C}^n$ of ${\mathbb P}_n$ sends $X^{(\alpha)}$ to $X^{(\alpha^\prime)}$.    Let $W$ be the space of all equivalence classes with the quotient map $\pi_0:{\mathbb P}_{N-1}^{(0)}\to W$ whose generic fiber is the general linear group $GL(n,{\mathbb C})$.  We have the following commutative diagram
$$
\begin{matrix}U&\hookrightarrow&{\mathbb P}_{N-1}^{(0)}&\stackrel{\pi_0}{\longrightarrow}&W\cr
\Phi_A\downarrow\qquad&&&&\|\cr
{\mathbb C}^{N-1}&\hookrightarrow&{\mathbb P}_{N-1}^{(0)}&\stackrel{\pi_0}{\longrightarrow}&W.\cr
\end{matrix}
$$
Since the fiber of the quotient map $\pi_0:{\mathbb P}_{N-1}^{(0)}\to W$ over a generic point of $W$ is the general linear group $GL(n,{\mathbb C})$, it follows that if $\alpha\mapsto A(\alpha)$ is replaced by $\alpha\mapsto A(\alpha)B(\alpha)$ for some generic $GL(n,{\mathbb C})$-valued holomorphic function $\alpha\mapsto B(\alpha)$, the map
$\Phi_{AB}$ from $U$ to ${\mathbb C}^{N-1}$ is locally biholomorphic at a generic point of $U$.

\bigbreak Now that we have (\ref{basepoint_freeness_jet_differential}.1) with a good generic choice of $\alpha\mapsto A(\alpha)$, by replacing $U$ by $\tilde U$ and $\alpha\mapsto\rho(\alpha)=\left(\hat\rho(\alpha),\alpha\right)$ for $\alpha\in U$ by $\alpha\mapsto\left(0,\alpha\right)$ for $\alpha\in\tilde U$, we can assume without loss of generality that $\hat\rho(\alpha)=0$ for $\alpha\in U$.

\medbreak Without loss of generality we can assume that

\medbreak\noindent(i) $U$ is the open ball $B_{N-1}\left(\alpha^{(0)},r_0\right)$ of some positive radius $r_0>0$ centered at $\alpha^{(0)}$ in some affine part ${\mathbb C}^{N-1}$ of the moduli space ${\mathbb P}_{N-1}^{(0)}$,

\medbreak\noindent(ii) some open neighborhood $W$ of $\rho(U)$ in ${\mathcal X}_0\cap\left({\mathbb C}^n\times U\right)$ is biholomorphic to $G\times U$ for some open subset $G$ of ${\mathbb C}^{n-1}$ under a biholomorphism $\phi_W$ between the fiber bundle ${\rm pr}_2:W\to U$ (where ${\rm pr}_2$ induced by the natural projection ${\mathbb P}_n\times{\mathbb P}_{N-1}^{(0)}\to{\mathbb P}_{N-1}^{(0)}$ onto the second factor) and the trivial fiber bundle ${\rm pr}_G:G\times U\to U$ with ${\rm pr}_G$ being the natural projection onto the second factor, and

\medbreak\noindent(iii) $\left({\mathcal Z}\cup{\mathcal Z}^\prime\right)\cap X^{(\alpha^{(0)})}$ is a proper subvariety of $X^{(\alpha^{(0)})}$.

\medbreak\noindent We now take a sequence of points $$y^{(j)}\in \left(W\cap X^{(\alpha^{(0)})}\right)-\left({\mathcal Z}\cup{\mathcal Z}^\prime\right)\quad{\rm for}\ \ j\in{\mathbb N}$$ which approaches $\rho\left(\alpha^{(0)}\right)$ in $X^{(\alpha^{(0)})}$. Let $0<r<r_0$.  Since the point $\rho\left(\alpha^{(0)}\right)$ of $X^{(\alpha^{(0)})}$ is represented by $\rho\left(\alpha^{(0)}\right)=\left(0,\alpha^{(0)}\right)\in{\mathbb C}^{n-1}\times U$ in terms of the affine coordinates of ${\mathbb C}^{n-1}$ and $U\subset{\mathbb C}^{N-1}$, by using the biholomorphism $\phi_U$ between the two fiber bundles ${\rm pr}_2:W\to U$ and ${\rm pr}_G:G\times U\to U$, for each $j\in{\mathbb N}$ we can construct a holomorphic section $\rho_j:U\to{\mathcal X}_0$ with the image of $\rho_j(\alpha)=\left(\hat\rho_j(\alpha),\alpha\right)\in W$ for $\alpha\in U$ such that $\rho_j\left(\alpha^{(0)}\right)=y^{(j)}$ and
$$
\varepsilon_j=\sup_{\alpha\in B_{N-1}\left(\alpha^{(0)},r\right)
}\left|\hat\rho_j(\alpha)\right|_{{\mathbb C}^{n-1}}
$$
approaches $0$ as $j\to\infty$, where $B_{N-1}\left(\alpha^{(0)},r\right)$ is the open ball of radius $r$ in ${\mathbb C}^{N-1}$ centered at $\alpha^{(0)}$ and the norm $\left|\hat\rho_j(\alpha)\right|_{{\mathbb C}^{n-1}}$ is the distance between the two points $\hat\rho_j(\alpha)$ and $0$ in ${\mathbb C}^{n-1}$ with respect to the Euclidean metric of ${\mathbb C}^{n-1}$.

\medbreak For $\alpha\in U$ and $j\in{\mathbb N}$ let $\sigma_{\alpha,j}:{\mathbb P}_n\to{\mathbb P}_n$ be the biholomorphism of ${\mathbb P}_n$ whose restriction to ${\mathbb C}^n$ is the translation in ${\mathbb C}^n$ which sends the origin $0$ of ${\mathbb C}^n$ to the point $\hat\rho_j(\alpha)$ of ${\mathbb C}^n$.  The biholomorphism $\sigma_{\alpha,j}$ of ${\mathbb P}_n$ pulls back the hypersurface $X^{(\alpha)}$ of ${\mathbb P}_n$ to the hypersurface $X^{\left(\tau_{\sigma_{\alpha,j}}^{-1}(\alpha)\right)}$ of ${\mathbb P}_n$ for $\alpha\in U$ so that the point $\hat\rho_j(\alpha)$ of the hypersurface $X^{(\alpha)}$ is pulled back to the origin point of the hypersurface $X^{\left(\tau_{\sigma_{\alpha,j}}^{-1}(\alpha)\right)}$.  From $0<r<r_0$ and from explicitly expressing $\sigma_{\alpha,j}$ in terms of $\alpha\in{\mathbb C}^{N-1}$ and $\hat\rho_j(\alpha)\in{\mathbb C}^{n-1}$ we conclude that there exists some positive number $M$ independent of $j\in{\mathbb N}$ such that
$$
\sup_{\alpha\in B_{N-1}\left(\alpha^{(0)},r\right)}\left|\tau_{\sigma_{\alpha,j}}^{-1}(\alpha)-\alpha\right|_{{\mathbb C}^{N-1}}\leq M\varepsilon_j
$$
for $j\in{\mathbb N}$, where the norm $\left|\tau_{\sigma_{\alpha,j}}^{-1}(\alpha)-\alpha\right|_{{\mathbb C}^{N-1}}$ is the distance between the two points $\tau_{\sigma_{\alpha,j}}^{-1}(\alpha)$ and $\alpha$ in ${\mathbb C}^{N-1}$ with respect to the Euclidean metric of ${\mathbb C}^{N-1}$.  Choose $\hat j\in{\mathbb N}$ such that $M\varepsilon_j<r_0-r$ for $j\geq\tilde j$.  Let $\psi_j:B_{N-1}\left(\alpha^{(0)},r\right)\to B_{N-1}\left(\alpha^{(0)},r_0\right)$ be defined by $\psi_j(\alpha)=\tau_{\sigma_{\alpha,j}}^{-1}(\alpha)$ for $\alpha\in B_{N-1}\left(\alpha^{(0)},r\right)$ and $j\in{\mathbb N}$ with $j\geq\tilde j$.

\medbreak Since the holomorphic map $\psi_j:B_{N-1}\left(\alpha^{(0)},r\right)\to B_{N-1}\left(\alpha^{(0)},r_0\right)$ approaches the inclusion map $B_{N-1}\left(\alpha^{(0)},r\right)\hookrightarrow B_{N-1}\left(\alpha^{(0)},r_0\right)$ uniformly on $B_{N-1}\left(\alpha^{(0)},r\right)$ as $j\to\infty$, it follows that the first-order partial derivatives of $\psi_j$ converges uniformly on compact subsets of $B_{N-1}\left(\alpha^{(0)},r\right)$ to the corresponding first-order partial derivatives of the inclusion map $B_{N-1}\left(\alpha^{(0)},r\right)\hookrightarrow B_{N-1}\left(\alpha^{(0)},r_0\right)$ uniformly on any compact subset of $B_{N-1}\left(\alpha^{(0)},r\right)$ as $j\to\infty$.  Hence for some $j\geq\tilde j$ there exists some nonempty open subset $U^\prime$ of $B_{N-1}(r)$ such that $\psi_j$ maps $U^\prime$ biholomorphically onto the open subset $\psi_j(U^\prime)$ of $B_{N-1}(r)$.  Since $\rho_j(U)$ contains the point $y^{(j)}$ which is not in ${\mathcal Z}\cup{\mathcal Z}^\prime$, it follows that $\rho_j(U)\cap\left({\mathcal Z}\cup{\mathcal Z}^\prime\right)$ is a proper subvariety of the connected complex manifold $\rho_j(U)$.  Let $S$ be the proper subvariety of $U$ such that $\rho_j$ maps $S$ bijectively onto $\rho_j(U)\cap\left({\mathcal Z}\cup{\mathcal Z}^\prime\right)$.

\medbreak Since the hypersurface $X^{(\alpha)}$ of ${\mathbb P}_n$ is pulled back by the biholomorphism $\sigma_{\alpha,j}$ of ${\mathbb P}_n$ to the hypersurface $X^{(\tau_{\sigma_{\alpha,j}}^{-1}(\alpha))}$ of ${\mathbb P}_n$, it follows from $\sigma_{\alpha,j}(\alpha)$ not belonging to ${\mathcal Z}\cup{\mathcal Z}^\prime$ for $\alpha\in U-S$ that both $\sigma_{\alpha,j}^*Q$ and $\sigma_{\alpha,j}^*F$ are nonzero and finite at $\rho(\tau_{\sigma_{\alpha,j}}^{-1}(\alpha))=\left(0,\tau_{\sigma_{\alpha,j}}^{-1}(\alpha)\right)$.  Since $\psi_j(\alpha)=\tau_{\sigma_{\alpha,j}}^{-1}(\alpha)$ belongs to $U$ for $\alpha\in U^\prime$, it follows from both $\sigma_{\alpha,j}^*Q$ and $\sigma_{\alpha,j}^*F$ being nonzero and finite at $\rho(\tau_{\sigma_{\alpha,j}}^{-1}(\alpha))=\left(0,\tau_{\sigma_{\alpha,j}}^{-1}(\alpha)\right)$ for $\alpha\in U^\prime-S$
that $\tau_{\sigma_{\alpha,j}}^{-1}(\alpha)$ is in $U$ and yet the point $\rho\left(\tau_{\sigma_{\alpha,j}}^{-1}(\alpha)\right)$ of $X^{\left(\tau_{\sigma_{\alpha,j}}^{-1}(\alpha)\right)}$ does not belong to ${\mathcal Z}\cup{\mathcal Z}^\prime$ for $\alpha\in U^\prime-S$, which contradicts the assumption that the point $\rho\left(\tau_{\sigma_{\alpha,j}}^{-1}(\alpha)\right)$ of the hypersurface $X^{\left(\tau_{\sigma_{\alpha,j}}^{-1}(\alpha)\right)}$ belongs to ${\mathcal Z}\cup{\mathcal Z}^\prime$.
\end{proof}

\begin{remark}\label{remark(4.2)}  The technique of slanted vector fields to reduce the vanishing orders of holomorphic jet differentials (vanishing on ample divisors) and to generate independent jet differentials (vanishing on ample divisors),
given in Proposition \ref{lower_vanishing_order_generate_independent_jet_differential}, is actually the infinitesimal or differential version of the above argument, given in Proposition \ref{basepoint_freeness_jet_differential}, of pulling back holomorphic jet differentials (vanishing on ample divisors) on neighboring fibers to reduce the common zero-set of holomorphic jet differentials (vanishing on ample divisors) on the original fiber.
\end{remark}

\bigbreak\begin{nodesignation}\label{proof_main_theorem}{\it Proof of Theorem \ref{main_theorem}.}\end{nodesignation}  For a generic hypersurface $X^{(\hat\alpha)}$ of sufficiently high degree $\delta$, by
Proposition \ref{basepoint_freeness_jet_differential} at every point $\hat y$ of $X^{(\hat\alpha)}$ and for every $P_0$ in the $(n-1)$-jet space $J_{n-1}\left(X^{(\hat\alpha)}\right)$ representable by a nonsingular complex curve germ at $\hat y$ there exists a holomorphic $(n-1)$-jet differential $$\omega^{(\hat\alpha)}=\frac{\sigma^*Q^{(\tau_\sigma(\hat\alpha))}}{\sigma^*F^{(\tau_\sigma(\hat\alpha))}}$$ of weight $m$ on $X^{(\hat\alpha)}$ which vanishes to order $>(m+1)(n-1)n(2n+1)$ on the intersection of $X^{(\hat\alpha)}$ and some hyperplane of ${\mathbb P}_n$, where $\sigma$ is a suitably chosen biholomorphism of ${\mathbb P}_n$ as described in the paragraph preceding Proposition \ref{basepoint_freeness_jet_differential}.  The condition of the jet differential $\omega^{(\hat\alpha)}$ vanishing to order $>(m+1)(n-1)n(2n+1)$ on the intersection of $X^{(\hat\alpha)}$ and some hyperplane of ${\mathbb P}_n$ is from the construction of the polynomial $Q$ in Proposition \ref{construction_jet_differential_as_polynomial} when the degree $\delta$ of $X^{(\hat\alpha)}$ is effectively sufficiently large.

\medbreak The holomorphic extension of $\omega^{(\hat\alpha)}$ to a holomorphic $(n-1)$-jet differential on $X^{(\alpha)}$ for $\alpha$ in some open neighborhood $U$ of $\hat\alpha$ in ${\mathbb P}_N$ satisfying the conditions of Proposition \ref{lower_vanishing_order_generate_independent_jet_differential} is automatic, because of the way the $(n-1)$-jet differential $\omega^{(\hat\alpha)}$ is constructed (or because of Proposition \ref{extendibility_to_neighborhing_fiber} when $\hat\alpha$ is assumed outside some subvariety of ${\mathbb P}_N$).  Now by Proposition \ref{hyperbolicity_after_construction_jet_differential} the hypersurface $X^{(\hat\alpha)}$ is hyperbolic in the sense that there is no nonconstant holomorphic map from ${\mathbb C}$ to $X^{(\hat\alpha)}$. This finishes the proof of Theorem \ref{main_theorem}.

\bigbreak\begin{nodesignation}\label{proof_hyperbolicity_complement_hypersurface}{\it Proof of Theorem \ref{hyperbolicity_complement_hypersurface}.}\end{nodesignation}  The proof is analogous to the proof of Theorem \ref{main_theorem}.  The difference is that
\begin{itemize}
\item[(i)] the holomorphic jet differential $\omega$ used in the proof of Theorem \ref{main_theorem} is replaced by the log-pole jet differential on ${\mathbb P}_n$ constructed in Theorem \ref{construction_log_pole_differential},
\item[(ii)] the use of Proposition \ref{sufficient_slanted_vector_fields} is replaced by the use of Proposition \ref{vector_field_projective_space},
\item[(iii)] the use of the Schwarz lemma of the vanishing of the pullback to ${\mathbb C}$ of a holomorphic jet differential vanishing on an ample divisor is replaced by Proposition \ref{general_schwarz_lemma} below concerning the general Schwarz lemma for log-pole jet differentials.
\end{itemize}

\bigbreak\section{\sc Essential Singularities, Varying Coefficients,
and Second Main Theorem}

\bigbreak Besides the Little Picard Theorem of nonexistence of nonconstant holomorphic maps from ${\mathbb C}$ to ${\mathbb P}_1-\{0,1,\infty\}$ there is a stronger statement which is the Big Picard Theorem
of no essential singularity at $\infty$ for any holomorphic function from ${\mathbb C}-\overline{\Delta_{r_0}}$ to ${\mathbb C}-\{0,1\}$ for $r_0>0$.  The Little Picard Theorem corresponds to our theorem on the hyperbolicity of a generic hypersurface $X^{(\alpha)}$ of sufficiently high degree in ${\mathbb P}_n$ (for $\alpha$ in ${\mathbb P}_N$ outside a proper subvariety ${\mathcal Z}$ of ${\mathbb P}_N$).  Corresponding to the Big Picard Theorem, there is a statement concerning the extendibility across a holomorphic map ${\mathbb C}-\overline{\Delta_{r_0}}\to X^{(\alpha)}$ to a holomorphic map ${\mathbb C}\cup\{\infty\}-\overline{\Delta_{r_0}}\to X^{(\alpha)}$.

\medbreak In this section we are going to prove such a theorem on removing the essential singularity at $\infty$ of a holomorphic map from ${\mathbb C}-\overline{\Delta_{r_0}}$ to a generic hypersurface of sufficiently high degree.

\medbreak The more quantitative version of the Big Picard Theorem was introduced by Nevanlinna \cite{Ne25} in his Second Main Theorem in his theory of value distribution theory.  In this section we discuss the Second Main Theorem from log-pole jet differentials which is more in the context of Cartan's generalization of Nevanlinna's Second Main Theorem to holomorphic maps from the affine complex line to ${\mathbb P}_n$ and a collection of hyperplanes in general position \cite{Ca33}.

\medbreak The hyperbolicity of a generic hypersurface of high degree $\delta$ in ${\mathbb P}_n$ can be reformulated as the nonexistence of $n+1$ entire holomorphic functions on ${\mathbb P}$ with some ratio nonconstant which satisfy a generic homogeneous polynomial of degree $\delta$ with constant coefficients.  Our solution of the hyperbolicity problem for
a generic hypersurface of high degree makes use of the universal hypersurface ${\mathcal X}$ in ${\mathbb P}_n\times{\mathbb P}_N$ and the variation of the hypersurface $X^{(\alpha)}$ in ${\mathbb P}_n$ with $\alpha\in{\mathbb P}_N$.  The variation of $X^{(\alpha)}$ corresponds to the varying of the constant coefficients of the homogeneous polynomial equation for the $n+1$ entire functions.  In this section we will discuss the problem of nonexistence of entire functions satisfying polynomial equations with slowly varying coefficients and also the more general result for removing essential singularities for holomorphic functions on ${\mathbb C}$ minus a disk for this setting.

\bigbreak\noindent{\bf I.} {\sc Removal of Essential Singularities}

\medbreak To extend our methods from maps ${\mathbb C}\to X^{(\alpha)}$ to
maps ${\mathbb C}-\overline{\Delta_{r_0}}\to X^{(\alpha)}$ for some $r_0>0$, we need a corresponding extension of Nevanlinna's logarithmic derivative lemma (p.51 of \cite{Ne25}).  For such an extension of Nevanlinna's logarithmic derivative lemma, we need the following trivial multiplicative version of the Heftungslemma \cite{AN64}.

\begin{lemma}\label{lemma(5.I.1)} {\rm (Trivial Multiplicative Version of Heftungslemma)}\ \   Let $r_0>0$ and $F$ be a meromorphic function on ${\mathbb C}-\overline{\Delta_{r_0}}$.  Let $r_0<r_1$.  Then there exists some function $G$ holomorphic and nowhere zero on ${\mathbb C}\cup\{\infty\}-\overline{\Delta_{r_1}}$ such that $FG$ is meromorphic on ${\mathbb C}$.  Moreover, when $F$ is holomorphic, $G$ can be chosen so that $FG$ is also holomorphic on ${\mathbb C}-\{0\}$.\end{lemma}

\begin{proof}  Choose $r_0<\rho_1<r_1<\rho_2$ such that $\partial\Delta_{\rho_j}$ contains no pole and no zero of $F$ for $j=1,2$.  Let
$\sum_{j=1}^{J_0} a_j$ be the zero-divisor of $F$ on $\overline{\Delta_{\rho_2}}-\Delta_{\rho_1}$ and
$\sum_{j=J_\infty} b_j$ be the pole-divisor of $F$ on $\overline{\Delta_{\rho_2}}-\Delta_{\rho_1}$.
Let
$$
h=\left(\prod_{j=J_0}\left(\zeta-a_j\right)\right)^{-1}
\left(\prod_{j=J_\infty}\left(\zeta-b_j\right)\right).
$$
Then $Fh$ is holomorphic and nowhwere zero on $\overline{\Delta_{\rho_2}}-\Delta_{\rho_1}$.  Let $\ell_0$ be the integer
$$
\frac{1}{2\pi}\int_{|\zeta|=r_1}d\log(Fh).
$$
Then
$$
\int_{|\zeta|=r_1}d\log\left(Fh\zeta^{-\ell_0}\right)=0
$$
and we can define a branch of $\log\left(Fh\zeta^{-\ell_0}\right)$ on $\overline{\Delta_{\rho_2}}-\Delta_{\rho_1}$, which we denote by $\Phi$.
From Cauchy's integral formula
$
\Phi\left(\zeta\right)=\Phi_0(\zeta)-\Phi_\infty$,
where
$$
\displaylines{
\Phi_0(\zeta)=\frac{1}{2\pi}\int_{|\hat\zeta|=\rho_2}\frac{\Phi(\hat\zeta)}{\hat\zeta-\zeta}\,d\hat\zeta,\cr
\Phi_\infty(\zeta)=\frac{1}{2\pi}\int_{|\hat\zeta|=\rho_2}\frac{\Phi(\hat\zeta)}{\hat\zeta-\zeta}\,d\hat\zeta\cr}
$$
for $\zeta\in{\Delta_{\rho_2}}-\overline{\Delta_{\rho_1}}$.
Exponentiating both sides
of $\Phi\left(\zeta\right)=\Phi_0(\zeta)-\Phi_\infty$, we get
$Fh\zeta^{-\ell_0}=e^{\Phi_0}e^{-\Phi_\infty}$ and
$Fe^{\Phi_\infty}=h^{-1}\zeta^{\ell_0}e^{\Phi_0}$.  Since the right-hand side $h^{-1}\zeta^{\ell_0}e^{\Phi_0}$ of
$Fe^{\Phi_\infty}=h^{-1}\zeta^{\ell_0}e^{\Phi_0}$ is meromorphic on $\Delta_{\rho_2}$ and the left-hand side of $Fe^{\Phi_\infty}=h^{-1}\zeta^{\ell_0}e^{\Phi_0}$ is meromorphic on ${\mathbb C}-\overline{\Delta_{\rho_1}}$, it follows $Fe^{\Phi_\infty}$ is meromorphic on all of ${\mathbb C}$.  We apply the transformation $w=\frac{1}{\zeta}$ to get
$$
\Phi_\infty=\frac{1}{2\pi}\int_{|\hat\zeta|=\rho_1}\frac{\Phi(\hat\zeta)}{\hat\zeta-\zeta}\,d\hat\zeta
=\frac{1}{2\pi}\int_{|\hat\zeta|=\rho_1}\frac{w\Phi(\zeta)}{w\zeta-1}\,d\hat\zeta
$$
which is holomorphic for $|w|<\frac{1}{\rho_1}$, that is, holomorphic for
$\zeta_0\in{\mathbb C}\cup\{\infty\}-\overline{\Delta_{\rho_1}}$.  Now the function $G(\zeta)=e^{\Phi_\infty}$ satisfies our requirement.  When $F$ is holomorphic on ${\mathbb C}-\overline{\Delta_{r_0}}$, the function $FG$ which is equal to $h^{-1}\zeta^{\ell_0}e^{\Phi_0}$ on $\Delta_{\rho_2}$ and is equal to $e^{\Phi_0}$ on ${\mathbb C}-\overline{\Delta_{\rho_1}}$ is clearly holomorphic on ${\mathbb C}-\{0\}$.
\end{proof}

\begin{nodesignation}\label{compare_characteristic_function}{\it Comparison of Characteristic Functions of Maps Defined Outside a Disk.}\end{nodesignation}  Let $r_1>r_0>0$.  For a meromorphic function $H$ on ${\mathbb C}-\overline{\Delta_{r_0}}$ we introduce for $c\in{\mathbb C}\cup\{\infty\}$ the counting function
$$
N\left(r,r_1,H,c\right)=\int_{\rho=r_1}^r n\left(\rho,r_1,H,c\right)\frac{d\rho}{\rho},
$$
where $n\left(\rho,r_1,H,c\right)$ is the number of roots of $H(\zeta)=c$ with multiplicities counted in $r_1\leq|\zeta|\leq\rho$ and also the characteristic function
$$
T(r, r_1,H)=N\left(r,r_1,H,\infty\right)+\frac{1}{2\pi}\int_{\theta=0}^{2\pi}\log^+\left|H\left(re^{i\theta}\right)\right|
\,d\theta.
$$
For a holomorphic map $\varphi$ from ${\mathbb C}-\overline{\Delta_{r_0}}$ to a complex manifold $Y$ and a $(1,1)$-form $\eta$ on $Y$, we introduce the characteristic function
$$
T(r, r_1,\varphi,\eta)=\int_{\rho=r_1}^r\left(\int_{\Delta_r-\overline{\Delta_{r_1}}}\varphi^*\eta\right)\frac{d\rho}
{\rho}.
$$
For $Y={\mathbb P}_n$ and $\eta$ being the Fubini-Study form, we drop $\eta$ in the notation $T(r, r_1,\varphi,\eta)$ and simply use $T(r, r_1,\varphi)$ when there is no confusion.  When a holomorphic map $\varphi$ from ${\mathbb C}-\overline{\Delta_{r_0}}$ to ${\mathbb P}_n$ is given by holomorphic functions $\left[F_0,\cdots,F_n\right]$ on ${\mathbb C}-\overline{\Delta_{r_0}}$ without common zeroes, its characteristic function is
$$T(r,r_1,\varphi)=\int_{\rho=r_1}^r\left(\int_{|\zeta|<\rho}\frac{1}{\pi}
\partial_\zeta\partial
_{\overline\zeta}\log\sum_{k=0}^N\left|F_k\right|^2\right)\frac{d\rho}{\rho}.$$
We would like to compare it with the characteristic function $T\left(r,r_1,\frac{F_j}{F_0}\right)$ for the meromorphic function $\frac{F_j}{F_0}$ for $1\leq j\leq n$.
For $1\leq j\leq n$ we have the inequality
$$
T\left(r,\frac{F_j}{F_0}\right)\leq T(r,\varphi)+O(1)\leq\sum_{k=1}^n T\left(r,\frac{F_k}{F_0}\right)+O(1).
$$
The verification of
$$
T\left(r,\frac{F_j}{F_0}\right)\leq T(r,\varphi)+O(1)
$$
for $1\leq j\leq n$ is as follows.

\medbreak From twice integration of Laplacian with $g=\log\sum_{k=0}^N\left|F_k\right|^2$ in (\ref{twice_integration}), we have
$$
\displaylines{\int_{\rho=r_1}^r\left(\int_{|\zeta|<\rho}\frac{1}{\pi}
\partial_\zeta\partial
_{\overline\zeta}\log\sum_{k=0}^N\left|F_k\right|^2\right)\frac{d\rho}{\rho}\cr=
\frac{1}{4\pi}\int_{\theta=0}^{2\pi}\log\sum_{k=0}^N\left|F_k\left(re^{i\theta}\right)\right|^2\,d\theta-
\frac{1}{4\pi}\int_{\theta=0}^{2\pi}\log\sum_{k=0}^N\left|F_k\left(r_1e^{i\theta}\right)\right|^2\,d\theta\cr}
$$
from which we conclude that
$$
\displaylines{\frac{1}{2\pi}\int_{\theta=0}^{2\pi}
\log^+\left|\frac{F_j}{F_0}\left(re^{i\theta}\right)\right|\,d\theta\cr
\leq\frac{1}{4\pi}\left(\int_{\theta=0}^{2\pi}\log\sum_{k=0}^N\left|F_k\right|^2\right)\frac{d\rho}{\rho}\cr
=\int_{\rho=r_1}^r\left(\int_{|\zeta|<\rho}\frac{1}{\pi}
\partial_\zeta\partial
_{\overline\zeta}\log\sum_{k=0}^N\left|F_k\right|^2\right)\frac{d\rho}{\rho}+
\frac{1}{4\pi}\int_{\theta=0}^{2\pi}\log\sum_{k=0}^N\left|F_k\left(r_1e^{i\theta}\right)\right|^2\,d\theta
\cr=T(r,\varphi)+O\left(\log r\right).}
$$
Finally from
$$
\displaylines{N\left(r,r_1,F_0,0\right)=\int_{\rho=r_1}^r\left(\frac{1}{\pi}\partial_\zeta\partial_{\bar\zeta}
\log\left|F_0\right|^2\right)\frac{d\rho}{\rho}\cr
=\int_{\rho=r_1}^r\left(\frac{1}{\pi}\partial_\zeta\partial_{\bar\zeta}
\log\left(\frac{\left|F_0\right|^2}{\sum_{j=0}^N\left|F_0\right|^2}
\cdot\sum_{j=0}^N\left|F_0\right|^2\right)\right)\frac{d\rho}{\rho}\cr
\leq\int_{\rho=r_1}^r\left(\frac{1}{\pi}\partial_\zeta\partial_{\bar\zeta}
\log\sum_{j=0}^N\left|F_0\right|^2\right)\frac{d\rho}{\rho}=T(r,\varphi)\cr}
$$
it follows that
$$\displaylines{T\left(r,r_1,\frac{F_j}{F_0}\right)=
\frac{1}{2\pi}\int_{\theta=0}^{2\pi}
\log^+\left|\frac{F_j}{F_0}\left(re^{i\theta}\right)\right|\,d\theta+N\left(r,r_1,F_0,0\right)\cr\leq O\left(T(r,r_1,\varphi)+\log r\right)\quad\|}$$

The verification of
$$
T(r,r_1,\varphi)\leq\sum_{j=k}^n T\left(r,r_1,\frac{F_k}{F_0}\right)+O(1)
$$
for $1\leq j\leq n$ is as follows.  From
$$
\displaylines{\int_{\rho=r_1}^r\left(\int_{|\zeta|<\rho}\frac{1}{\pi}
\partial_\zeta\partial
_{\overline\zeta}\log\sum_{k=0}^N\left|F_k\right|^2\right)\frac{d\rho}{\rho}\cr=
\frac{1}{4\pi}\int_{\theta=0}^{2\pi}\log\sum_{k=0}^N\left|F_k\left(re^{i\theta}\right)\right|^2\,d\theta-
\frac{1}{4\pi}\int_{\theta=0}^{2\pi}\log\sum_{k=0}^N\left|F_k\left(r_1e^{i\theta}\right)\right|^2\,d\theta\cr
=\frac{1}{4\pi}\int_{\theta=0}^{2\pi}\log\left(1+\sum_{k=1}^N\left|\frac{F_k}{F_0}\left(re^{i\theta}\right)\right|^2\right)d\theta
+\frac{1}{2\pi}\int_{\theta=0}^{2\pi}\log\left|F_0\left(re^{i\theta}\right)\right|d\theta\cr-
\frac{1}{4\pi}\int_{\theta=0}^{2\pi}\log\sum_{k=0}^N\left|F_k\left(r_1e^{i\theta}\right)\right|^2\,d\theta\cr
}
$$
we have
$$
T(r,r_1,\varphi)\leq\sum_{k=1}^n\frac{1}{2\pi}\int_{\theta=0}^{2\pi}\log^+\left|\frac{F_k}{F_0}\left(re^{i\theta}\right)\right|d\theta
+\frac{1}{2\pi}\int_{\theta=0}^{2\pi}\log\left|F_0\left(re^{i\theta}\right)\right|d\theta+O(1).
$$
The verification of the inequality comparing characteristic functions of maps and meromorphic functions is a straightforward modification of the proof of Lemma (2.1.2) on p.426 of \cite{Si95}, where the map is from ${\mathbb C}$ instead of from ${\mathbb C}-\overline{\Delta_{r_0}}$.

\bigbreak The logarithmic derivative lemma holds for meromorphic functions on ${\mathbb C}-\overline{\Delta_{r_0}}$ for $r_0>0$ in the following form.

\begin{proposition}\label{proposition(5.I.2)} {\rm (Logarithmic Derivative Lemma for Functions Meromorphic in Punctured Disk Centered at Infinity)}\ \   Let $r_1>r_0>0$ and $F$ be a meromorphic function on ${\mathbb C}-\overline{\Delta_{r_0}}$.  Then
$$
\int_{\theta=0}^{2\pi}\log^+\left|\frac{F^\prime\left(re^{i\theta}\right)}
{F\left(re^{i\theta}\right)}\right|\,d\theta
=O\left(\log T(r, r_1,F)+\log r\right)\quad\|
$$
for $r>r_1$.
\end{proposition}

\begin{proof}  By \ref{lemma(5.I.1)} there exists some function $G$ holomorphic and nowhere zero on ${\mathbb C}\cup\{\infty\}-\overline{\Delta_{r_1}}$ such that $FG$ is meromorphic on ${\mathbb C}$.  Let $H=FG$.  Then $\left(\log F\right)^\prime=\left(\log H\right)^\prime-\left(\log G\right)^\prime$ and
$$
\log^+\left|\frac{F^\prime\left(re^{i\theta}\right)}
{F\left(re^{i\theta}\right)}\right|\leq
\log^+\left|\frac{H^\prime\left(re^{i\theta}\right)}
{H\left(re^{i\theta}\right)}\right|+\log^+\left|\frac{G^\prime\left(re^{i\theta}\right)}
{G\left(re^{i\theta}\right)}\right|+\log 2.
$$
Thus,
$$
\displaylines{\int_{\theta=0}^{2\pi}\log^+\left|\frac{F^\prime\left(re^{i\theta}\right)}
{F\left(re^{i\theta}\right)}\right|\,d\theta
\leq
\int_{\theta=0}^{2\pi}\log^+\left|\frac{H^\prime\left(re^{i\theta}\right)}
{H\left(re^{i\theta}\right)}\right|+
\frac{1}{r}\,\int_{|\zeta|=r}\log^+\left|dG\right|
+\log 2\cr\leq \int_{\theta=0}^{2\pi}\log^+\left|\frac{H^\prime\left(re^{i\theta}\right)}
{H\left(re^{i\theta}\right)}\right|+O(1)\leq
O\left(\log T(r,H)+\log r\right)\quad\|\cr}
$$
because $G$ holomorphic and nowhere zero on ${\mathbb C}\cup\{\infty\}-\overline{\Delta_{r_1}}$.  The required statement follows from
$$
N\left(r,H,\infty\right)\leq N\left(r,r_1,F,\infty\right)+O\left(\log r\right)
$$
and
$\log^+\left|H\right|\leq\log^+\left|F\right|+O(1)$.
By \ref{lemma(5.I.1)} on the trivial multiplicative version of Heftungslemma, for some holomorphic nowhere-zero function $G_0$ on ${\mathbb C}\cup\{\infty\}-\overline{\Delta_{r_2}}$ with $r_0<r_2<r_1$ such that $H(\zeta):=\zeta^\ell F_0(\zeta)G_0(\zeta)$ is holomorphic on ${\mathbb C}$ (with coordinate $\zeta$) and is nonzero at $\zeta=0$.
$$
\displaylines{\frac{1}{2\pi}\int_{\theta=0}^{2\pi}\log\left|F_0\left(re^{i\theta}\right)\right|d\theta\cr=
\frac{1}{2\pi}\int_{\theta=0}^{2\pi}\log\left|\left(re^{i\theta}\right)^\ell\left(F_0G_0\left(re^{i\theta}\right)\right)\right|d\theta-
\frac{1}{2\pi}\int_{\theta=0}^{2\pi}\log\left|\left(re^{i\theta}\right)^{-\ell}\left(G_0\left(re^{i\theta}\right)\right)\right|d\theta\cr
=\log\left|H(0)\right|+N(r,H,0)-\frac{1}{2\pi}\int_{\theta=0}^{2\pi}\log\left|\left(re^{i\theta}\right)^{-\ell}\left(G_0\left(re^{i\theta}\right)\right)\right|d\theta\cr
=N(r,H,0)+O(1)\leq N\left(r,r_1,F_0,0\right)+O(1)\leq\sum_{k=1}^nN\left(r,r_1,\frac{F_k}{F_0},\infty\right)+O(1)}
$$
because $F_0,F_1,\cdots,F_n$ are assumed to have no common zeroes on ${\mathbb C}-\overline{\Delta_{r_0}}$.  Thus,
$$
\displaylines{
T(r,\varphi)\leq\sum_{k=1}^n\frac{1}{2\pi}\int_{\theta=0}^{2\pi}\log^+\left|\frac{F_k}{F_0}\left(re^{i\theta}\right)\right|d\theta
+\sum_{k=1}^nN\left(r,r_1,\frac{F_k}{F_0},\infty\right)+O(1)\cr=\sum_{k=1}^n T\left(r,r_1,\frac{F_k}{F_0}\right)+O(1).\cr}
$$

\end{proof}

\begin{proposition}\label{proposition(5.I.3)} {\rm (Vanishing of Pullback of Jet Differential Vanishing on an Ample Divisor by Holomorphic Map to Punctured Disk Centered at Infinity)}  Let $X$ be a compact complex manifold with a K\"ahler form $\eta$ and $\omega$ be a holomorphic jet differential on $X$ vanishing on some ample divisor $D$ of $X$.  Let $r_1>r_0>0$ and $\varphi:{\mathbb C}-\overline{\Delta_{r_0}}\to X$ be a holomorphic map.  Let ${\rm eval}_{{\rm id}_{\mathbb C}}(\varphi^*\omega)$ denote the function on ${\mathbb C}-\overline{\Delta_{r_0}}$ whose value at $\zeta$ is the evaluation of the jet differential $\varphi^*\omega$ at the jet defined by the identity map of ${\mathbb C}-\overline{\Delta_{r_0}}$ at $\zeta$.  Then either ${\rm eval}_{{\rm id}_{\mathbb C}}(\varphi^*\omega)$ is identically zero on ${\mathbb C}-\overline{\Delta_{r_0}}$ or $T\left(r,r_1,\varphi,\eta\right)=O\left(\log r\right)\ \ \|$.
\end{proposition}

\begin{proof} Let $k$ be the order of jet differential $\omega$ and $m$ be its weight. Let $L_D$ be the line bundle associated to the ample divisor $D$.  Let $e^{-\chi_D}$ be a smooth metric of $L_D$ whose curvature form $\eta_D$ is strictly positive definite on $X$.  Let $s_D$ be a holomorphic section of $L_D$ whose divisor is $D$.
Let $\Phi={\rm eval}_{{\rm id}_{\mathbb C}}(\varphi^*\omega)$.  We assume that $\Phi$ is not identically zero.  We apply twice integration of Laplacian in (\ref{twice_integration}) to
$$g(\zeta)=\log\left(\frac{\left|\Phi\right|^2}{\,\left|s_D\right|^2 e^{-\chi_D}\,}\right).
$$
Since $\omega$ is holomorphic on $X$ and vanishing on $D$, it follows that
$$
T\left(r,r_1,\varphi,\eta_D\right)
\leq\frac{1}{4\pi}
\int_{\theta=0}^{2\pi}\log\left(\frac{\left|\Phi\right|^2}{\,\left|s_D\right|^2 e^{-\chi_D}\,}\right)d\theta+O(1).
$$
Here we have the inequality instead of an identity, because of possible contribution from the zero-set of $\frac{\omega}{s_D}$. At this point enters Bloch's technique of applying the logarithmic derivative lemma by using the logarithm of global meromorphic functions as local coordinates.  As functions on the $k$-jet space $J_k(X)$ of $X$ (with the right-hand side being global functions and the left-hand side being only local functions due to the transition functions of the line bundles $L_D$),
$$
\left|\frac{\omega}{s_D}\right|\leq C\sum_{\lambda=1}^\Lambda
\prod_{\nu_{j,\ell},\,j,\ell}\left|d^\ell\log F^{(\lambda)}_{j,\ell}\right|^{\nu_{j,\ell}}
$$
for some $C>0$ and a finite collection $\left\{F^{(\lambda)}_{j,\ell}\right\}$ of global meromorphic functions on $X$, where the product is taken over the indices $\nu_{j,\ell},\,j,\ell$ with the ranges $1\leq j\leq n,\,1\leq\ell\leq k$ and $\sum_{1\leq j\leq n,\,1\leq\ell\leq k}\ell\,\nu_{j,\ell}=m$.  By Nevanlinna's logarithmic derivative lemma (extended to ${\mathbb C}$ outside a disk centered at the origin) given in (\ref{proposition(5.I.2)}),
$$
\int_{\theta=0}^{2\pi}\log^+\left|d^\ell\log F^{(\lambda)}_{j,\ell}\right|\left(re^{i\theta}\right)d\theta=O\left(\log T\left(r,r_1,\varphi,\eta_D\right)+\log r\right)\ \ \|.
$$
Hence
$$
\int_{\theta=0}^{2\pi}\log\left(\frac{\left|\Phi\right|^2}{\,\left|s_D\right|^2 e^{-\chi_D}\,}\right)\left(re^{i\theta}\right)d\theta
=O\left(\log T\left(r,r_1,\varphi,\eta_D\right)+\log r\right)\ \ \|
$$
and we get
$$
T\left(r,r_1,\varphi,\eta_D\right)\leq O\left(\log T\left(r,r_1,\varphi,\eta_D\right)+\log r\right)\ \ \|
$$
from which it follows that
$$T\left(r,r_1,\varphi,\eta_D\right)=O\left(\log r\right)\ \ \|$$
and
$$T\left(r,r_1,\varphi,\eta\right)=O\left(\log r\right)\ \ \|.$$

\end{proof}

\begin{remark}\label{to_entire_affine_line_case}  The argument in Proposition (\ref{proposition(5.I.3)}) is simply a modification of the case of the usual Schwarz lemma on pullbacks of jet differentials (see {\it e.g.,} Theorem 2 on p.1140 of \cite{SY97}) when the holomorphic map $\varphi$ is from the entire affine complex line ${\mathbb C}$ to $X$.  For this case, the pullback $\varphi^*\omega$ is always identically zero on ${\mathbb C}$ for the following reason.  We can replace $\varphi$ by the composite $\psi$ of $\varphi$ with the exponential map from ${\mathbb C}$ to ${\mathcal C}$ to rule out the case of $T\left(r,\psi,\eta\right)=O\left(\log r\right)\ \ \|$ so that ${\rm eval}_{{\rm id}_{\mathbb C}}(\psi^*\omega)$ vanishes identically on ${\mathbb C}$.  Since any $k$-jet of ${\mathbb C}$ at any point $\zeta_0$ of ${\mathbb C}$ can realized by some holomorphic map $\sigma$ from ${\mathbb C}$ to itself, from the vanishing of ${\rm eval}_{{\rm id}_{\mathbb C}}(\sigma^*\omega)$ vanishes identically on ${\mathbb C}$ it follows that $\psi^*\omega$ is identically zero on ${\mathbb C}$, which implies that $\varphi^*\omega$ is identically zero on ${\mathbb C}$.
\end{remark}

\begin{lemma}\label{extend_across_infinity_point} {\rm (Extension of Holomorphic Maps with Log Order Growth Characteristic Function Across Infinity Point)} \ \  Let $\varphi$ be a holomorphic map from ${\mathbb C}-\overline{\Delta_{r_0}}$ to ${\mathbb P}_n$ given by holomorphic functions $\left[F_0,\cdots,F_n\right]$ on ${\mathbb C}-\overline{\Delta_{r_0}}$ without common zeroes.  If $T(r,r_1,\varphi)=O(\log r)\quad\|$, then $\varphi$ can be extended to a holomorphic map from ${\mathbb C}\cup\{\infty\}-\overline{\Delta_{r_0}}$ to ${\mathbb P}_n$.
\end{lemma}

\begin{proof} By the comparison of characteristic functions of maps defined outside a disk given in (\ref{compare_characteristic_function}),we have
$$T\left(r,\frac{F_j}{F_0}\right)\leq T\left(r,\varphi\right)=O(\log r)\quad\|$$ for $1\leq j\leq n$.   By the trivial multiplicative version of the Heftungslemma given in (\ref{lemma(5.I.1)}), there exists some holomorphic nowhere zero function $G_j$ on ${\mathbb C}\cup\{\infty\}-\overline{\Delta_{r_1}}$ such that $G_j\frac{F_j}{F_0}$ is meromorphic on ${\mathbb C}$. From
$$T\left(r,G_j\frac{F_j}{F_0}\right)=T\left(r,\frac{F_j}{F_0}\right)+O(1)\leq T\left(r,\varphi\right)=O(\log r)\quad\|$$
for $1\leq j\leq n$
we conclude that $G_j\frac{F_j}{F_0}$ is a rational function on ${\mathbb C}$.  Hence $\frac{F_j}{F_0}$ can be extended to a meromorphic function on ${\mathbb C}\cup\{\infty\}-\overline{\Delta_{r_0}}$.  Thus $\varphi$ can be can be extended to a holomorphic map from ${\mathbb C}\cup\{\infty\}-\overline{\Delta_{r_0}}$ to ${\mathbb P}_n$.
\end{proof}

\begin{proposition}\label{extendibility_map}  Let $X$ be a compact complex manifold.  Let $\omega_1,\cdots,\omega_N$ be holomorphic $k$-jet differentials of total weight $m$ on $X$ with each vanishing on some ample divisor of $X$.   Assume that at any point $P_0$ of $J_k(X)$ which is representable by a nonsingular complex curve germ, at least one of $\omega_1,\cdots,\omega_N$ is nonzero at $P_0$ for some $1\leq j\leq N$.
Then any holomorphic map $\varphi$ from ${\mathbb C}-\overline{\Delta_{r_0}}$ to $X$ can be extended to a holomorphic map from ${\mathbb C}\cup\{\infty\}-\overline{\Delta_{r_0}}$ to $X$.\end{proposition}

\begin{proof}  We can assume without loss of generality that $\varphi$ is not a constant map so that at some point $\zeta_0$ of ${\mathbb C}-\overline{\Delta_{r_0}}$ the differential of $\varphi$ is nonzero at $\zeta_0$.  Let $P_0$ be the element of $J_k(X)$ at the point $\varphi(\zeta_0)$ of $X$ defined by the nonsingular complex curve germ represented by $\varphi$ at $\zeta_0$.  By assumption, there exists some $1\leq j_0\leq N$ such that $\omega_{j_0}$ is nonzero at $P_0$.  It follows that the function ${\rm eval}_{{\rm id}_{\mathbb C}}(\varphi^*\omega_{j_0})$ associated to $\varphi^*\omega_{j_0}$ as described in (\ref{pullback_function}) is nonzero at $\zeta_0$.  Let $\eta$ be a K\"ahler form of $X$ and let $r_1>r_0$.  By (\ref{proposition(5.I.3)}) we have $$T(r,r_1,\varphi,\eta)=O(\log r)\quad\|,$$
which implies that the holomorphic map $\varphi$ from ${\mathbb C}-\overline{\Delta_{r_0}}$ to $X$ can be extended to a holomorphic map from ${\mathbb C}\cup\{\infty\}-\overline{\Delta_{r_0}}$ to $X$.  \end{proof}

\begin{theorem}\label{big_picard_for_hypersurface} For any integer
$n\geq 3$ there exists a positive integer $\delta_n$ with the
following property.  For any positive integer $\delta\geq\delta$ there exists a proper subvariety ${\mathcal Z}$ in the moduli space ${\mathbb P}_N$ of all hypersurfaces of degree $\delta$ in ${\mathbb P}_n$ (where $N={n+\delta\choose n}$) such that for $\alpha\in{\mathbb P}_n-{\mathcal Z}$ and any holomorphic map $\varphi:{\mathbb C}-\overline{\Delta_{r_0}}\to X^{(\alpha)}$ (where $r_0>0$) can be extended to a holomorphic map ${\mathbb C}\cup\{\infty\}-\overline{\Delta_{r_0}}\to X^{(\alpha)}$, where $X^{(\alpha)}$ is the hypersurface of degree $\delta$ in ${\mathbb P}_n$ corresponding to the point $\alpha$ in the moduli space ${\mathbb P}_N$.
\end{theorem}

\begin{proof} By Proposition (\ref{sufficient_slanted_vector_fields}) on global generation on jet space by slanted vector fields at points representable by regular curve germs,

\end{proof}

\bigbreak\noindent{\bf II.} {\sc Entire Function Solutions of Polynomial Equations of Slowly Varying Coefficients}

\begin{nodesignation}\label{II}{\it Historical Background, Osculation Condition, and Log-Pole Jet Differential.}\end{nodesignation}Before the introduction of the language of geometry of manifolds, hyperbolicity problems were formulated in terms of entire functions satisfying functional equations.  For example, a theorem of Borel states that if entire function $\varphi_1,\cdots,\varphi_n$ satisfy $e^{\varphi_1}+\cdots+e^{\varphi_n}=0$, then
$\varphi_j-\varphi_k$ is constant with $1\leq j<k\leq n$.  In the formulation in terms of functional equations satisfied by entire functions, the hyperbolicity of a generic hypersurface of high degree $\delta$ states that no $n+1$ entire holomorphic functions $\varphi_0(\zeta),\cdots,\varphi_n(\zeta)$, with the ratios $\frac{\varphi_j}{\varphi_\ell}$ not all constant, can satisfy a homogeneous polynomial equation
$$
\sum_{\nu_0+\cdots+\nu_n=\delta}\alpha_{\nu_0,\cdots,\nu_n}\varphi_1^{\nu_1}\cdots\varphi_1^{\nu_1}\equiv 0$$
of degree $\delta$ whose constant coefficients $\alpha_{\nu_0,\cdots,\nu_n}$ are generic.

\medbreak There have also been considerable investigations on the situation when the constant coefficients are allowed to vary slowly.  For example, on p.387 of his 1897 paper \cite{Bo97}, Emile Borel studied the problem of entire functions $\gamma_1(\zeta),\cdots,\gamma_n(\zeta)$ and $\varphi_1(\zeta),\cdots,\varphi_n(\zeta)$ satisfying
$\gamma_1e^{\varphi_1}+\cdots+\gamma_ne^{\varphi_n}=0$ and proved that $\gamma_1(\zeta),\cdots,\gamma_n(\zeta)$ must be all identically zero if the growth rate on $|\zeta|=r$ of
$\gamma_1(\zeta),\cdots,\gamma_n(\zeta)$ and $\varphi_1(\zeta),\cdots,\varphi_n(\zeta)$ is no more than $e^{\mu(|\zeta|)}$ while the growth rate on $|\zeta|=r$ of $\varphi_j(\zeta)-\varphi_\ell(\zeta)$ for $j\not=\ell$ is at least $\mu(r)^2$ for some function $\mu(r)$ as $r\to\infty$.

\medbreak For the hyperbolicity problem of generic hypersurface of degree $\delta$, we now study the question of entire functions satisfying a homogeneous polynomial equation of degree $\delta$ with varying coefficients.  More precisely, we ask whether there are entire functions $\varphi_0(\zeta),\cdots,\varphi_n(\zeta)$ without common zeroes and entire functions $\alpha_{\nu_0,\cdots,\nu_n}(\zeta)$ for $\nu_0+\cdots+\nu_n=\delta$ without any common zeroes satisfying
$$\sum_{\nu_0+\cdots+\nu_n=\delta}\alpha_{\nu_0,\cdots,\nu_n}(\zeta)\varphi_0(\zeta)^{\nu_0}\cdots \varphi_n(\zeta)^{\nu_n}\equiv 0
\leqno{(\ref{II}.1)_0}
$$ such that
\begin{itemize}\item[(i)] $\psi:\zeta\mapsto\alpha(\zeta)=\left(\alpha_{\nu_0,\cdots,\nu_n}(\zeta)\right)\in{\mathbb P}_N$ is nonconstant,
\item[(ii)] $\alpha\left(\zeta_0\right)=\left(\alpha_{\nu_0,\cdots,\nu_n}\left(\zeta_0\right)\right)$
is not in the exceptional set ${\mathcal Z}$ in ${\mathbb P}_N$ for some $\zeta_0\in{\mathbb C}$, and
\item[(iii)]
$T\left(r,\psi\right)=o\left(T\left(r,\varphi\right)+\log r\right)\ \ \|$, where $\varphi:{\mathbb C}\to{\mathbb P}_n$ is defined by $\left[\varphi_0,\cdots,\varphi_n\right]$.
\end{itemize}
Here we handle the simpler question which assumes in addition that
$$\sum_{\nu_0+\cdots+\nu_n=\delta}\alpha_{\nu_0,\cdots,\nu_n}(\zeta)
\frac{d^j}{d\zeta^j}\left(\varphi_0(\zeta)^{\nu_0}\cdots \varphi_n(\zeta)^{\nu_n}\right)\equiv 0
\leqno{(\ref{II}.1)_j}
$$
for $1\leq j\leq n-1$.  The additional set of $n-1$ equations $
(\ref{II}.1)_j$ for $1\leq j\leq n-1$ is equivalent to the set of $n-1$ equations
$$\sum_{\nu_0+\cdots+\nu_n=\delta}\left(\frac{d^j}{d\zeta^j}\alpha_{\nu_0,\cdots,\nu_n}(\zeta)\right)
\left(\varphi_0(\zeta)^{\nu_0}\cdots \varphi_n(\zeta)^{\nu_n}\right)\equiv 0
\leqno{(\ref{II}.2)_j}
$$
for $1\leq j\leq n-1$, because of the equation $
(\ref{II}.1)_0$ itself and the result obtained by differentiating it $j$-times with respect to $\zeta$.

\medbreak A geometric interpretation of the conditions $
(\ref{II}.1)_0$ and $
(\ref{II}.2)_j$ for $1\leq j\leq n-1$ is the following.  When $$\left\{\varphi_0(\zeta)^{\nu_0}\cdots \varphi_n(\zeta)^{\nu_n}\right\}_{\nu_0+\cdots+\nu_\nu=\delta}$$ is considered as the set of coefficients of a linear equation which defines a hyperplane
$H(\zeta)$ in ${\mathbb P}_N$, as $\zeta$ varies in ${\mathbb C}$ we have a moving hyperplane depending on $\zeta$.  Having entire functions $\varphi_\ell(\zeta)$ for $0\leq\ell\leq n$ and tneire functions $\alpha_{\nu_0,\cdots,\nu_n}(\zeta)$ for $\nu_0+\cdots+\nu_n=\delta$ satisfying $
(\ref{II}.1)_j$ for $0\leq j\leq n-1$ is equivalent to the existence of a holomorphic map $$\zeta\mapsto\alpha(\zeta)=\left(\alpha_{\nu_0,\cdots,\nu_n}\left(\zeta\right)\right)_{\nu_0+\cdots+\nu_n=\delta}$$
from ${\mathbb C}$ to ${\mathbb P}_N$ which osculates the hyperplane $H(\zeta)$ to order $n-1$ in the sense that the curve $\zeta\mapsto\alpha(\zeta)$ in ${\mathbb P}_N$ is tangential to order $n-1$ to the hyperplane $H\left(\zeta_0\right)$ of ${\mathbb P}_N$ at the point $\alpha\left(\zeta_0\right)\in{\mathbb P}_N$.    Condition (iii) of
$T\left(r,\psi\right)$ being of order $o\left(T\left(r,\varphi\right)+\log r\right)$ is the condition of slowly varying coefficients.

\medbreak For this question of polynomial equations with slowly varying coefficients under additional assumption of osculation, we present here two results.  The first one, corresponding to hyperbolicity, is that when the map $\zeta\mapsto\alpha(\zeta)$ is slowly moving compared to the map $\zeta\mapsto\varphi(\zeta)\in{\mathbb P}_n$, no such pair of curves  $\zeta\mapsto\alpha(\zeta)\in{\mathbb P}_N$ and $\zeta\mapsto\varphi(\zeta)\in{\mathbb P}_n$ exist.

\medbreak The second result, corresponding to the Big Picard Theorem, concerns extension across $\infty$ when the pair of maps
$\zeta\mapsto\varphi(\zeta)\in{\mathbb P}_n$ and $\zeta\mapsto\alpha(\zeta)\in{\mathbb P}_N$
are only defined for $\zeta\in{\mathbb C}-\overline{\Delta_{r_0}}$ instead of on ${\mathbb C}$.

\medbreak Since the Schwarz lemma is the crucial tool for the hyperbolicity problem, for the more general case of slowly varying coefficients we need a variation of the Schwarz lemma for it.  We are going to present it in the form which is more than we need by allowing log-pole jet differentials rather than just holomorphic jet differentials so that it can be used later in this article in the proof of Second Main Theorems for log-pole jet differentials (see Thoerem \ref{second_main_theorem_jet_differential} and Theorem \ref{second_main_theorem_jet_differential_moving_target} below).  A log-pole jet differential means that locally it is of the form
$$
\sum_\lambda G_\lambda\left(d^{\ell_{1,\lambda}}x_1\right)^{\nu_{1,\lambda}}\cdots
\left(d^{\ell_{1,\lambda}}x_1\right)^{\nu_{1,\lambda}}
\left(d^{\sigma_{1,\lambda}}\log F_1\right)^{\tau_{1,\lambda}}\cdots
\left(d^{\sigma_{\mu_\lambda,\lambda}}\log F_{\mu_\lambda}\right)^{\tau_{\mu_\lambda,\lambda}},
$$
where $x_1,\cdots,x_n$ are local holomorphic coordinates, $G_\lambda$ and $F_1,\cdots,F_{\mu_\lambda}$ are local holomorphic functions.  Each $\left(d^\ell\log F\right)^\nu$ contributes $\nu\ell$ times the divisor of $F$ to the log-pole divisor (with multiplicities counted) of the log-pole jet differential.

\begin{proposition}\label{general_schwarz_lemma} {\rm (General Schwarz Lemma for Log-Pole Differential on Subvariety of Jets and Map with Slow Growth for Pole Set)}\ \   Let $X$ be a compact complex algebraic manifold of complex dimension ${\, n\, }$ and $Y$ be a complex subvariety of the space $J_k(X)$ of $k$-jets of $X$.  Let $\pi_k:J_k(X)\to X$ be the natural projection map.  Let $D$ and $E$ be nonnegative divisors of $X$ whose associated line bundles $L_D$ and $L_E$ respectively have smooth metrics $e^{-\chi_D}$ and $e^{-\chi_D}$ with smooth $(1,1)$-forms $\eta_D$ and $\eta_E$ as curvature forms such that $D+E$ is an ample divisor of $X$ and its curvature form $\eta_D+\eta_E$ for the metric $e^{-\chi_D-\chi_E}$ is strictly positive on $X$.  Let $s_D$ (respectively $s_E$) be the holomorphic section of $L_D$ (respectively $L_E$) whose divisor is $D$ (respectively $E$).  Let $F$ be a nonnegative divisor of $X$ and ${\rm Supp}\,F$ be its support.  Let $\omega$ be a function on $Y$ such that $s_E\left(s_D\right)^{-1}\omega$ is holomorphic on $Y-\pi_k^{-1}\left({\rm Supp}\,F\right)$.  Assume that for some finite open cover $\{U_j\}_{j=1}^J$ of $X$, there exists a log-pole $k$-form $\omega_j$ on $U_j$ (for $1\leq j\leq J$), whose log pole is contained in $F$ with multiplicities counted, such that on $Y\cap\pi_k^{-1}\left(U_j\right)$ the function $\omega$ agrees with the function on $\pi_k^{-1}\left(U_j\right)$ defined by $s_D\left(s_E\right)^{-1}\omega_j$ for $1\leq j\le J$.  Let $r_1>r_0>0$. Let $\varphi:{\mathbb C}-\overline{\Delta_{r_0}}\to X$ be a holomorphic map such that the image of the map $J_k(\varphi):J_k\left({\mathbb C}-\overline{\Delta_{r_0}}\right)\to J_k(X)$ induced by $\varphi$ is contained in $Y$.
Let $G_j(\zeta)$ be the function ${\rm eval}_{{\rm id}_{\mathbb C}}(\varphi_j^*\omega)$ associated to $\varphi^*\omega_j$ as explained in (\ref{pullback_function}).  If $G_j(\zeta)$ is not identically zero on ${\mathbb C}-\overline{\Delta_{r_0}}$, then
$$
\begin{aligned}
T\left(r,r_1,\varphi,\eta_D\right)\leq &\ T\left(r,r_1,\varphi,\eta_E\right)+N\left(r,r_1,\varphi,F\right)\cr
&+O\left(\log T\left(r,r_1,\varphi,\eta_D+\eta_E\right)+\log r\right)\ \ \|.\cr
\end{aligned}
$$
In particular, if for some $\varepsilon>0$ one assumes that $$N\left(r,r_1,\varphi,F\right)+T\left(r,r_1,\varphi,\left|\eta_E\right|\right)\leq\left(1-\varepsilon\right)\left(T\left(r,r_1,\varphi,\eta_D\right)\right)\ \ \|,$$
then either $G_j(\zeta)$ is identically zero for all $1\leq j\leq J$ or
$$T\left(r,\varphi,\eta_D+\eta_E\right)=O\left(\log r\right)\ \ \|.$$
\end{proposition}

\bigbreak\noindent\begin{proof}
We assume that $G_{j_0}(\zeta)$ is not identically zero for some $1\leq j_0\leq J$.    We apply twice integration of Laplacian with in (\ref{twice_integration}) to
$$g(\zeta)=\log\left(\left|G_{j_0}(\zeta)\right|^2\varphi^*\left(\frac{\,\left|s_E\right|^2 e^{-\chi_E}\,}{\left|s_D\right|^2 e^{-\chi_D}}\right)\right).
$$
Since $\omega_j$ is holomorphic on $Y-\pi_k^{-1}\left({\rm Supp}\,F\right)$, it follows that
$$
\displaylines{T\left(r,r_1,\varphi,\eta_D\right)
-T\left(r,r_1,\varphi,\eta_E\right)-N\left(r,r_1,\varphi,F\right)\cr
\leq\frac{1}{4\pi}
\int_{\theta=0}^{2\pi}\log\left(\left|G_{j_0}(\zeta)\right|^2\varphi^*\left(\frac{\,\left|s_E\right|^2 e^{-\chi_E}\,}{\left|s_D\right|^2 e^{-\chi_D}}\right)\right)\left(re^{i\theta}\right)d\theta+O(1).}
$$
Here we have the inequality instead of an identity, because of possible contribution from the zero-set of $\frac{\omega_{j_0}}{s_D}$. At this point enters Bloch's technique of applying the logarithmic derivative lemma by using the logarithm of global meromorphic functions as local coordinates.  As functions on the space $J^{\rm{\scriptstyle vert}}_{n-1}\left({\mathcal X}\right)$ of vertical $(n-1)$-jets on ${\mathcal X}$ (with the right-hand side being global functions and the left-hand side being only local functions due to the transition functions of the line bundles $L_D$ and $L_E$),
$$
\left|{\omega_{j_0} s_E}{s_D}\right|\leq C\sum_{\lambda=1}^\Lambda
\prod_{\nu_{j,\ell},\,j,\ell}\left|d^\ell\log F^{(\lambda)}_{j,\ell}\right|^{\nu_{j,\ell}}
$$
for some $C>0$ and a finite collection $\left\{F^{(\lambda)}_{j,\ell}\right\}$ of global meromorphic functions on $X$, where the product is taken over the indices $\nu_{j,\ell},\,j,\ell$ with the ranges $1\leq j\leq n,\,1\leq\ell\leq k$ and $\sum_{1\leq j\leq n,\,1\leq\ell\leq k}\ell\,\nu_{j,\ell}=m$.  By Nevanlinna's logarithmic derivative lemma (extended to ${\mathbb C}$ outside a disk centered at the origin) given in (\ref{proposition(5.I.2)}),
$$
\int_{\theta=0}^{2\pi}\log^+\left|d^\ell\log F^{(\lambda)}_{j,\ell}\right|\left(re^{i\theta}\right)d\theta=O\left(\log T\left(r,r_1,\varphi,\eta_D+\eta_E\right)+\log r\right)\ \ \|.
$$
Hence
$$
\displaylines{\int_{\theta=0}^{2\pi}\log\left(\left|G_{j_0}(\zeta)\right|^2\varphi^*\left(\frac{\,\left|s_E\right|^2 e^{-\chi_E}\,}{\left|s_D\right|^2 e^{-\chi_D}}\right)\right)\left(re^{i\theta}\right)d\theta\cr
=O\left(\log T\left(r,r_1,\varphi,\eta_D+\eta_E\right)+\log r\right)\ \ \|}
$$
and we get
$$
\begin{aligned}
T\left(r,r_1,\varphi,\eta_D\right)\leq &\ T\left(r,r_1,\varphi,\eta_E\right)+N\left(r,r_1,\varphi,F\right)\cr
&+O\left(\log T\left(r,r_1,\varphi,\eta_D+\eta_E\right)+\log r\right)\ \ \|.\cr
\end{aligned}
$$
If now for some $\varepsilon>0$ one assumes that $$N\left(r,r_1,\varphi,F\right)+T\left(r,r_1,\varphi,\left|\eta_E\right|\right)\leq\left(1-\varepsilon\right)\left(T\left(r,r_1,\varphi,\eta_D\right)\right)\ \ \|,$$
then one obtains right away $T\left(r,\varphi,\eta_D+\eta_E\right)=O\left(\log r\right)\ \ \|$.\end{proof}

\begin{theorem}\label{solution_slow_varying_coefficient} {\rm (Entire Function Solution of Polynomial Equations with Slowing Varying Coefficients)}\ \ There exists a positive integer $\delta_n$ and for $\delta\geq\delta_n$ there exists a property subvariety ${\mathcal Z}$ of ${\mathbb P}_N$ (where $N={\delta+n\choose n}$) with the following property.  There cannot exist entire functions $\varphi_0(\zeta),\cdots,\varphi_n(\zeta)$ without common zeroes and entire functions $\alpha_{\nu_0,\cdots,\nu_n}(\zeta)$ for $\nu_0+\cdots+\nu_n=\delta$ without any common zeroes satisfying
$$\sum_{\nu_0+\cdots+\nu_n=\delta}\alpha_{\nu_0,\cdots,\nu_n}(\zeta)\frac{
d^j}{d\zeta^j}\left(\varphi_0(\zeta)^{\nu_0}\cdots \varphi_n(\zeta)^{\nu_n}\right)\equiv 0\quad{\rm for}\ \ 0\leq j\leq n-1$$ such that
\begin{itemize}\item[(i)] the map $\psi:\zeta\mapsto\alpha(\zeta)=\left(\alpha_{\nu_0,\cdots,\nu_n}(\zeta)\right)\in{\mathbb P}_N$ is nonconstant,
\item[(ii)] $\alpha\left(\zeta_0\right)=\left(\alpha_{\nu_0,\cdots,\nu_n}\left(\zeta_0\right)\right)$
is not in the exceptional set ${\mathcal Z}$ in ${\mathbb P}_N$ for some $\zeta_0\in{\mathbb C}$, and
\item[(iii)]
$T\left(r,\psi\right)=o\left(T\left(r,\varphi\right)+\log r\right)\ \ \|$, where $\varphi:{\mathbb C}\to{\mathbb P}_n$ is defined by $\left[\varphi_0,\cdots,\varphi_n\right]$.
\end{itemize}
\end{theorem}

\begin{proof} We apply Proposition \ref{general_schwarz_lemma} with $k=n-1$ to the the space $J_{n-1}({\mathcal X})$ of $(n-1)$-jets of the universal hypersurface ${\mathcal X}$ with  subvariety $Y$ equal to the space $J^{\rm{\scriptstyle vert}}_{n-1}\left({\mathcal X}\right)$ of vertical $(n-1)$-jets of ${\mathcal X}$.  The assumption
$$\sum_{\nu_0+\cdots+\nu_n=\delta}\alpha_{\nu_0,\cdots,\nu_n}(\zeta)\frac{
d^j}{d\zeta^j}\left(\varphi_0(\zeta)^{\nu_0}\cdots \varphi_n(\zeta)^{\nu_n}\right)\equiv 0\quad{\rm for}\ \ 0\leq j\leq n-1$$
implies that for every $\zeta\in{\mathbb C}$ the element of $J_{n-1}({\mathcal X})$ represented by the parametrized complex curve germ $\varphi$ at $\zeta$ belongs to $Y=J^{\rm{\scriptstyle vert}}_{n-1}\left({\mathcal X}\right)$.

\medbreak By Proposition \ref{basepoint_freeness_jet_differential} we have a proper subvariety ${\mathcal Z}$ of ${\mathbb P}_N$ and for $\alpha\in{\mathbb P}_N-{\mathcal Z}$ and $1\leq j\leq J$ a holomorphic family of $(n-1)$-jet differentials $\omega^{(\alpha)}_j$ on $X^{(\alpha)}$ vanishing on the infinity hyperplane of ${\mathbb P}_n$ (extendible to a meromorphic family over all of ${\mathbb P}_N$) such that, at any point $P_0$ of $J_{n-1}\left(X^{(\alpha)}\right)$ with $\alpha\in{\mathbb P}_N-{\mathcal Z}$ which is representable by a nonsingular complex curve germ, at least one $\omega^{(\alpha)}_j$ is nonzero at $P_0$ for some $1\leq j\leq J$.

\medbreak Since for each $1\leq j\leq J$ the holomorphic family $\omega_j^{(\alpha)}$ for $\alpha\in{\mathbb P}_N-{\mathcal Z}$ can be extended to a meromorphic family for $\alpha$ varying in all of ${\mathbb P}_N$, we can find a divisor $E_j$ in ${\mathbb P}_N$ such that for all $1\leq j\leq J$ the pole-set of $\omega_j^{(\alpha)}$ as a meromorphic vertical $(n-1)$-jet differential on ${\mathcal X}$ is contained in the intersection of ${\mathcal X}$ and ${\mathbb P}_n\times E_j$ with multiplicities counted.  For $1\leq j\leq J$, because $\omega_j^{(\alpha)}$ vanishes on an ample divisor of $X^{(\alpha)}$ for $\alpha\in{\mathbb P}_N-{\mathcal Z}$, we can find a divisor $D_j$ in ${\mathcal X}$ such that $D_j+E_j$ is an ample divisor of ${\mathcal X}$ and the zero-set of $\omega_j^{(\alpha)}$ as a meromorphic vertical $(n-1)$-jet differential on ${\mathcal X}$ contains $D_j$ with multiplicities counted.

\medbreak Since $\psi:{\mathbb C}\to{\mathbb P}_N$ is nonconstant and
$T\left(r,\psi\right)=o\left(T\left(r,\varphi\right)+\log r\right)\ \ \|$, it follows that the differential $d\varphi$ is nonzero for some $\zeta_0\in{\mathbb C}$.  Denote by $P_0$ the point in $J_{n-1}\left(X^{(\alpha)}\right)$ represented by the nonsingular complex curve germ $\varphi$ at $\zeta_0$.  Some $\omega_{j_0}^{(\alpha)}$ has nonzero value at $P_0$.  From Proposition \ref{general_schwarz_lemma} applied to $\omega_{j_0}^{(\alpha)}$, it follows that $$T(r,\varphi)=O(\log r)\ \ \|,$$
which would contradict   $$\limsup_{r\to\infty}\frac{T(r,\psi)}{\log r}>0$$ from the nonconstancy of the map $\psi$ and the assumption $$T\left(r,\psi\right)=o\left(T\left(r,\varphi\right)+\log r\right)\ \ \|.$$
\end{proof}

\medbreak The analogue of the Big Picard Theorem about removable essential singularities is the following result.

\begin{theorem} {\rm (Removing Essential Singularity for Holomorphic Solution of Polynomial Equations with Slowing Varying Coefficients)}\ \ There exists a positive integer $\delta_n$ and for $\delta\geq\delta_n$ there exists a proper subvariety ${\mathcal Z}$ of ${\mathbb P}_N$ (where $N={\delta+n\choose n}$) with the following property.  For some $r>r_0>0$ let $\varphi_0(\zeta),\cdots,\varphi_n(\zeta)$ be holomorphic functions on ${\mathbb C}-\overline{\Delta_{r_0}}$ without common zeroes and let  $\alpha_{\nu_0,\cdots,\nu_n}(\zeta)$ for $\nu_0+\cdots+\nu_n=\delta$ be holomorphic functions on ${\mathbb C}-\overline{\Delta_{r_0}}$ without common zeroes.  Assume that
$$\sum_{\nu_0+\cdots+\nu_n=\delta}\alpha_{\nu_0,\cdots,\nu_n}(\zeta)\frac{
d^j}{d\zeta^j}\left(\varphi_0(\zeta)^{\nu_0}\cdots \varphi_n(\zeta)^{\nu_n}\right)\equiv 0\quad{\rm for}\ \ 0\leq j\leq n-1$$ on ${\mathbb C}-\overline{\Delta_{r_0}}$ such that
\begin{itemize}\item[(i)] the map $\psi:\zeta\mapsto\alpha(\zeta)=\left(\alpha_{\nu_0,\cdots,\nu_n}(\zeta)\right)\in{\mathbb P}_N$ is nonconstant,
\item[(ii)] $\alpha\left(\zeta_0\right)=\left(\alpha_{\nu_0,\cdots,\nu_n}\left(\zeta_0\right)\right)$
is not in the exceptional set ${\mathcal Z}$ in ${\mathbb P}_N$ for some $\zeta_0\in{\mathbb C}-\overline{\Delta_{r_0}}$, and
\item[(iii)]
$T\left(r,r_1,\psi\right)=o\left(T\left(r,r_1,\varphi\right)+\log r\right)\ \ \|$.
\end{itemize}
Then $\varphi_0(\zeta),\cdots,\varphi_n(\zeta)$ and $\alpha_{\nu_0,\cdots,\nu_n}(\zeta)$ for
for $\nu_0+\cdots+\nu_n=\delta$ can be extended to meromorphic functions on ${\mathbb C}-\overline{\Delta_{r_0}}$.
\end{theorem}

\begin{proof} We use the same notations as in the proof of Theorem \ref{solution_slow_varying_coefficient} except the domains for the maps
$\varphi:{\mathbb C}-\overline{\Delta_{r_0}}\to{\mathcal X}$ and $psi:{\mathbb C}-\overline{\Delta_{r_0}}\to{\mathbb P}_N$ are now different.  Without
loss of generality we can assume that the map $\varphi$ is nonconstant, otherwise the extendibility of $\varphi$ and $\psi$ is clear.  Denote by $P_0$ the point in $J_{n-1}\left(X^{(\alpha)}\right)$ represented by the nonsingular complex curve germ $\varphi$ at $\zeta_0$.  Some $\omega_{j_0}^{(\alpha)}$ has nonzero value at $P_0$.  From Proposition \ref{general_schwarz_lemma} applied to $\omega_{j_0}^{(\alpha)}$, it follows that $$T(r,\varphi)=O(\log r)\ \ \|.$$
Now the extendibility of $\varphi$ and $\psi$ to respectively holomorphic maps ${\mathbb C}-\overline{\Delta_{r_0}}\to{\mathcal X}$ and ${\mathbb C}-\overline{\Delta_{r_0}}\to{\mathbb P}_N$ follows from Lemma \ref{extend_across_infinity_point}.
\end{proof}

\bigbreak\noindent{\bf III.} {\sc Second Main Theorem from Log-Pole Jet Differential}

\medbreak Nevanlinna's Second Main Theorem is a quantitative version of the Little Picard Theorem.  The hyperbolicity of generic hypersurface of high degree corresponds to the Little Picard Theorem.  We now discuss the analogue of Nevanlinna's Second Main Theorem for any regular hypersurface of high degree from our approach of jet differentials.

\medbreak In contrast to the use of holomorphic jet differentials vanishing on an ample divisor in the hyperbolicity problem, the jet differentials used for the Second Main Theorem are log-pole jet differentials vanishing on ample divisor.  Our method fits in with Cartan's proof of the Second Main Theorem for entire holomorphic curves in ${\mathbb P}_n$ and a collection of hyperplanes in ${\mathbb P}_n$ in general position given in \cite{Ca33}.

\medbreak We will first show how to construct log-pole jet differentials on ${\mathbb P}_n$ which vanishes on an appropriate ample divisor of ${\mathbb P}_n$ and whose log pole-set is contained in the hypersurface.  We then present two Second Main Theorems for log-pole jet differentials, with the second one dealing with the situation of slowly moving targets.  Then we show how Cartan's proof can be recast in our setting of Second Main Theorem for log-pole jet differentials vanishing on an appropriate ample divisor.

\medbreak Second main theorems are useful only when the estimates are reasonably sharp. In the case at hand, because of our construction of jet differentials is so far away from the conjectured optimal situation, the discussion about Second Main Theorems can only serve as pointing out a connection between Second Main Theorems and jet differentials and their construction.

\begin{theorem}\label{construction_log_pole_differential} {\rm (Existence of Log Pole Jet Differential)}\ \ Let $0<\varepsilon_0,\varepsilon_0^\prime<1$.  There exists a positive integer $\hat\delta_n$ such that for any regular hypersurface $X$ of degree $\delta\geq\hat\delta_n$ in ${\mathbb P}_n$ there exists a non identically zero log-pole $n$-jet differential $\omega$ on ${\mathbb P}_n$ of weight $\leq \delta^{\varepsilon_0}$ which vanishes with multiplicity at least $\delta^{1-\varepsilon_0^\prime}$ on the infinity hyperplane of ${\mathbb P}_n$ and which is holomorphic on ${\mathbb P}_n-X$. In particular, the log-pole divisor of $\omega$ is no more than $\lambda$ times $X$ with $\lambda\leq n\delta^{\varepsilon_0}$.\end{theorem}

\begin{proof}  We choose $\epsilon$, $\epsilon^\prime$, $\theta_0$, $\theta$, $\theta^\prime$ in the open interval $(0,1)$ such that
$(n+1)\theta_0+\theta\geq (n+1)+\epsilon$,
$1-\varepsilon_0^\prime\leq\theta^\prime<1-\epsilon^\prime$, and $\varepsilon_0\leq\theta$.  We apply Proposition \ref{construction_jet_differential_as_polynomial} to get $A=A(n+1,\epsilon,\epsilon^\prime)$ from it and then set $\hat\delta_{n+1}=A$.

\medbreak Let $f\left(x_1,\cdots,x_n\right)$ be a polynomial in terms of the inhomogeneous coordinates $x_1,\cdots,x_n$ of ${\mathbb P}_n$ which defines $X$.  Let $\hat X$ be the regular hypersurface in ${\mathbb P}_{n+1}$ defined by the polynomial $F=f\left(x_1,\cdots,x_n\right)-x_{n+1}^\delta$ in the inhomogeneous coordinates $x_1,\cdots,x_{n+1}$ of ${\mathbb P}_{n+1}$.  We apply Proposition \ref{construction_jet_differential_as_polynomial} to $F$ (instead of to $f$) to get
an
$n$-jet differential $\hat\omega$ of the form
${Q\over F_{x_1}-1}$ which vanishes to order $\geq \delta^{\theta^\prime}$ at the infinity hyperplane of ${\mathbb P}_{n+1}$, where $Q$ is a polynomial in
$$d^jx_1,\cdots,d^jx_{n+1}\quad(0\leq j\leq n)$$ which is of degree
$m_0=\left\lceil\delta^{\theta_0}\right\rceil$ in $x_1,\cdots,x_{n+1}$
and is of homogeneous weight
$m=\left\lceil\delta^\theta\right\rceil$ in
$$d^jx_1,\cdots,d^jx_{n+1}\quad(1\leq j\leq n)$$ when the weight of
$d^j x_\ell$ is assigned to be $j$.

\medbreak We choose a nonzero integer $\ell$ such that one nonzero term of $\frac{Q}{x_{n+1}^\ell}$ is of the form
$$
Q_0\,\left(\frac{dx_{n+1}}{x_{n+1}}\right)^{\nu_1}\cdots\left(\frac{d^nx_{n+1}}{x_{n+1}}\right)^{\nu_n},
$$
where $Q_0$ is a polynomial in the variables
$$d^jx_1,\cdots,d^jx_n\quad(0\leq j\leq n)$$
with constant coefficients and $\nu_1,\cdots,\nu_n$ are nonnegative integers.

\medbreak The complex manifold $\hat X$ is a branched cover over ${\mathbb P}_n$ with cyclic branching  of order $\delta$ at $X$ under the projection map $\hat\pi:\hat X\to{\mathbb P}_n$ induced by $\left(x_1,\cdots,x_n,x_{n+1}\right)\mapsto \left(x_1,\cdots,x_n\right)$.  Let $\omega$ be the direct image of $$\frac{\hat\omega}{x_{n+1}^\ell}=\frac{Q}{x_{n+1}^\ell\left(f_{x_1}-1\right)}$$ under $\hat\pi$.  The $n$-jet differential $\omega$ on ${\mathbb P}_n$ can be computed as follows. First we express $d^\ell x_{n+1}$ (by induction on $\ell$) as a polynomial of the variables $$x_{n+1}, d\log x_{n+1}, d^2\log x_{n+1},\cdots, d^\ell\log x_{n+1}$$ with constant coefficients so that $\frac{Q}{x_{n+1}^\ell}$ is expressed as a polynomial of $x_1,\cdots,x_{n+1}$ and
$$d^jx_1,\cdots,d^jx_n,d^j\log x_{n+1}\quad(0\leq j\leq n)$$
with constant coefficients.  Then we obtain $\omega$ from $\frac{Q}{x_{n+1}^\ell\left(f_{x_1}-1\right)}$ by replacing $d^j\log x_{n+1}$ by $d^j\log f$ and setting $x_{n+1}$ equal to $0$.  The log-pole jet differential $\omega$ on ${\mathbb P}_n$ is not identically zero, because of the nonzero term
$$
Q_0\,\left(\frac{dx_{n+1}}{x_{n+1}}\right)^{\nu_1}\cdots\left(\frac{d^nx_{n+1}}{x_{n+1}}\right)^{\nu_n},
$$
in $Q$.  The log-pole divisor of $\omega$ is no more than $\lambda$ times $X$ with $\lambda\leq n\delta^{\varepsilon_0}$, because $\omega$ is an $n$-jet differential of weight $\leq \delta^{\varepsilon_0}$.
\end{proof}

\bigbreak By applying Proposition \ref{general_schwarz_lemma}, we have the following two Second Main Theorems for log-pole jet differentials, with second one dealing with the case of slowly moving targets.

\begin{theorem}\label{second_main_theorem_jet_differential} {\rm (Second Main Theorem from Log-Pole Jet Differentials)}\ \  Let $X$ be an $n$-dimensional compact complex manifold with an ample line bundle $L$.  Let $D_1,\cdots,D_p,E_1,\cdots,E_q$ be divisors of $L$.  Let $\omega$ be a log-pole jet differential on $X$ vanishing on $D=D_1+\cdots+D_p$ such that the log-pole set of $\omega$ is contained in $E=E_1+\cdots+E_q$ with multiplicities counted.  Then for any holomorphic map $\varphi$ from the affine complex line ${\mathbb C}$ to $X$ such that the image of $\varphi$ is not contained in $E$ and the pullback $\varphi^*\omega$ is not identically zero,
$$
pT(r,\varphi,L)\leq N(r,\varphi,E)+O\left(\log T(r,\varphi,L)\right)\quad\|
$$
holds.
In other words,
$$
\sum_{j=1}^q m\left(r,\varphi,E_j\right)\leq(q-p)pT(r,\varphi,L)+O\left(\log T(r,\varphi,L)\right)\quad\|.
$$
The meaning of the log-pole set of $\omega$ being contained in $E=E_1+\cdots+E_q$ with multiplicities counted is the following.  Locally $\omega$ is of the form $$\sum_{\tau_1,\lambda_1,\cdots,\tau_k,\lambda_k}h_{\tau_1,\lambda_1,\cdots,\tau_k,\lambda_k}
\left(d^{\tau_1}\log F_1\right)^{\lambda_1}\cdots\left(d^{\tau_\ell}\log F_\ell\right)^{\lambda_\ell}$$ with
$$
\tau_1\lambda_1\,{\rm div}\,F_1+\cdots+\tau_\ell\lambda_\ell\,{\rm div}\,F_\ell
$$
contained in $E$ with multiplicities counted, where ${\rm div}\,F_j$ is the divisor of $F_j$.
\end{theorem}

\begin{theorem}\label{second_main_theorem_jet_differential_moving_target} {\rm (Second Main Theorem for Jet Differential with Slowly Moving Targets)}\ \  Let $S\subset{\mathbb P}_N$ be a complex algebraic manifold and ${\mathcal X}\subset{\mathbb P}_{\hat n}\times S$ be a complex algebraic manifold.  Let $\pi:{\mathcal X}\to S$ be the projection induced by the natural projection ${\mathbb P}_{\hat n}\times{\mathbb P}_N\to{\mathbb P}_N$ to the second factor.  Let $L_S$ be an ample line bundle on $S$. Let $L$ be a line bundle on ${\mathcal X}$ such that $L+\pi^{-1}\left(L_S\right)$ is ample on $X$. Let $D_1,\cdots,D_p,E_1,\cdots,E_q$ be divisors of $L$. Let $D=D_1+\cdots+D_p$ and $E=E_1+\cdots+E_q$.  For $\alpha\in S$ let $X^{(\alpha)}=\pi^{-1}(\alpha)$ and $D^{(\alpha)}=D|_{X^{(\alpha)}}$ and $E^{(\alpha)}=E|_{X^{(\alpha)}}$.  Let $Z$ be a proper subvariety of $S$.  For $\alpha\in S-Z$ let $\omega^{(\alpha)}$ be a log-pole jet differential on $X^{(\alpha)}$ such that $\omega^{(\alpha)}$ vanishes on the divisor $D^{(\alpha)}$ and the log-pole set of $\omega^{(\alpha)}$ is contained in the divisor $E^{(\alpha)}$ with multiplicities counted.  Assume that $\omega^{(\alpha)}$ is holomorphic in $\alpha$ for $\alpha\in S-Z$ and is meromorphic in $\alpha$ for $\alpha\in S$.  Let $\varphi$ be a holomorphic map from the affine complex line ${\mathbb C}$ to ${\mathcal X}$ such that the image of $\pi\circ\varphi$ is not contained in $Z$ and  $T\left(r,\pi\circ\varphi, L_S\right)=o\left(T\left(r,\varphi,L+\pi^{-1}\left(L_S\right)\right)\right)$, then
$$
qT\left(r,\varphi,L+\pi^{-1}\left(L_S\right)\right)\leq N(r,\varphi,D)+o\left(T\left(r,\varphi,L+\pi^{-1}\left(L_S\right)\right)\right)\quad\|.
$$
In other words,
$$
\sum_{j=1}^p m\left(r,\varphi,D_j\right)\leq(q-p)\,T\left(r,\varphi,L+\pi^{-1}\left(L_S\right)\right)
+o\left(T\left(r,\varphi,L+\pi^{-1}\left(L_S\right)\right)\right)\quad\|.
$$
In the product case of ${\mathcal X}=X^{(0)}\times S$ with $D_j=D^{(0)}_j$, the proximity function $m\left(r,\varphi,D_j\right)$ in the formulation can be replaced by the proximity function $m\left(r,{\rm pr}_1\circ\varphi,D^{(0)}_j\right)$, because $N(r,\varphi,D)$ is equal to $N\left(r,{\rm pr}_1\circ\varphi,D_j^{(0)}\right)$ and we can apply Nevanlinna's First Main Theorem to
${\rm pr}_1\circ\varphi$ and the divisor $D_j^{(0)}$ and use the assumption that $T\left(r, \pi\circ\varphi, L_S\right)=o\left(T\left(r,\varphi,L+\pi^{-1}\left(L_S\right)\right)\right)$.  Here ${\rm pr}_1$ means the natural projection ${\rm pr}_1:{\mathbb P}_{\hat n}\times{\mathbb P}_N\to{\mathbb P}_{\hat n}$ to the first factor.
\end{theorem}

\bigbreak In the following remark we discuss how Cartan's proof of his Second Main Theorem for hyperplanes in general position can be interpreted in the setting of the Second Main Theorem for log-pole jet differentials.

\begin{remark}  Cartan's Second Main Theorem for hyperplanes in ${\mathbb P}_n$ for hyperplanes in general position given in \cite{Ca33} is simply the special case of Theorem \ref{second_main_theorem_jet_differential} with
$$
\omega=\frac{{\rm Wron}(dx_1,\cdots,dx_n)}{F_1\cdots F_q}
$$
in inhomogeneous coordinates $x_1,\cdots,x_n$ of ${\mathbb P}_n$, where $F_1,\cdots,F_q$ are the degree-one polynomial in $x_1,\cdots,x_n$ which define the $q$ hyperplanes in ${\mathbb P}_n$ in general position.

\medbreak Here the notation for the Wronskian
$${\rm Wron}\left(\eta_1,\cdots,\eta_\ell\right)$$ for jet differentials $\eta_1,\cdots,\eta_\ell$ on a complex manifold $Y$ is used to mean
the jet differential
$$
\det\left(d^{\lambda-1}\eta_j\right)_{1\leq\lambda,j\leq\ell}=\sum_{\sigma\in S_\ell}\left({\rm sgn}\,\sigma\right) \eta_{\sigma(1)} \left(d\eta_{\sigma(2)}\right)\cdots\left(d^{\ell-1}\eta_{\sigma(\ell)}\right)
$$
on $Y$, where $S_\ell$ is the group of all permutations of $\left\{1,2,\cdots,\ell\right\}$ and ${\rm sgn}\,\sigma$ is the signature of the permutation $\sigma$.

\medbreak The denominator $F_1\cdots F_q$ in $\omega$ gives the vanishing order $q$ at the infinity hyperplane of ${\mathbb P}_n$.  The key argument here is that from the general position assumption of the zero-sets of $F_1,\cdots,F_q$ we can locally write $\omega$ as a constant times
$$
\frac{{\rm Wron}\left(dF_{\nu_1},\cdots,dF_{\nu_n}\right)}{F_1\cdots F_q}=
\frac{{\rm Wron}\left(d\log F_{\nu_1},\cdots,d\log F_{\nu_n}\right)}{F_{\nu_{n+1}}\cdots F_{\nu_q}}
$$
in a neighborhood $U$ of a point when $F_j$ is nowhere zero on $U$ for $j$ not equal to any of the indices $\nu_1,\cdots,\nu_n$.
\end{remark}

\bigbreak\noindent{\it Author's address:} Department of Mathematics, Harvard University, Cambridge, MA 02138, U.S.A.

\bigbreak\noindent{\it Author's e-mail address:} siu@math.harvard.edu

\begin{thebibliography}{[KLR73]}

\bibitem[AN64]{AN64} Aldo Andreotti and Raghavan Narasimhan,
Oka's Heftungslemma and the Levi problem for complex spaces.
{\it Trans. Amer. Math. Soc.} \textbf{111}, 345–-366 (1964).

\bibitem[B26]{B26} Andr\'e Bloch, Sur les
syst\`emes de fonctions uniformes satisfaisant \`a l'\'equa\-tion
d'une vari\'et\'e alg\'ebrique dont l'irr\'egularit\'e d\'epasse la
dimension. {\it J. de Mathematiques Pures et Appliqu\'ees} \textbf{5}, 19--66 (1926).

\bibitem[Bo70]{Bo70} Enrico Bombieri, Algebraic values of meromorphic maps, {\it Invent.
Math.} \textbf{10}, 267-287 (1970).  Addendum. {\it Invent.
Math.} \textbf{11}, 163--166 (1970).

\bibitem[BL70]{BL70} Enrico Bombieri and Serge Lang,
Analytic subgroups of group varieties.
{\it Invent. Math.} \textbf{11}, 1--14 (1970).

\bibitem[Bo97]{Bo97} \'Emile Borel, Sur les z\'eros des fonctions enti\`eres, {\it Acta Math.} \textbf{20}, 357--396  (1897).

\bibitem[Ca33]{Ca33} Henri Cartan, Sur les z\'eros des combinaisons
lin\'eaires de  $p$  fonctions holomorphes donn\'ees, {\it Mathematica}
(Cluf) \textbf{7}, 5--31  (1933).

\bibitem[Cl86]{Cl86} Herb Clemens, Curves on generic hypersurfaces, {\it Ann. \'Ec. Norm.
Sup.} \textbf{19}, 629--636 (1986).

\bibitem[DMR10]{DMR10} Simone Diverio, Jo\"el Merker, and Erwan Rousseau,
Effective algebraic degeneracy. {\it Invent. Math.} \textbf{180}, 161–-223 (2010).

\bibitem[Ei88]{Ei88} Lawrence Ein, Subvarieties of generic complete intersections,
{\it Invent. Math.} {\textbf 94}, 163--169  (1988). II.  {\it Math. Ann.}
\textbf{289}, 465--471 (1991).

\bibitem[Ge34]{Ge34} Alexander Osipovich Gelfond, Sur le septi\`eme Probl\`eme de D. Hilbert. {\it Comptes Rendus Acad. Sci. URSS Moscou} \textbf{2} (1934), 1--6. {\it Bull. Acad. Sci. URSS Leningrade} \textbf{7}, 623--634 (1934).

\bibitem[La62]{La62} Serge Lang,
Transcendental points on group varieties.
{\it Topology} \textbf{1}, 313--318 (1962).

\bibitem[La65]{La65} Serge Lang,
Algebraic values of meromorphic functions.
{\it Topology} \textbf{3}, 183--191 (1965).

\bibitem[La66]{La66} Serge Lang,
{\it Introduction to transcendental numbers}. Addison-Wesley Publishing Co., Reading, Mass.-London-Don Mills, Ont. 1966.

\bibitem[Me09]{Me09}
Jo\"el Merker,
Low pole order frames on vertical jets of the universal hypersurface.
{\it Ann. Inst. Fourier} (Grenoble) \textbf{59}, 1077–-1104 (2009).

\bibitem[Ne25]{Ne25}
Rolf Nevanlinna,
Zur Theorie der Meromorphen Funktionen.
{\it Acta Math.} \textbf{46}, 1--99 (1925).

\bibitem[Sc34]{Sc34} Theodor Schneider, Transzendenzuntersuchungen periodischer Funktionen. I, II. {\it J. reine angew. Math.} \textbf{172}, 65--74 (1934).

\bibitem[Si95]{Si95} Yum-Tong Siu, Hyperbolicity problems in function
theory. In: {\it Five Decades as a Mathematician and Educator - on
the 80th birthday of Professor Yung-Chow Wong}, ed. Kai-Yuen Chan
and Ming-Chit Liu, World Scientific: Singapore, New Jersey,
London, Hong Kong, 1995, pp.409--514.

\bibitem[Si98]{Si98} Yum-Tong Siu,
Invariance of plurigenera.
{\it Invent. Math.} \textbf{134}, 661--673 (1998).

\bibitem[Si00]{Si00} Yum-Tong Siu,
Extension of twisted pluricanonical sections with plurisubharmonic weight and invariance of semipositively twisted plurigenera for manifolds not necessarily of general type. In: {\it Complex geometry} (G\"ottingen, 2000), 223--277, Springer, Berlin, 2002.

\bibitem[Si02]{Si02}
Yum-Tong Siu,
Some recent transcendental techniques in algebraic and complex geometry.  In: {\it Proceedings of the International Congress of Mathematicians}, Vol. I (Beijing, 2002), 439–-448, Higher Ed. Press, Beijing, 2002.

\bibitem[Si04]{Si04} Yum-Tong Siu,
Hyperbolicity in complex geometry. In: {\it The legacy of Niels Henrik Abel}, 543–-566, Springer, Berlin, 2004.

\bibitem[SY97]{SY97} Yum-Tong Siu and See-Kai Yeung, Defects for ample divisors of abelian
varieties, Schwarz lemma, and hyperbolic hypersurfaces of low
degrees. {\it Amer. J. Math.} \textbf{119}, 1139--1172 (1997). Addendum.
{\it Amer. J. Math.} \textbf{125}, 441--448 (2003).

\bibitem[Voi96]{Voi96} Claire Voisin, On a conjecture of Clemens on rational curves on
hypersurfaces, {\it J. Diff. Geom.} {\textbf 44}, 200--213 (1996). A
correction: ``On a conjecture of Clemens on rational curves on
hypersurfaces'' {\it J.~ Diff. Geom.} \textbf{49}, 601–-611 (1998).

\bibitem[Voj87]{Voj87} Paul Vojta,
Diophantine approximations and value distribution theory.
{\it Lecture Notes in Mathematics} \textbf{1239}. Springer-Verlag, Berlin, 1987.

\end{thebibliography}
\end{document}